\documentclass[a4paper,11pt]{article}

\usepackage[latin1]{inputenc}
\usepackage{amsfonts,amsmath,amssymb,amsthm,graphicx,epsfig,float}
\usepackage[T1]{fontenc}
\usepackage[english,francais]{babel}
\usepackage[all]{xy}

\usepackage{amsthm}
\newtheorem{Thm}{Théorème}
\newtheorem{Pro}{Proposition}[section]
\newtheorem{Lem}{Lemme}[section]

\theoremstyle{definition}

\theoremstyle{remark}
\newtheorem*{Rem}{Remarque}

\newcommand{\Cc}{\mathbb{C}}

\newcommand{\Zz}{\mathbb{Z}}
\newcommand{\Pp}{\mathbb{P}}

\title{{\bf Etude des mesures hyperboliques pour les applications méromorphes}}
\author{Henry De Thélin et Franck Nguyen Van Sang}
\date{}

\begin{document}
\maketitle

\def\figurename{{Fig.}}%
\def\proofname{Preuve}
\def\contentsname{Sommaire}%

\begin{abstract}

Nous montrons un résultat du type Closing Lemma pour les mesures non uniformément hyperboliques associées à des applications méromorphes. Nous prouvons aussi qu'il est possible d'approximer la dynamique de telles mesures par des codages du type Bernoulli.

\end{abstract}

\selectlanguage{english}
\begin{center}
{\bf{ }}
\end{center}

\begin{abstract}

We prove a Closing Lemma for nonuniformly hyperbolic measures of meromorphic maps. We prove also a theorem of approximation of the dynamics of such measures by Bernoulli coding maps.

\end{abstract}

\selectlanguage{francais}

Mots-clefs: dynamique complexe, Closing Lemma, entropie, approximation par des horsehoes. 

Classification: 32H50, 37Dxx.

\section*{{\bf Introduction}}
\par

Soit $X$ une variété Riemannienne lisse compacte de dimension $k$ et $f: X \longrightarrow X$ un difféomorphisme de classe $C^{1 + \alpha}$. 

On considère $\mu$ une mesure de probabilité invariante par $f$, ergodique et hyperbolique (c'est-à-dire que ses exposants de Lyapounov sont non nuls).

Un résultat majeur, obtenu par A. Katok (voir \cite{Ka}), est la construction dans ce contexte d'orbites périodiques. C'est le Closing Lemma: sous certaines hypothèses, lorsqu'un point $x$ et son itérée $f^m(x)$ (avec $m \geq 1$) sont proches, on peut trouver un point périodique près de $x$. Ce résultat est crucial pour approximer le système dynamique par des horseshoes (voir A. Katok et L. Mendoza dans \cite{KH}).

L'objet de cet article est d'étendre ces théorèmes en dynamique méromorphe.

Soit donc $X$ une variété complexe compacte de dimension $k$ et $f: X \longrightarrow X$ une application méromorphe dominante. Nous noterons $I$ son ensemble d'indétermination. 

On considère $\mu$ une mesure de probabilité qui vérifie $\int \log d(x,I) d \mu (x)>-\infty$. On suppose $\mu$ invariante par $f$, ergodique et hyperbolique, c'est-à-dire que ses exposants de Lyapounov vérifient : 

$$\chi_1\geq  \cdots \geq \chi_{m_0}>0>\chi_{m_0+1}\geq \cdots \geq \chi_{k} \mbox{  pour un  } 1\leq m_0\leq k-1.$$ 

Sous ces hypothèses, en suivant la stratégie d'A. Katok (voir \cite{Ka}), nous démontrons le Closing Lemma (voir le paragraphe \ref{closinglemma} et le théorème \ref{theoreme1}).

Ensuite, grâce à la démonstration de ce résultat, nous pouvons approcher la dynamique de $\mu$ par des ensembles qui sont codés par des Bernoulli. Plus précisément, on a

\begin{Thm}{\label{theoreme2}}

On suppose en plus que $h_{\mu}(f)>0$. 

Alors pour tout $\rho > 0$ et $\varphi_1,\cdots,\varphi_{k_0}\in \mathrm{C}^0(X,\mathbb{R})$ on a: 
\begin{enumerate}
\item Il existe $n\geq 1$ et un codage: 

$$
\xymatrix{
\{1,\cdots,N\}^\mathbb{Z}  \ar[r]^{\sigma}  \ar[d]_{S_0}  & \{1,\cdots,N\}^\mathbb{Z} \ar[d]^{S_0}\\
\Lambda_0 \ar[r]^{f^n} & \Lambda_0
}
$$

où $N=e^{h_{\mu}(f)n-\rho n}$, $f^n\circ S_0=S_0\circ \sigma$, $\sigma$ est le décalage à  droite et $S_0$ est continue. 

De plus $\nu_0=(S_0)_*\lambda_0$ (où $\lambda_0$ est la mesure de Bernoulli) est une mesure invariante par $f^n$ et d'entropie métrique supérieure ou égale à  $n(h_{\mu}(f)-\rho)$.
\item $h_{\mathrm{top}}(f_{| \Lambda_0})\geq h_{\mu}(f)-\rho$ (où ici $\Lambda_0=S_0(\{1,\cdots,N\}^\mathbb{Z})$).
\item Si on note $\Lambda=\Lambda_0 \cup f(\Lambda_0) \cup \cdots \cup f^{n-1}(\Lambda_0)$, on a que $\Lambda$ est contenu dans un $\rho$-voisinage de $\mbox{supp}( \mu)$ et la mesure 

$$\nu=\frac{1}{n}(\nu_0+f_* \nu_0+\cdots+(f^{n-1})_* \nu_0),$$

 portée par $\Lambda$, est $f$-invariante et vérifie: 

$$\forall i=1,\cdots,k_0  \mbox{    } \mbox{    }  \left| \int \varphi_i\mathrm{d}\nu - \int \varphi_i\mathrm{d}\mu \right| < \rho.$$

\end{enumerate}
\end{Thm}

\begin{Rem}
Dans le théorème précédent, si on sait de plus que $f$ est inversible et que $\int \log d(x,I(f^{-1}))   d \mu(x)>-\infty$ (où $I(f^{-1})$ est l'ensemble d'indétermination de $f^{-1}$), alors $S_0$ est un homéomorphisme et $\Lambda$ est un horseshoe comme dans le cas des difféomorphismes (voir \cite{KH}).
\end{Rem}

Signalons que l'hypothèse $\int \log d(x,I) d \mu (x)>-\infty$ est assez faible et qu'elle est naturelle pour faire de la théorie de Pesin. Lorsque $f$ est inversible et que $\mu$ est plus régulière, typiquement une mesure SRB (voir la condition 1.1 de \cite{KS}), il existe aussi des résultats d'approximation des systèmes dynamiques dits singuliers (voir \cite{KS}) qui se basent sur l'approche de Pesin (voir \cite{Pe}).

Ce type de théorème est vrai aussi pour les applications avec singularités lorsque tous les exposants de Lyapounov sont strictement positifs (voir \cite{Ge}). 

Par ailleurs, il existe aussi des résultats de codage par des Bernouilli pour la mesure d'entropie maximale d'un endomorphisme holomorphe de $\Pp^k(\Cc)$ (voir \cite{Br} et \cite{Du1}).

Le Closing Lemma de notre article est à comprendre en tant que fermeture d'une orbite sans changement de l'application $f$. Il y a aussi des théorèmes en dynamique holomorphe où l'orbite est fermée en faisant varier $f$ (voir \cite{FS1} et \cite{FS2}).

Voici maintenant le plan de l'article: dans un premier paragraphe nous faisons des rappels sur la théorie de Pesin et sur les théorèmes de transformées de graphes que nous utiliserons. Dans le second, nous donnons un énoncé précis du Closing Lemma et nous le démontrons. Le troisième paragraphe est consacré à la preuve du théorème \ref{theoreme2}, c'est-à-dire à l'approximation du système dynamique par des Bernoulli.

\section{{\bf Théorie de Pesin et transformées de graphes}}{\label{paragraphePesin}}
\par

Dans ce paragraphe nous rappelons la théorie de Pesin ainsi que les théorèmes de transformées de graphes.

\subsection{{\bf Théorème d'Oseledets et théorie de Pesin}}{\label{Pesin}}

Soit $X$ une variété complexe compacte de dimension $k$ et $f: X \rightarrow X$ une application méromorphe dominante. On notera $I$ son ensemble d'indétermination.

On munit $X$ d'une famille de cartes $(\tau_x)_{x \in X}$ qui vérifient $\tau_x(0)=x$, $\tau_x$ est définie sur une boule $B(0, \epsilon_0) \subset \Cc^k$ avec $\epsilon_0$ indépendant de $x$ et la norme de la dérivée première et seconde de $\tau_x$ sur $B(0, \epsilon_0)$ est majorée par une constante indépendante de $x$. Pour construire ces cartes il suffit de partir d'une famille finie $(U_i, \psi_i)$ de cartes de $X$ et de les composer par des translations.

On considère une probabilité $\mu$ telle que $\int \log d(x,I) d \mu(x) > - \infty$. On la suppose aussi invariante par $f$ et ergodique. On note $\Omega=X \setminus \cup_{i \geq 0} f^{-i}(I)$. Comme $\mu$ ne charge pas $I$ et qu'elle est invariante par $f$, $\mu$ est une probabilité de $\Omega$.

On définit l'extension naturelle:

$$\widehat{\Omega}:= \{ \widehat{x}=( \cdots, x_0, \cdots , x_n , \cdots) \in \Omega^{\Zz} \mbox{ , } f(x_{n})=x_{n+1} \}.$$

Dans cet espace, $f$ induit une application $\widehat{f}$ qui est le décalage à  gauche. Si on note $\pi$ la projection canonique $\pi(\widehat{x})=x_0$ alors $\mu$ se relève en une unique probabilité $\widehat{\mu}$ invariante par $\widehat{f}$ qui vérifie $\pi_{*} \widehat{\mu}=\mu$.

Dans toute la suite, on notera $f_x= \tau_{f(x)}^{-1} \circ f \circ \tau_x$ qui est définie au voisinage de $0$ quand $x$ n'est pas dans $I$. Le cocycle auquel nous allons appliquer la théorie de Pesin est:
\begin{equation*}
\begin{split}
A :\ & \widehat{\Omega} \longrightarrow M_k(\Cc)\\
& \widehat{x} \longrightarrow Df_x(0)\\
\end{split}
\end{equation*}

où $M_k(\Cc)$ est l'ensemble des matrices carrées $k \times k$ à  coefficients dans $\Cc$ et $\pi(\widehat{x})=x$. Grâce à l'hypothèse d'intégrabilité de la fonction $\log d(x,I)$, on a (voir le lemme 8 de \cite{Det2})

$$\int \log^+ \| A( \widehat{x}) \| d \widehat{\mu} ( \widehat{x}) < + \infty.$$

Il en découle que l'on a un théorème du type Oseledets (voir  \cite{FLQ} et \cite{Th} ainsi que le théorème 2.3 de \cite{N}, le théorème 6.1 dans \cite{Du} et \cite{M1}):

\begin{Thm}

 Il existe des réels $\lambda_1 > \lambda_2 > \cdots > \lambda_l \geq - \infty$, des entiers $m_1, \cdots, m_l$ et un ensemble $\widehat{\Gamma}$ de mesure pleine pour $\widehat{\mu}$ tels que pour $\widehat{x} \in \widehat{\Gamma}$ on ait une décomposition de $\Cc^k$ de la forme $\Cc^k= \bigoplus_{i=1}^{l} E_i(\widehat{x})$ où les $E_i(\widehat{x})$ sont des sous-espaces vectoriels de dimension $m_i$ qui vérifient:

1) $A(\widehat{x}) E_i(\widehat{x}) \subset E_i(\widehat{f}(\widehat{x}))$ avec égalité si $\lambda_i > - \infty$.

2) Pour $v \in E_i(\widehat{x}) \setminus \{0 \}$, on a 
$$\lim_{n \rightarrow + \infty} \frac{1}{n} \log \|A(\widehat{f}^{n-1}(\widehat{x})) \cdots A(\widehat{x}) \| = \lambda_i.$$

Si de plus, $\lambda_i > - \infty$, on a la même limite quand $n$ tend vers $- \infty$.

Pour tout $\gamma > 0$, il existe une fonction $C_{\gamma} :  \widehat{\Gamma} \longrightarrow GL_k(\Cc)$ telle que pour $\widehat{x} \in \widehat{\Gamma}$:

1) $\lim_{n \rightarrow \infty} \frac{1}{n} \log \| C^{\pm 1}_{\gamma} (\widehat{f}^n(\widehat{x})) \|=0$ (on parle de fonction tempérée).

2) $C_{\gamma}(\widehat{x})$ envoie la décomposition standard $\bigoplus_{i=1}^{l} \Cc^{m_i}$ sur $\bigoplus_{i=1}^{l} E_i(\widehat{x})$.

3) La matrice $A_{\gamma}(\widehat{x})= C_{\gamma}^{-1}(\widehat{f}(\widehat{x})) A(\widehat{x}) C_{\gamma}(\widehat{x})$ est diagonale par bloc $(A^1_{\gamma}(\widehat{x}), \cdots, A^l_{\gamma}(\widehat{x}))$ où chaque $A^{i}_{\gamma}(\widehat{x})$ est une matrice carrée $m_i \times m_i$ et
$$\forall v \in \Cc^{m_i}   \mbox{    } \mbox{  on a   } \mbox{    } e^{\lambda_i - \gamma} \|v\| \leq \| A^{i}_{\gamma}(\widehat{x}) v \| \leq e^{\lambda_i + \gamma} \|v\|$$

si $\lambda_i > - \infty$ et  

$$\forall v \in \Cc^{m_l}   \mbox{    } \mbox{     } \mbox{    }  \| A^{l}_{\gamma}(\widehat{x}) v \| \leq e^{- 4 \gamma} \|v\|$$

si $\lambda_l= - \infty$.

\end{Thm}

Pour le dernier point, voir la remarque après le théorème 14 de \cite{Det3}.

Notons maintenant $g_{\widehat{x}}$ la lecture de $f_x$ dans les cartes $C_{\gamma}$ c'est-à-dire $g_{\widehat{x}}= C_{\gamma}^{-1}(\widehat{f}(\widehat{x}))  \circ f_x \circ C_{\gamma}(\widehat{x})$ où $\pi(\widehat{x})=x$. 

Nous allons construire des cartes de Pesin qui seront adaptées pour montrer le Closing Lemma.

\begin{Pro}{\label{proposition10}}

Il existe un ensemble $\widehat{\Gamma'}$ de mesure pleine pour $\widehat{\mu}$ et une application mesurable $r : \widehat{\Gamma'} \rightarrow ]0, 1]$ tels que pour tout $\widehat{x} \in \widehat{\Gamma'}$ on ait $e^{- \gamma} \leq \frac{r(\widehat{f}(\widehat{x}))}{r(\widehat{x})} \leq e^{\gamma}$ et

1) $g_{\widehat{x}}(0)=0$.

2) $D g_{\widehat{x}}(0)=A_{\gamma}(\widehat{x})$.

3) $g_{\widehat{x}}(w)$ est holomorphe pour $\| w \| \leq r(\widehat{x})$ et on a $\|D^2 g_{\widehat{x}}(w)\| \leq  \frac{1}{r(\widehat{x})}$ pour $\| w \| \leq r(\widehat{x})$.

En particulier, si on note $g_{\widehat{x}}(w)=D g_{\widehat{x}}(0)w + h(w)$, on a $\| Dh(w) \| \leq \frac{1}{r(\widehat{x})} \|w\|$ pour $\| w \| \leq r(\widehat{x})$.

\end{Pro}

\begin{proof}

On se place sur $\widehat{\Gamma}$. Les deux premiers points proviennent du théorème précédent.

La démonstration de la troisième propriété est basée sur la proposition 8 de \cite{Det1} (voir aussi la proposition 10 de \cite{Det2}). On reprend les notations de \cite{Det2}.

On a l'existence d'une constante $\epsilon_1 > 0$ telle que $g_{\widehat{x}}(w)$ est holomorphe pour $\| w \| \leq \frac{\epsilon_1 dist(x,I)^p}{\|C_{\gamma}(\widehat{x})\|}$. De plus, pour $\| w \| \leq \frac{\epsilon_1 dist(x,I)^p}{\|C_{\gamma}(\widehat{x})\|}$, on a $D g_{\widehat{x}}(w)=C_{\gamma}^{-1}(\widehat{f}(\widehat{x})) \circ Df_x(C_{\gamma}(\widehat{x}) (w)) \circ C_{\gamma}(\widehat{x})$, ce qui donne la majoration

$$\|D^2 g_{\widehat{x}}(w)\| \leq \|C_{\gamma}^{-1}(\widehat{f}(\widehat{x}))\| \| D^2 f_x(C_{\gamma}(\widehat{x}) (w))\|   \|C_{\gamma}(\widehat{x})\|^2.$$

\noindent Comme dans la preuve de la proposition 10 de \cite{Det2}, en utilisant que $\| D^2f(x) \| \leq C d(x,I)^{-p}$ (voir le lemme 2.1 de \cite {DiDu}) et que l'image par $\tau_x$ du segment $[0, C_{\gamma}(\widehat{x})(w)]$ vit dans la boule $B(x , K \| C_{\gamma}(\widehat{x})(w) \| )$ où $K$ est une constante qui ne dépend que de $X$, on a (toujours pour $\| w \| \leq \frac{\epsilon_1 dist(x,I)^p}{\|C_{\gamma}(\widehat{x})\|}$)

\begin{equation*}
\begin{split}
\|D^2 g_{\widehat{x}}(w)\| &\leq   C \| C_{\gamma}^{-1}(\widehat{f}(\widehat{x}))\|   \| C_{\gamma}(\widehat{x}) \|^2 dist(\tau_x([0, C_{\gamma}(\widehat{x})(w)]), I)^{-p} \\
&\leq C(1-K \epsilon_1)^{-p}  \| C_{\gamma}^{-1}(\widehat{f}(\widehat{x}))\| \|C_{\gamma}(\widehat{x})\|^2  dist(x,I)^{-p} \\
& \leq 2C \|C_{\gamma}^{-1}(\widehat{f}(\widehat{x}))\|  \|C_{\gamma}(\widehat{x})\|^2  dist(x,I)^{-p} .\\
\end{split}
\end{equation*}

Notons maintenant (avec $x= \pi( \widehat{x})$)

$$\alpha(\widehat{x})= \max \left( 1, \frac{\|C_{\gamma}(\widehat{x})\|}{\epsilon_1 dist(x,I)^p}, 2C \|C_{\gamma}^{-1}(\widehat{f}(\widehat{x}))\|  \|C_{\gamma}(\widehat{x})\|^2  dist(x,I)^{-p} \right).$$

Comme $\int \log d(\pi(\widehat{x}),I) d \widehat{\mu} (\widehat{x})= \int \log d(x,I) d \mu(x) > - \infty$, le théorème de Birkhoff implique qu'il existe un ensemble de mesure pleine $ \mathcal{A}$ pour $\widehat{\mu}$ sur lequel

\begin{equation*}
\begin{split}
\lim_{n \rightarrow + \infty} \frac{1}{n} \sum_{i=0}^{n-1} \log d(\pi(\widehat{f}^{i}(\widehat{x})),I)&=\lim_{n \rightarrow + \infty} \frac{1}{n} \sum_{i=0}^{n-1} \log d(\pi(\widehat{f}^{-i}(\widehat{x})),I)\\
=\int \log d(\pi(\widehat{x}),I) d \widehat{\mu} (\widehat{x}) > - \infty.
\end{split}
\end{equation*}

Classiquement, cela implique que pour $\widehat{x} \in \mathcal{A}$, on a

$$\lim_{n \rightarrow \pm \infty} \frac{1}{n}  \log d(\pi(\widehat{f}^{n}(\widehat{x})),I)=0. $$

Comme les $\|C_{\gamma}(\widehat{x})\|$ sont aussi tempérées, on obtient que $\alpha(\widehat{x})$ est une fonction tempérée sur $\widehat{\Gamma'}= \widehat{\Gamma} \cap \mathcal{A}$.

Par le lemme S.2.12 dans \cite{KH}, il existe une fonction $\alpha_{\gamma} \geq \alpha$ qui vérifie $e^{- \gamma} \leq \frac{\alpha_{\gamma}(\widehat{f}(\widehat{x}))}{\alpha_{\gamma}(\widehat{x})} \leq e^{\gamma}$ pour tout $\widehat{x} \in \widehat{\Gamma'}$.

On pose alors $r(\widehat{x})= \frac{1}{\alpha_{\gamma}(\widehat{x})}$ (qui est plus petit que $1$). On a $e^{- \gamma} \leq \frac{r(\widehat{f}(\widehat{x}))}{r(\widehat{x})} \leq e^{\gamma}$ pour tout $\widehat{x} \in \widehat{\Gamma'}$, la fonction $g_{\widehat{x}}(w)$ est holomorphe pour $\| w \| \leq r(\widehat{x})= \frac{1}{\alpha_{\gamma}(\widehat{x})} \leq \frac{1}{\alpha(\widehat{x})} \leq \frac{\epsilon_1 dist(x,I)^p}{\|C_{\gamma}(\widehat{x})\|}$ et enfin

$$\|D^2 g_{\widehat{x}}(w)\|  \leq 2C \|C_{\gamma}^{-1}(\widehat{f}(\widehat{x}))\|  \|C_{\gamma}(\widehat{x})\|^2  dist(x,I)^{-p}  \leq \alpha(\widehat{x}) \leq \alpha_{\gamma}(\widehat{x})= \frac{1}{r(\widehat{x})}$$

pour $\|w\| \leq r(\widehat{x})$. C'est ce que l'on voulait démontrer.

\end{proof}

\subsection{{\bf Transformées de graphes}}

Nous utiliserons les théorèmes de transformées de graphes en avant et en arrière par des applications non inversibles.

Commençons par la transformée de graphe en avant (voir le paragraphe 4 de \cite{Det1}).

On se place dans $\Cc^k$ et dans la suite $\| \cdot \|$ désignera la norme Euclidienne et $B_l(0,R)$ la boule de centre $0$ et de rayon $R$ dans $\Cc^l$.

On considère

$$g(X,Y)=(g_1(X,Y),g_2(X,Y))=(AX + R(X,Y), BY + U(X,Y))$$

avec $(X,0) \in E_1$ (abscisses), $(0,Y) \in E_2$ (ordonnées) et $A: \Cc^{k_1} \longrightarrow \Cc^{k_1}$, $B: \Cc^{k_2} \longrightarrow \Cc^{k_2}$ des applications linéaires avec $k=k_1 + k_2$. On suppose que $g: B_k(0,R_0) \longrightarrow B_k(0, R_1)$ est holomorphe avec $R_0 \leq R_1$, $g(0)=0$ et $\max(\|DR(Z)\|, \|DU(Z)\|) \leq \delta$ sur $B_k(0,R_0)$. 

On supposera aussi $A$ inversible, $\|B\| < \|A^{-1}\|^{-1}$ et on note $\xi=1 - \|B\| \|A^{-1}\| \in ]0,1]$.

Nous utiliserons la version raffinée suivante du théorème de la transformée de graphe (voir le paragraphe 4 de \cite{Det1} ) : 

\begin{Thm}

Soit $\{(X, \phi(X)), X \in D\}$ un graphe dans $B_k(0,R_0)$ au-dessus d'une partie $D$ de $E_1$ qui vérifie $Lip(\phi)\leq \gamma_0 \leq 1$. 

Si $\delta \|A^{-1}\| (1 + \gamma_0)<1$ alors l'image par $g$ de ce graphe est un graphe au-dessus de $\pi_0(g( \mbox{graphe de } \phi))$ où $\pi_0$ est la projection sur les abscisses. Par ailleurs, si $(X, \psi(X))$ désigne ce nouveau graphe, on a:

$$\| \psi(X_1) - \psi(X_2) \| \leq \frac{\| B \| \gamma_0 + \delta( 1 + \gamma_0)}{\|A^{-1}\|^{-1} - \delta(1 + \gamma_0)}\|X_1 - X_2\|$$

qui est inférieur à  $\gamma_0 \|X_1 - X_2\|$ si $\delta \leq \epsilon(\gamma_0, \xi)$.

Enfin, si de plus $B_{k_1}(0, \alpha) \subset D$ et $\| \phi(0)\| \leq \beta$, alors $\pi_0(g( \mbox{graphe de } \phi))$ contient $B_{k_1}(0, (\|A^{-1}\|^{-1} - \delta(1 + \gamma_0)) \alpha  - \delta \beta )$ et $\| \Psi(0) \| \leq (1 + \gamma_0)(\| B \| \beta + \delta \beta + \| D^2g\|_{B_k(0,R_0)} \beta^2)$ (si $\delta \leq \epsilon(\gamma_0, \xi)$).

\end{Thm}

\begin{proof}

Il s'agit de reprendre la preuve de \cite{Det1} et de modifier juste la fin.

Tout est identique jusqu'à la ligne 11 p.102 sauf qu'ici les graphes sont au-dessus des abscisses.

Ensuite, on a

$$ \| \lambda(X_1)-\lambda(X_2)\|\geq (\| A^{-1}\|^{-1}-\delta(1+\gamma_0))\| X_1-X_2\| .$$

En particulier, 

\begin{equation*}
\begin{split}
\| \lambda(X)-\lambda(0)\| & \geq (\| A^{-1}\|^{-1}-\delta(1+\gamma_0))\| X\| \\
& =(\| A^{-1}\|^{-1}-\delta(1+\gamma_0))\alpha \mbox{  si  } X \in \partial B(0,\alpha).
\end{split}
\end{equation*}

Mais $\lambda(0)=A 0 + R(0,\phi(0))$ est de norme inférieure à  

$$\| R(0,\phi(0))\| =\| R(0,\phi(0)) - R(0,0)\|\leq \delta \| \phi(0)\| \leq \delta \beta.$$ 

On en déduit que $\pi_0(g( \mbox{graphe de } \phi))$ contient $B_{k_1}(0, (\|A^{-1}\|^{-1} - \delta(1 + \gamma_0)) \alpha  - \delta \beta )$.

C'est la modification que l'on voulait montrer.

\end{proof}

Nous utiliserons aussi la transformée de graphe en arrière dans le cas non-inversible (voir \cite{Du} théorème 6.4). 

\begin{Thm}{\label{noninversible}}

Soit $(\phi(Y),Y)$ un graphe dans $B_k(0,R_0)$ au-dessus de $B_{k_2}(0, \alpha)$  qui vérifie $\|\phi(0)\| \leq \beta$ et $Lip( \phi) \leq \gamma_0 \leq 1$ pour un certain $\beta \leq \alpha \leq R_0 /4$. Soit $\gamma > 0$ tel que $e^{2 \gamma} < \frac{3}{2}$.

On suppose que

$$\gamma_0(1- \xi)+2 \delta(1 + \gamma_0) \|A^{-1}\| \leq 1  \mbox{   }\mbox{  et }$$

$$(\gamma_0 \|B\|+ \delta(1 + \gamma_0))(\|A^{-1}\|^{-1}-\delta(1 + \gamma_0))^{-1} \leq e^{- \gamma} \gamma_0 \mbox{   }\mbox{  et }$$

$$(\|B\| + 2 \delta)e^{2 \gamma} + \delta \leq 1 \mbox{   }\mbox{  et } \mbox{   } \delta(1+ \gamma_0) \leq \min \{ (\|A^{-1}\|^{-1} - \gamma_0 \| B \|)/2 ,\|A^{-1}\|^{-1} -e^{2 \gamma} \}.$$

Alors il existe $\psi:  B_{k_2} ( 0, \alpha e^{2 \gamma}) \longrightarrow \Cc^{k_1}$ avec $Lip (\psi) \leq e^{- \gamma} \gamma_0$, $\| \psi(0) \| \leq \beta e^{-2\gamma}$ et $g(\mbox{graphe}(\psi)) \subset \mbox{graphe}(\phi)$.

\end{Thm}

\begin{proof}

Il s'agit d'une légère adaptation de la preuve du théorème 6.4 dans \cite{Du}. Nous en reprenons aussi les notations.

Soit $Y \in B_{k_2} ( 0, \alpha e^{2 \gamma})$ et

$$L_{\beta e^{- 2 \gamma}}=\{ (X,Y) \mbox{ , } \|X\| \leq \|Y\| + \beta e^{- 2 \gamma} \}.$$

Tout d'abord, on a $L_{\beta e^{- 2 \gamma}} \subset B_k(0,R_0)$ car pour $(X,Y) \in L_{\beta e^{- 2 \gamma}}$

\begin{equation*}
\begin{split}
\|(X,Y)\| & \leq \|X\| + \|Y\| \leq 2 \|Y\|+ \beta e^{-2 \gamma} \leq 2 \alpha e^{2 \gamma} + \alpha\\
& \leq \frac{R_0}{2} e^{2 \gamma} + \frac{R_0}{4} < R_0.
\end{split}
\end{equation*}

Ensuite, toujours pour $(X,Y) \in L_{\beta e^{- 2 \gamma}}$, on a

\begin{equation*}
\begin{split}
\|g_2(X,Y)\| &\leq \|B\| \|Y\| + \|U(X,Y)\|= \|B\| \|Y\| + \|U(X,Y)-U(0,0)\|\\
& \leq \|B\| \|Y\| + \delta \|(X,Y)\| \leq \|B\| \|Y\| + \delta \|X\| + \delta \|Y\|\\
& \leq ( \|B\| + 2 \delta) \|Y\| + \delta \beta e^{- 2 \gamma}  < ( \|B\| + 2 \delta) \alpha e^{2 \gamma} + \delta \alpha \leq \alpha \\
\end{split}
\end{equation*}

car $( \|B\| + 2 \delta) e^{2 \gamma} + \delta \leq 1$.

On en déduit que $\|g_2(X,Y)\| < \alpha$ pour tout $(X,Y) \in L_{\beta e^{- 2 \gamma}}$. En particulier, si on note

$$U= \{Y \in B_{k_2}(0,R_0) \mbox{  ,  }  g_2(L_{\beta e^{- 2 \gamma}}) \subset B_{k_2}(0,\alpha) \},$$

on a  $B_{k_2} ( 0, \alpha e^{2 \gamma}) \subset U$.

Soit $Y \in B_{k_2} ( 0, \alpha e^{2 \gamma})$ fixé. On veut trouver $(X,Y) \in L_{\beta e^{- 2 \gamma}}$ avec $\phi(g_2(X,Y)) =g_1(X,Y)$. C'est équivalent à trouver $(X,Y) \in L_{\beta e^{- 2 \gamma}}$ avec

$$\Lambda_Y(X)=A^{-1} \left[ \phi(g_2(X,Y)) - (g_1(X,Y)-AX) \right]=X.$$

Pour cela montrons d'abord que $\Lambda_Y: \overline{B_{k_1}(0, \|Y\| + \beta e^{-2 \gamma})} \longrightarrow \overline{B_{k_1}(0, \|Y\| + \beta e^{-2 \gamma})}$. Pour $X \in \overline{B_{k_1}(0, \|Y\| + \beta e^{-2 \gamma})}$, on a:

\begin{equation*}
\begin{split}
\| \Lambda_Y(X) \| & \leq \|A^{-1} \| \left[ \| \phi(g_2(X,Y)) - \phi(0)\| + \| \phi(0)\| + \|g_1(X,Y)-AX\| \right]\\
& \leq \|A^{-1} \| \left[ \gamma_0 ( \|B\| \|Y\| + \|U(X,Y)\|) + \beta + \|R(X,Y)\| \right]\\
& \leq \|A^{-1} \| \left[ \gamma_0 ( \|B\| \|Y\| + \delta \|X\| + \delta \|Y\|) + \beta + \delta \|X\| + \delta \|Y\|  \right]\\
& \leq \|A^{-1} \| \left[ \gamma_0 ( \|B\| \|Y\| +  2\delta \|Y\| ) + \gamma_0 \delta \beta e^{-2 \gamma} + \beta + 2 \delta \|Y\| + \delta \beta e^{-2 \gamma}  \right]\\
& = \|A^{-1} \| \left[ ( \gamma_0 \|B\| +  2\delta (1+ \gamma_0)) \|Y\| + \beta (1+ \delta ( 1+ \gamma_0 )e^{-2 \gamma} ) \right].\\
\end{split}
\end{equation*}

Mais $\gamma_0 \|B\| +  2\delta (1+ \gamma_0) \leq \|A^{-1}\|^{-1}$ et $e^{2 \gamma} + \delta(1 + \gamma_0) \leq \|A^{-1}\|^{-1}$, d'où

$$\| \Lambda_Y(X) \| \leq  \|A^{-1} \| \left[ \|A^{-1}\|^{-1} \|Y\| + \beta e^{- 2 \gamma} \|A^{-1}\|^{-1} \right] \leq \|Y\| + \beta e^{- 2 \gamma}.$$

C'est ce que l'on voulait.

Maintenant, en suivant mot pour mot la preuve de \cite{Du}, on a que $\Lambda_Y$ est une contraction ce qui donne l'existence et l'unicité d'un point fixe dans $\overline{B_{k_1}(0, \|Y\| + \beta e^{-2 \gamma})}$ que l'on note $\Psi(Y)$.

Ensuite, toujours en suivant  \cite{Du} et en utilisant l'hypothèse

$$(\gamma_0 \|B\|+ \delta(1 + \gamma_0))(\|A^{-1}\|^{-1}-\delta(1 + \gamma_0))^{-1} \leq e^{- \gamma} \gamma_0 ,$$

on a $Lip (\psi) \leq e^{- \gamma} \gamma_0$.

Enfin, comme $\Lambda_0: \overline{B_{k_1}(0, \beta e^{-2 \gamma})} \longrightarrow \overline{B_{k_1}(0, \beta e^{-2 \gamma})}$, on obtient que $\| \psi(0) \| \leq \beta e^{-2 \gamma}$. Cela termine la démonstration.

\end{proof}

\section{{\bf Le Closing Lemma}}{\label{closinglemma}}
\par

\subsection{{\bf Cadre et énoncé du théorème}}

On commence par donner le cadre du théorème. On utilisera les notations et les résultats du paragraphe \ref{paragraphePesin}.

Soit $\mu$ une mesure de probabilité avec $\int \log d(x,I) d \mu (x)>-\infty$. On la suppose invariante par $f$, ergodique et hyperbolique, c'est-à-dire que ses exposants de Lyapounov vérifient : 

$$\chi_1\geq  \cdots \geq \chi_{m_0}>0>\chi_{m_0+1}\geq \cdots \geq \chi_{k} \mbox{  pour un  } 1\leq m_0\leq k-1.$$ 

Comme les exposants de Lyapounov ne dépendent pas du choix des cartes pour $X$, on peut supposer que $\mu$ ne charge pas le bord des cartes ($U_i$, $\psi_i$).

On fixe $\gamma>0$ petit devant les exposants de Lyapounov avec $4\gamma<\chi_{m_0}-\chi_{m_0 +1}$ et on applique le théorème d'Oseledets avec ce $\gamma$. 

Soit $0< \delta < \frac{1}{6}$. Pour $\epsilon_0>0$ suffisamment petit, la masse pour $\mu$ d'un $\epsilon_0$-voisinage des bords des cartes qui recouvrent $X$ (pour une métrique fixée sur $X$) est plus petite que $\frac{\delta}{2}$.

On note $V_{\epsilon_0}$ ce voisinage et on a $\widehat{\mu}(\pi^{-1}(V_{\epsilon_0}^c))=\mu(V_{\epsilon_0}^c)\geq 1-\frac{\delta}{2}$.

Par le théorème de Lusin, on peut trouver un ensemble $\Lambda_{\delta}$ compact, inclus dans $\pi^{-1}(V_{\epsilon_0}^c)$ avec $\widehat{\mu}(\Lambda_{\delta})\geq 1-\delta$ et $\widehat{x}\mapsto C_{\gamma}^{\pm 1}(\widehat{x})$, $\widehat{x}\mapsto r(\widehat{x})$ continues sur $\Lambda_{\delta}$. 

On fixe une métrique standard sur $\widehat{\Omega}$ qui vérifie $dist(\widehat{x}, \widehat{y}) \geq dist(\pi(\widehat{x}), \pi(\widehat{y}))$ pour tout $\widehat{x}, \widehat{y} \in \widehat{\Omega}$ et on a alors le théorème : 

\begin{Thm}[Closing Lemma]{\label{theoreme1}}
Soit $\epsilon>0$. Il existe $\eta >0$ tel que si $\widehat{x}\in \Lambda_{\delta}$, $\widehat{f}^m(\widehat{x})\in \Lambda_{\delta}$ pour un certain $m\geq 1$ avec $dist(\widehat{x},\widehat{f}^m(\widehat{x}))<\eta$ alors il existe $z \in X$ tel que : 
\begin{enumerate}
\item $f^m(z)=z$,
\item $dist(f^i(x),f^i(z))\leq \epsilon \times \max(e^{-\gamma i},e^{-\gamma (m-i)})$ pour $0\leq i \leq m$ (où $x=\pi(\widehat{x})$),
\item $z$ est hyperbolique.
\end{enumerate}
\end{Thm}

\subsection{{\bf Preuve du Closing Lemma}}

\subsubsection{{\bf Démonstration de l'existence du point périodique $z$}}

Pour démontrer le Closing Lemma, nous suivons l'approche d'A. Katok (voir \cite{Ka}). Ce qui change c'est que $f$ n'est pas inversible et qu'il y a un ensemble d'indétermination $I$ sur lequel $f$ n'est même pas définie. Pour surmonter ces difficultés, nous utiliserons des transformées de graphes pour applications non inversibles (voir le théorème \ref{noninversible}) et des cartes de Pesin qui prennent en compte $I$.

Soit $0< \gamma_0 < \frac{1}{5}$ fixé suffisamment petit pour que $\gamma_0 < e^{\frac{\gamma}{4}}-1$, $\gamma_0 < 1-e^{- \frac{\gamma}{2}}$ et $\gamma_0e^{\chi_1+\gamma}+e^{-4\gamma} < e^{-\frac{7}{2}\gamma}$. 

$\Lambda_{\delta}$ est compact, donc il existe $r_0>0$ tel que : 

$$\forall \widehat{x}\in \Lambda_{\delta},\ r(\widehat{x})\geq r_0 \mbox{  et  } r_0\leq \| C_{\gamma}(\widehat{x})^{\pm 1}\| \leq \frac{1}{r_0}.$$

Dans un premier temps, nous faisons des transformées de graphes en avant et en arrière dans des boîtes de taille $hr(\widehat{f}^i(\widehat{x}))$ avec $h$ suffisamment petit (nous préciserons à  quel point au fur et à  mesure). 

\bigskip

{\bf{1) Poussées en avant de graphes à  partir de $x$ par $g_{\widehat{x}}, g_{\widehat{f}(\widehat{x})}, \cdots, g_{\widehat{f}^{m-1}(\widehat{x})}$ :} }

Soit $\widehat{x}\in \Lambda_{\delta}$. On note $E_u(\widehat{x})$ la somme directe de tous les $E_i(\widehat{x})$ correspondants aux exposants strictement positifs et $E_s(\widehat{x})$ la somme directe des $E_i(\widehat{x})$ avec des exposants strictement négatifs. 

On a $\mathrm{dim}\ E_u(\widehat{x})=k_1$ et $\mathrm{dim}\ E_s(\widehat{x})=k_2=k-k_1$. 

Pour $i$ compris entre $0$ et $m-1$, on se place dans le repère : 

$$C_{\gamma}^{-1}(\widehat{f}^i(\widehat{x}))E_u(\widehat{f}^i(\widehat{x}))\bigoplus C_{\gamma}^{-1}(\widehat{f}^i(\widehat{x}))E_s(\widehat{f}^i(\widehat{x}))$$

et on a : 

\begin{Lem}{\label{lemme1}}
Soit $(X,\phi(X))$ un graphe au-dessus de $B_{k_1}(0,\alpha)\subset C_{\gamma}^{-1}(\widehat{f}^i(\widehat{x}))E_u(\widehat{f}^i(\widehat{x}))$ où $\alpha=hr(\widehat{f}^i(\widehat{x}))$, $\beta=\|\phi(0)\| \leq h r(\widehat{f}^i(\widehat{x}))$ et $\mathrm{Lip}( \phi) \leq \gamma_0$. \\
Alors l'image par $g_{\widehat{f}^i(\widehat{x})}$ de ce graphe est un graphe $(X,\psi(X))$ au-dessus d'une partie de $C_{\gamma}^{-1}(\widehat{f}^{i+1}(\widehat{x}))E_u(\widehat{f}^{i+1}(\widehat{x}))$ au moins pour $\| X\|\leq hr(\widehat{f}^{i+1}(\widehat{x}))e^{\gamma}$ et $\| \psi(0)\|\leq e^{-2\gamma}\beta$, $\mathrm{Lip}( \psi) \leq \gamma_0e^{-\gamma}$. 
\end{Lem}

\begin{proof}

Si $\chi_{m_0 + 1}= - \infty$, les calculs qui vont suivre restent valables en remplaçant $\chi_{m_0 + 1}$ par $- 5 \gamma$.

Dans le repère précédent $g_{\widehat{f}^{i}(\widehat{x})}$ s'écrit (voir le paragraphe \ref{Pesin}): 

$$g_{\widehat{f}^i(\widehat{x})}(X,Y)=(AX+R(X,Y),BY+U(X,Y))$$

avec $(A,B)=\mathrm{diag}(A_{\gamma}^1(\widehat{f}^i(\widehat{x})),\cdots ,A_{\gamma}^l(\widehat{f}^i(\widehat{x})))$ . On a $\| A^{-1}\| ^{-1}\geq e^{\chi_{m_0}-\gamma}$ et $\| B\| \leq e^{\chi_{m_0+1}+\gamma}$ et pour $\| (X,Y)\| \leq 5hr(\widehat{f}^i(\widehat{x}))$ 

$$\max(\| DR(X,Y)\|,\| DU(X,Y)\|)\leq \frac{1}{r(\widehat{f}^i(\widehat{x}))}\| (X,Y)\| \leq 5h.$$

On applique le théorème de la transformée de graphe en avant avec $R_0=5hr(\widehat{f}^i(\widehat{x}))$.

Tout d'abord, le graphe $(X,\phi(X))$ est bien dans $B_k(0,R_0)$ car
 
$$\| \phi(X)\|\leq \| \phi(0)\|+\gamma_0 \| X\| \leq hr(\widehat{f}^i(\widehat{x}))\gamma_0+h r(\widehat{f}^i(\widehat{x}))\leq hr(\widehat{f}^i(\widehat{x})) (\gamma_0 + 1) \leq 2h r(\widehat{f}^i(\widehat{x})).$$

Ensuite, pour les hypothèses du théorème on a d'une part

$$\delta \| A^{-1}\| (1+\gamma_0) \leq 5he^{-\chi_{m_0}+\gamma}\times 2 < 1$$

pour $h$ petit par rapport à $\gamma$ et d'autre part

$$\frac{\| B\|\gamma_0+\delta(1+\gamma_0)}{\| A^{-1}\|^{-1}-\delta(1+\gamma_0)} \leq  \frac{\gamma_0e^{\chi_{m_0+1}+\gamma}+6h}{e^{\chi_{m_0}-\gamma}-6h} \leq e^{-\gamma}\gamma_0$$

toujours pour $h$ petit car $4\gamma<\chi_{m_0}-\chi_{m_0 +1}$.

Le théorème permet alors d'affirmer que l'image par $g_{\widehat{f}^i(\widehat{x})}$ du graphe de $\phi$ est un graphe $(X,\psi(X))$ au-dessus d'une partie de $C_{\gamma}^{-1}(\widehat{f}^{i+1}(\widehat{x}))E_u(\widehat{f}^{i+1}(\widehat{x}))$ qui contient $B_{k_1}(0,t)$ avec : 

\begin{equation*}
\begin{split}
t & \geq (\| A^{-1}\|^{-1}-\delta(1+\gamma_0))\alpha -\delta \beta \geq (e^{\chi_{m_0}-\gamma}-6h)hr(\widehat{f}^i(\widehat{x}))-5h\times hr(\widehat{f}^i(\widehat{x})) \\
& \geq hr(\widehat{f}^i(\widehat{x}))(e^{\chi_{m_0}-\gamma}-6h-5h)  \geq hr(\widehat{f}^i(\widehat{x}))\underbrace{(e^{4\gamma}-11h)}_{>e^{2\gamma} 
\mbox{  si  } h<<\gamma} \mbox{  car si on veut  } e^{\chi_{m_0}-\gamma}\geq e^{4\gamma}.\\
\end{split}
\end{equation*}

Cette dernière quantité est supérieure à $h r(\widehat{f}^{i+1}(\widehat{x})) e^{\gamma}$ car $\frac{r(\widehat{f}^{i}(\widehat{x}))}{r(\widehat{f}^{i+1}(\widehat{x}))}\geq e^{-\gamma}$.

Enfin on a : 
\begin{equation*}
\begin{split}
\| \psi(0)\| & \leq (1+\gamma_0)(\| B\| \beta+\delta \beta +\| D^2g_{\widehat{f}^i(\widehat{x})}\|_{B(0,R_0)}\beta^2) \\
& \leq e^{\frac{\gamma}{2}}(e^{\chi_{m_0+1}+\gamma}+5h+\frac{1}{r(\widehat{f}^i(\widehat{x}))}hr(\widehat{f}^i(\widehat{x})))\beta \\
& \leq e^{\frac{\gamma}{2}}(e^{-4\gamma}+5h+h)\beta  \mbox{  car si on veut  } e^{\chi_{m_0+1}+\gamma} \leq e^{-4\gamma}\\
& \leq e^{-2\gamma}\beta  \mbox{  si  } h<<\gamma.
\end{split}
\end{equation*}

\noindent Cela achève la démonstration du lemme. 
\end{proof}

On applique maintenant ce lemme de la manière suivante. 

On  part d'un graphe $(X,\phi_0(X))$ dans le repère $C_{\gamma}^{-1}(\widehat{x})E_u(\widehat{x}) \bigoplus C_{\gamma}^{-1}(\widehat{x}) E_s(\widehat{x})$ avec $\phi_0(X)=0$ pour $\|X\| \leq h r(\widehat{x})$.

En poussant en avant par $g_{\widehat{x}}$ on obtient un graphe $(X,\phi_1(X))$ au moins pour $\| X\| \leq hr(\widehat{f}(\widehat{x}))e^{\gamma}$ avec $\| \phi_1(0)\|=0\leq hr(\widehat{f}(\widehat{x})) e^{-\gamma}$ et $\mathrm{Lip}( \phi_1) \leq \gamma_0e^{-\gamma}$. 

Dans ce graphe on ne garde que la partie pour $\| X\|\leq hr(\widehat{f}(\widehat{x}))$ (on fait un cut-off) et on recommence en prenant l'image par $g_{\widehat{f}(\widehat{x})}$. Après $m$ étapes (poussées successives par $g_{\widehat{x}},g_{\widehat{f}(\widehat{x})},\cdots,g_{\widehat{f}^{m-1}(\widehat{x})}$), on obtient un graphe $(X,\phi_m(X))$ dans le repère:

$$C_{\gamma}^{-1}(\widehat{f}^m(\widehat{x}))E_u(\widehat{f}^m(\widehat{x}))\bigoplus C_{\gamma}^{-1}(\widehat{f}^m(\widehat{x}))E_s(\widehat{f}^m(\widehat{x}))$$

au moins pour $\| X\|\leq hr(\widehat{f}^m(\widehat{x}))e^{\gamma}$ avec $\| \phi_m(0)\| =0\leq hr(\widehat{f}^m(\widehat{x}))e^{-\gamma}$ et $\mathrm{Lip}( \phi_m) \leq \gamma_0e^{-\gamma}$. 

On veut maintenant remettre le graphe $(X,\phi_m(X))$ obtenu dans le repère initial:

$$C_{\gamma}^{-1}(\widehat{x})E_u(\widehat{x})\bigoplus C_{\gamma}^{-1}(\widehat{x})E_s(\widehat{x})$$
 
afin de recommencer l'opération. 

\begin{Rem}
On va le faire en supposant seulement que $\| \phi_m(0)\|\leq h r(\widehat{f}^m(\widehat{x}))e^{-\gamma}$ (et non $\| \phi_m(0)\|=0$) car cela nous sera utile dans la suite. 
\end{Rem}

Cela revient à  prendre l'image de $(X,\phi_m(X))$ par l'application:

$$C(X,Y):=C_{\gamma}^{-1}(\widehat{x})\tau_x^{-1}\tau_{f^m(x)}C_{\gamma}(\widehat{f}^m(\widehat{x}))(X,Y).
$$

Par définition, $\Lambda_{\delta}$ ne contient pas des points de $\pi^{-1}(V_{\epsilon_0})$. En particulier, si
$\widehat{x},\widehat{y}\in\Lambda_{\delta}$ avec $\mathrm{dist}(\widehat{x},\widehat{y})<\eta<\epsilon_0$ on a $\mathrm{dist}(x,y)\leq \mathrm{dist}(\widehat{x},\widehat{y})<\epsilon_0$, c'est-à-dire que $x$ et $y$ sont dans une même carte.

On peut supposer que la famille $(\tau_x)_x$ est choisie de sorte que dans ce cas $\tau_x^{-1}\circ \tau_y$ soit juste une translation (on rappelle que $\tau_x$ est la composée $\psi\circ t_x$ où $\psi$ est une carte et $t_x$ une translation). On supposera que les $(\tau_x)_x$ vérifient cette propriété dans la suite. 

Comme $\mathrm{dist}(\widehat{x},\widehat{f}^m(\widehat{x}))<\eta<\epsilon_0$, on a $\tau_x=\psi\circ t_x$ et $\tau_{f^m(x)}=\psi\circ t_{f^m(x)}$.

On en déduit qu'il existe un vecteur $a_m$ tel que $\tau_x^{-1}\tau_{f^m(x)}(w)=w+a_m$, ce qui permet de simplifier l'expression de $C$ sous la forme: 

$$C(w)=g_1(w)+C_{\gamma}^{-1}(\widehat{x})a_m$$

où $g_1(w)=C_{\gamma}^{-1}(\widehat{x})C_{\gamma}(\widehat{f}^m(\widehat{x}))(w)$.

Commençons par prendre l'image du graphe $(X,\phi_m(X))$ par $g_1$. On a: 

\begin{equation*}
\begin{split}
g_1(w) & =C_{\gamma}^{-1}(\widehat{x})\circ C_{\gamma}(\widehat{x})w+C_{\gamma}^{-1}(\widehat{x})\circ (C_{\gamma}(\widehat{f}^m(\widehat{x}))-C_{\gamma}(\widehat{x}))w \\
 & =w+C_{\gamma}^{-1}(\widehat{x})\circ (C_{\gamma}(\widehat{f}^m(\widehat{x}))-C_{\gamma}(\widehat{x}))w \\
\end{split}
\end{equation*}

avec $\| C_{\gamma}^{-1}(\widehat{x})\circ (C_{\gamma}(\widehat{f}^m(\widehat{x}))-C_{\gamma}(\widehat{x}))\|\leq \frac{1}{r_0}\| C_{\gamma}(\widehat{f}^m(\widehat{x}))-C_{\gamma}(\widehat{x})\| $.

Par ailleurs $\widehat{y}\mapsto C_{\gamma}(\widehat{y})$ est continue sur le compact $\Lambda_{\delta}$, donc y est uniformément continue. En particulier, il existe une fonction $\displaystyle\epsilon(\eta)\underset{\eta\to 0}{\longrightarrow} 0$ telle que l'on ait l'implication suivante: 

$$\mathrm{dist}(\widehat{z},\widehat{y})\leq \eta \mbox{  avec  } \widehat{z},\widehat{y}\in \Lambda_{\delta} \Longrightarrow \frac{1}{r_0}\| C_{\gamma}(\widehat{z})-C_{\gamma}(\widehat{y})\| \leq \epsilon(\eta).$$

Pour $\eta$ petit par rapport à $\gamma$ et $\gamma_0$ l'image du graphe $(X,\phi_m(X))$ par $g_1$ va être un graphe $(X,\psi_m^1(X))$. En effet, $g_1$ s'écrit sous la forme $g_1(X,Y)=(AX+CY,BY+DX)$ avec $\| A^{-1}\|^{-1}\geq 1-\epsilon(\eta)$, $\| B\|\leq 1+\epsilon(\eta)$, $\| C\| \leq \epsilon(\eta)$ et $\| D\| \leq \epsilon(\eta)$. 

Si on reprend les notations du théorème de la transformée de graphe on a donc: 

$$\delta \| A^{-1}\|(1+\gamma_0)\leq \epsilon(\eta)\times \frac{1}{1-\epsilon(\eta)}\times 2\underbrace{<}_{\mbox{  pour  } \eta \mbox{  petit  }} 1 \mbox{   } \mbox{   } \mbox{   } \mbox{  et}$$ 

$$\frac{\| B\|\gamma_0 e^{-\gamma}+\delta (1+\gamma_0)}{\| A^{-1}\| ^{-1}-\delta (1+\gamma_0)}\leq \frac{(1+\epsilon(\eta))e^{-\gamma}\gamma_0+2\epsilon(\eta)}{1-\epsilon(\eta)-2\epsilon(\eta)}\underbrace{\leq}_{\mbox{  pour  } \eta<<\gamma,\gamma_0} \gamma_0.$$

Les hypothèses de ce théorème étant vérifiées, on obtient un graphe $(X,\psi_m^1(X))$ avec $\mathrm{Lip}( \psi_m^1) \leq \gamma_0$ au moins au-dessus de $B_{k_1}(0,t)$ avec: 

$$t \geq (1-\epsilon(\eta)-3 \epsilon(\eta))e^{\gamma}hr(\widehat{f}^m(\widehat{x})) = (1-4 \epsilon(\eta))e^{\gamma}hr(\widehat{f}^m(\widehat{x})).$$

Mais $r(\widehat{f}^m(\widehat{x}))\geq e^{-\frac{\gamma}{4}}r(\widehat{x})$ si $\eta$ est petit par rapport à  $\gamma$. 

En effet, $\widehat{x}\mapsto r(\widehat{x})$ est uniformément continue sur $\Lambda_{\delta}$ et on peut donc supposer que: 

$$\mathrm{dist}(\widehat{x},\widehat{y})< \eta \mbox{  avec  } \widehat{x},\widehat{y}\in \Lambda_{\delta} \Longrightarrow | r(\widehat{x})-r(\widehat{y})| \leq (1-e^{-\frac{\gamma}{4}})r_0$$

ce qui donne 

$$r(\widehat{y})\geq r(\widehat{x})-(1-e^{-\frac{\gamma}{4}})r_0\geq r(\widehat{x})-(1-e^{-\frac{\gamma}{4}})r(\widehat{x})=e^{-\frac{\gamma}{4}}r(\widehat{x}).$$

On a donc $t \geq (1-4 \epsilon(\eta))e^{\gamma}he^{-\frac{\gamma}{4}}r(\widehat{x}) > h e^{\frac{\gamma}{4}}r(\widehat{x})$ si $\eta<<1$.

Enfin comme $g_1$ est linéaire,

\begin{equation*}
\begin{split}
\| \psi_m^{1}(0)\| &  \leq (1+\gamma_0)(1+\epsilon(\eta)+\epsilon(\eta))he^{-\gamma}r(\widehat{f}^m(\widehat{x}))\\
 & \leq e^{\frac{\gamma}{4}}(1+2\epsilon(\eta))he^{-\gamma}r(\widehat{f}^m(\widehat{x}))  \leq  e^{-\frac{\gamma}{2}}(1+2\epsilon(\eta))  hr(\widehat{x})
\end{split}
\end{equation*}

car comme précédemment on a $r(\widehat{f}^m(\widehat{x}))\leq e^{\frac{\gamma}{4}}r(\widehat{x})$.

Ainsi, $\| \psi_m^{1}(0)\| \leq   e^{-\frac{\gamma}{4}}h r(\widehat{x})$ si $\eta<<\gamma$.

Il reste désormais à  translater ce graphe par $C_{\gamma}^{-1}(\widehat{x})a_m$ (on rappelle que $C(w)=g_1(w)+C_{\gamma}^{-1}(\widehat{x})a_m$). 

Le vecteur $C_{\gamma}^{-1}(\widehat{x})a_m$ s'écrit $(t_1,t_2)$ donc l'image de $(X,\psi_m^1(X))$ est:

$$(X',\psi_m(X')):=(X+t_1,\psi_m^1(X)+t_2)$$

où $\psi_m(X')=\psi_m^1(X)+t_2=\psi_m^1(X'-t_1)+t_2$. 

C'est donc un graphe au-dessus d'une partie de l'axe des abscisses $C_{\gamma}^{-1}(\widehat{x})E_u(\widehat{x})$. De plus, $\mathrm{Lip}( \psi_m)=\mathrm{Lip}( \psi_m^1) \leq \gamma_0$ et c'est un graphe au moins au-dessus de $B(0,hr(\widehat{x})e^{\frac{\gamma}{4}} -\| C_{\gamma}^{-1}(\widehat{x})\| \| a_m\|)$. 

Mais en utilisant que $a_m=\overrightarrow{\psi^{-1}(x)\psi^{-1}(f^m(x))}$, on a 

$$\| C_{\gamma}^{-1}(\widehat{x})\| \| a_m\| \leq \frac{1}{r_0}\| \psi^{-1}(f^m(x))-\psi^{-1}(x)\|\underbrace{\leq}_{\mbox{  si  } \eta <<h,r_0} \frac{h^2 r_0^2}{r_0}=h^2r_0 \leq h^2r(\widehat{x}).$$

D'où $hr(\widehat{x})e^{\frac{\gamma}{4}}-\| C_{\gamma}^{-1}(\widehat{x})\| \| a_m\|\geq hr(\widehat{x})(e^{\frac{\gamma}{4}}-h) > hr(\widehat{x})$ si $h<<\gamma$. 

Enfin:

\begin{equation*}
\begin{split}
\| \psi_m(0)\| & = \| \psi_m^1(-t_1)+t_2\| \leq \| \psi_m^1(-t_1)-\psi_m^1(0)\|+\| \psi_m^1(0)\|+\| t_2\|\\
& \leq \gamma_0\| -t_1\| +\| t_2\| +e^{-\frac{\gamma}{4}}hr(\widehat{x}) \leq  (1+\gamma_0)\| C_{\gamma}^{-1}(\widehat{x})\| \| a_m\|+e^{-\frac{\gamma}{4}}hr(\widehat{x}) \\
&  \leq e^{\frac{\gamma}{4}}h^2 r(\widehat{x})+e^{-\frac{\gamma}{4}}hr(\widehat{x})  \leq hr(\widehat{x})(h e^{\frac{\gamma}{4}}+e^{-\frac{\gamma}{4}})  \leq hr(\widehat{x}) \mbox{  } \mbox{  pour  } h<<\gamma. \\
\end{split}
\end{equation*}

On a donc bien récupéré un graphe avec les bonnes estimées pour pouvoir recommencer tout le processus (voir le lemme précédent) avec $(X,\psi_m(X))$ à  la place de $(X,\phi_0(X))$. 

On notera $B_0$ le graphe initial $(X,\phi_0(X))$, $B_1$ le graphe obtenu $(X,\psi_m(X))$ et ainsi de suite.
On a ainsi construit une suite $(B_j)_j$ de graphes.

On va maintenant faire la même chose en tirant en arrière. 

\bigskip

{\bf{2) Tirés en arrière de graphes à  partir de $f^m(x)$ par $g_{\widehat{f}^{m-1}(\widehat{x})}, \cdots , g_{\widehat{x}}$: }}

Par hypothèse $\widehat{f}^m(\widehat{x})\in \Lambda_{\delta}$. Pour le lemme qui suit, on se place dans le repère: 

$$C_{\gamma}^{-1}(\widehat{f}^{i+1}(\widehat{x}))E_u(\widehat{f}^{i+1}(\widehat{x}))\bigoplus C_{\gamma}^{-1}(\widehat{f}^{i+1}(\widehat{x}))E_s(\widehat{f}^{i+1}(\widehat{x}))$$

(pour $i$ compris entre $0$ et $m-1$): 

\begin{Lem}{\label{lemme2}}
Soit $(\psi(Y),Y)$ un graphe au-dessus de $B_{k_2}(0,\alpha)\subset C_{\gamma}^{-1}(\widehat{f}^{i+1}(\widehat{x}))E_s(\widehat{f}^{i+1}(\widehat{x}))$ où $\alpha=hr(\widehat{f}^{i+1}(\widehat{x}))$, $\beta=\|\psi(0)\|\leq h r(\widehat{f}^{i+1}(\widehat{x}))$, $\mathrm{Lip}( \psi) \leq\gamma_0$. \\
Alors il existe un graphe $(\phi(Y),Y)$ au-dessus d'une partie de $C_{\gamma}^{-1}(\widehat{f}^{i}(\widehat{x}))E_s(\widehat{f}^{i}(\widehat{x}))$ pour au moins $\| Y\|\leq e^{2\gamma}\alpha$ avec $g_{\widehat{f}^{i}(\widehat{x})}(\mathrm{graphe}\ \phi)\subset \mathrm{graphe}\ \psi$, $\| \phi(0)\|\leq e^{-2\gamma}\beta$ et $\mathrm{Lip}( \phi ) \leq e^{-\gamma}\gamma_0$. 
\end{Lem}
\begin{proof}

Si $\chi_{m_0 + 1}= - \infty$, les calculs qui vont suivre restent valables en remplaçant $\chi_{m_0 + 1}$ par $- 5 \gamma$. Dans le repère:

$$C_{\gamma}^{-1}(\widehat{f}^{i}(\widehat{x}))E_u(\widehat{f}^{i}(\widehat{x}))\bigoplus C_{\gamma}^{-1}(\widehat{f}^{i}(\widehat{x}))E_s(\widehat{f}^{i}(\widehat{x}))$$

$g_{\widehat{f}^i(\widehat{x})}$ s'écrit:

$$g_{\widehat{f}^i(\widehat{x})}(X,Y)=(AX+R(X,Y),BY+U(X,Y))$$

où $(A,B)=\mathrm{diag}(A_{\gamma}^1(\widehat{f}^i(\widehat{x})),\cdots ,A_{\gamma}^l(\widehat{f}^i(\widehat{x})))$, $\| A^{-1}\| ^{-1}\geq e^{\chi_{m_0}-\gamma}$ et $\| B\| \leq e^{\chi_{m_0+1}+\gamma}$.

Comme dans la démonstration du lemme précédent, on a pour $\| (X,Y)\|\leq 5hr(\widehat{f}^{i}(\widehat{x}))=R_0$:

$$\max(\| DR(X,Y)\|,\| DU(X,Y)\|)\leq 5h.$$

On a aussi: 

$$\beta\leq \alpha\leq e^{\gamma}hr(\widehat{f}^i(\widehat{x}))\leq \frac{R_0}{4}.$$

La deuxième inégalité vient du fait que $r(\widehat{f}^{i+1}(\widehat{x}))\leq e^{\gamma}r(\widehat{f}^i(\widehat{x}))$. La dernière découle de $4e^{\gamma}\leq 5$ (qui est vraie pour $\gamma$ suffisamment petit). 

On peut alors appliquer le théorème de la transformée de graphe en arrière. En effet, on a:

$$ \gamma_0(1-\xi)+2\delta (1+\gamma_0)\| A^{-1}\| \leq \gamma_0+6h\times 2\times e^{-\chi_{m_0}+\gamma}\leq \gamma_0+12h<1$$ 

car $\gamma_0<\frac{1}{5}$ et $h$ petit. Ensuite,

$$ \frac{\gamma_0\| B\| +\delta(1+\gamma_0)}{\| A^{-1}\| ^{-1}-\delta(1+\gamma_0)}\leq \frac{\gamma_0 e^{\chi_{m_0+1}+\gamma}+6h}{e^{\chi_{m_0}-\gamma}-6h}<e^{-\gamma}\gamma_0$$ 

si $h<<\gamma,\gamma_0$ car $e^{\chi_{m_0+1}+\gamma-\chi_{m_0}+\gamma}\leq e^{-2\gamma}$. Puis,

$$ (\| B\| +2\delta )e^{2\gamma}+\delta\leq (e^{\chi_{m_0+1}+\gamma}+10h) e^{2 \gamma}+5h<1$$

pour $h<<\gamma$ car $e^{\chi_{m_0+1}+\gamma}\leq e^{-4\gamma}$. Enfin, d'une part,

$$\delta(1+\gamma_0)\leq 6h<\frac{\| A^{-1}\|^{-1}-\gamma_0\| B\|}{2}$$
car $\| A^{-1}\|^{-1}-\gamma_0\| B\|\geq e^{\chi_{m_0}-\gamma}-\gamma_0 e^{\chi_{m_0+1}+\gamma}\geq e^{4\gamma}-\gamma_0\geq e^{2\gamma} > 12h$ par l'hypothèse faite sur $\gamma_0$ au début et le fait que $e^{2\gamma}>12h$ pour $h$ petit. D'autre part, 

$$\| A^{-1}\| ^{-1}-e^{2\gamma}\geq e^{\chi_{m_0}-\gamma}-e^{2\gamma}\geq e^{4\gamma}-e^{2\gamma}>6h$$ 

aussi. Ce qui achève la démonstration du lemme. 
\end{proof}

On va maintenant appliquer ce lemme. On commence par se placer dans le repère 

$$C_{\gamma}^{-1}(\widehat{f}^m(\widehat{x}))E_u(\widehat{f}^m(\widehat{x}))\bigoplus C_{\gamma}^{-1}(\widehat{f}^m(\widehat{x}))E_s(\widehat{f}^m(\widehat{x}))$$

et on part de $(\psi_m(Y),Y)$ avec $\psi_m \equiv 0$ et $Y\in B_{k_2}(0,hr(\widehat{f}^m(\widehat{x})))$. 

Par le lemme précédent, on a l'existence d'un graphe $(\psi_{m-1}(Y),Y)$ avec 

$$g_{\widehat{f}^{m-1}(\widehat{x})}(\mathrm{graphe}\ \psi_{m-1})\subset \mathrm{graphe}\ \psi_{m},$$

$\mathrm{Lip}( \psi_{m-1}) \leq \gamma_0e^{-\gamma}$ et ce graphe est défini au moins pour $\| Y\|\leq e^{2\gamma}hr(\widehat{f}^m(\widehat{x}))$ c'est-à-dire au moins pour $\| Y\|\leq e^{\gamma}hr(\widehat{f}^{m-1}(\widehat{x}))$. De plus $\| \psi_{m-1}(0)\|=0\leq e^{-2\gamma}h r(\widehat{f}^m(\widehat{x}))\leq e^{-\gamma}h r(\widehat{f}^{m-1}(\widehat{x}))$. 

Dans ce graphe on ne garde que la partie pour $\| Y\|\leq hr(\widehat{f}^{m-1}(\widehat{x}))$ (on fait un cut-off) et on recommence avec $g_{\widehat{f}^{m-2}(\widehat{x})}$. 

Après $m$ étapes (tirés en arrière successifs par $g_{\widehat{f}^{m-1}(\widehat{x})}, g_{\widehat{f}^{m-2}(\widehat{x})}, \cdots ,g_{\widehat{x}}$), on obtient un graphe $(\psi_0(Y),Y)$ dans  le repère 

$$C_{\gamma}^{-1}(\widehat{x})E_u(\widehat{x})\bigoplus C_{\gamma}^{-1}(\widehat{x})E_s(\widehat{x})$$

au moins pour $\| Y\|\leq e^{2\gamma}hr(\widehat{f}(\widehat{x}))$ (donc au moins pour $\| Y\|\leq e^{\gamma}hr(\widehat{x})$), $\| \psi_0(0)\|=0\leq e^{-2\gamma}h r(\widehat{f}(\widehat{x}))\leq e^{-\gamma}h r(\widehat{x})$ et $\mathrm{Lip}( \psi_0) \leq e^{-\gamma}\gamma_0$.

On veut maintenant remettre ce graphe dans le repère initial 

$$C_{\gamma}^{-1}(\widehat{f}^m(\widehat{x}))E_u(\widehat{f}^m(\widehat{x}))\bigoplus C_{\gamma}^{-1}(\widehat{f}^m(\widehat{x}))E_s(\widehat{f}^m(\widehat{x}))$$

afin de recommencer l'opération.

On va le faire en supposant seulement $\| \psi_0(0)\|\leq e^{-\gamma}h r(\widehat{x})$ et non pas $0$ car cela nous servira pour les étapes suivantes.

On prend l'image de $(\psi_0(Y),Y)$ par $C^{-1}$. Par les mêmes arguments que précédemment, l'image de $(\psi_0(Y),Y)$ est un graphe $(\phi_0(Y),Y)$ au moins pour $\| Y\|\leq hr(\widehat{f}^m(\widehat{x}))$, avec $\mathrm{Lip}( \phi_0) \leq \gamma_0$ et $\| \phi_0(0)\|\leq h r(\widehat{f}^m(\widehat{x}))$. 

Ce sont les bonnes estimées pour pouvoir recommencer le processus avec $(\phi_0(Y),Y)$ à  la place de $(\psi_m(Y),Y)$. 

On note $A_0$ le graphe initial $(\psi_m(Y),Y)$, $A_1$ celui obtenu $(\phi_0(Y),Y)$. On a donc construit de cette manière une suite de graphes $(A_l)_{l\geq 0}$.

Nous noterons aussi $A_l^0$ le graphe obtenu juste avant de prendre l'image par $C^{-1}$ et créer $A_{l+1}$. Le graphe $A_l^0$ est dans le repère $C_{\gamma}^{-1}(\widehat{x}) E_u(\widehat{x})\bigoplus C_{\gamma}^{-1}(\widehat{x}) E_s(\widehat{x})$. 

Les ensembles $B_j$ et $A_l^0$ se coupent en un seul point (voir la démonstration de la proposition S.3.7 dans \cite{KH} car $\gamma_0^2 < 1$). 

Pour $l,j \geq 0$, on pose: 

$$z_{l,j}=\tau_x\circ C_{\gamma}(\widehat{x})(B_j\cap A_{l}^0).$$

Le lemme suivant énonce des propriétés de la suite $(z_{l,j})_{l,j}$.

\begin{Lem}{\label{lemme3}}
La suite $(z_{l,j})_{l,j}$ vérifie: 
\begin{enumerate}
\item $f^m(z_{l,j})=z_{l-1,j+1}$ pour $l\geq 1,j\geq 0$.
\item $\| C_{\gamma}^{-1}(\widehat{x})\circ \tau_x^{-1}(z_{l_1,j})-C_{\gamma}^{-1}(\widehat{x})\circ \tau_x^{-1}(z_{l_2,j})\|\leq e^{-\gamma}\| C_{\gamma}^{-1}(\widehat{x})\circ \tau_x^{-1}(z_{l_1-1,j+1})-C_{\gamma}^{-1}(\widehat{x})\circ \tau_x^{-1}(z_{l_2-1,j+1})\|$ pour $l_1,l_2\geq 1$ et $j\geq 0$.
\item $\| C_{\gamma}^{-1}(\widehat{x})\circ \tau_x^{-1}(z_{l,j_1})-C_{\gamma}^{-1}(\widehat{x})\circ \tau_x^{-1}(z_{l,j_2})\|\geq e^{\gamma}\| C_{\gamma}^{-1}(\widehat{x})\circ \tau_x^{-1}(z_{l-1,j_1+1})-C_{\gamma}^{-1}(\widehat{x})\circ \tau_x^{-1}(z_{l-1,j_2+1})\|$ pour $j_1,j_2\geq 0$ et $l\geq 1$.
\end{enumerate}
\end{Lem}

\begin{proof}
Par construction des $A_l^0$, on a pour $l\geq 1$:

$$g_{\widehat{f}^{m-1}(\widehat{x})}\circ \cdots \circ g_{\widehat{x}}(A^0_l)\subset A_l \subset B_{k_1}(0,h e^{\gamma} r(\widehat{f}^{m}(\widehat{x})))\times B_{k_2}(0,h r(\widehat{f}^{m}(\widehat{x})))$$
 
et pour $i=0,\cdots,m-1$:

$$g_{\widehat{f}^{i-1}(\widehat{x})}\circ \cdots \circ g_{\widehat{x}}(A^0_l)\subset B_{k_1}(0,hr(\widehat{f}^{i}(\widehat{x})))\times B_{k_2}(0,h e^{\gamma} r(\widehat{f}^{i}(\widehat{x})))$$

(car $\| \psi_i(Y)\| \leq \| \psi_i(Y)-\psi_i(0)\|+\| \psi_i(0)\|\leq \gamma_0 e^{- \gamma} e^{\gamma} h  r(\widehat{f}^{i}(\widehat{x}))+h e^{- \gamma} r(\widehat{f}^{i}(\widehat{x}))\leq hr(\widehat{f}^{i}(\widehat{x}))$). 

Les abscisses de ces ensembles sont donc dans la partie que l'on garde quand on fait le cut-off pour le poussé en avant, c'est-à-dire dans la construction des $B_j$. 

En particulier, $C_{\gamma}(\widehat{x})^{-1}\tau_x^{-1}\tau_{f^m(x)}C_{\gamma}(\widehat{f}^m(\widehat{x}))(g_{\widehat{f}^{m-1}(\widehat{x})}\circ \cdots \circ g_{\widehat{x}}(C_{\gamma}^{-1}(\widehat{x})\tau_x^{-1}z_{l,j}))$ est dans $A_{l-1}^0\cap B_{j+1}$. D'où 

$$\{C_{\gamma}(\widehat{x})^{-1}\tau_x^{-1}f^m(z_{l,j})\}=A_{l-1}^0\cap B_{j+1}=\{C_{\gamma}^{-1}(\widehat{x})\tau_x^{-1}z_{l-1,j+1}\}.$$

Donc $f^m(z_{l,j})=\tau_x\circ C_{\gamma}(\widehat{x})C_{\gamma}^{-1}(\widehat{x})\tau_x^{-1}z_{l-1,j+1}=z_{l-1,j+1}$. 

Ce qui achève la démonstration du point 1. 

Passons aux points 2 et 3 du lemme. Leurs démonstrations utiliseront les deux lemmes suivants: 

\begin{Lem}{\label{lemme4}}
Notons $B_j^i$ pour $i=0,\cdots, m$ les étapes intermédiaires (après cut-off) entre $B_j$ et $B_{j+1}$ lorsque l'on fait les cut-off successifs. En particulier, pour $i=1, \cdots m-1$, $B_j^i$ est un graphe $(X,\psi(X))$ dans le repère 

$$C_{\gamma}^{-1}(\widehat{f}^{i}(\widehat{x}))E_u(\widehat{f}^{i}(\widehat{x}))\bigoplus C_{\gamma}^{-1}(\widehat{f}^{i}(\widehat{x}))E_s(\widehat{f}^{i}(\widehat{x}))$$

au moins au-dessus de $B_{k_1}(0,hr(\widehat{f}^{i}(\widehat{x})))\subset C_{\gamma}^{-1}(\widehat{f}^{i}(\widehat{x}))E_u(\widehat{f}^{i}(\widehat{x}))$, $\mathrm{Lip}( \psi) \leq \gamma_0 e^{-\gamma}$ et $\| \psi(0)\|\leq h r(\widehat{f}^{i}(\widehat{x}))e^{-\gamma}$ (pour $i=0$ ce sont les mêmes estimées sans les $e^{-\gamma}$ et pour $i=m$ c'est $B_{k_1}(0,hr(\widehat{f}^m(\widehat{x}))e^{\gamma})$ au lieu de $B_{k_1}(0,hr(\widehat{f}^m(\widehat{x})))$).

$B_j^0=B_j$ et $C_{\gamma}^{-1}(\widehat{x})\tau_x^{-1}\tau_{f^m(x)}C_{\gamma}(\widehat{f}^m(\widehat{x}))B_j^m \supset B_{j+1}$ (cut-off). 

Alors si $z_1,z_2\in B_j^{i}$ $(i=0,\cdots ,m-1)$ on a:

$$\| g_{\widehat{f}^{i}(\widehat{x})}(z_1)-g_{\widehat{f}^{i}(\widehat{x})}(z_2)\|\geq e^{2\gamma}\| z_1-z_2\|.$$

\end{Lem}

\begin{proof}

On a vu au précédemment (voir lemme \ref{lemme1}) que $g_{\widehat{f}^{i}(\widehat{x})}$ s'écrit: 

$$g_{\widehat{f}^i(\widehat{x})}(X,Y)=(AX+R(X,Y),BY+U(X,Y))$$

où $(A,B)=\mathrm{diag}(A_{\gamma}^1(\widehat{f}^i(\widehat{x})),\cdots ,A_{\gamma}^l(\widehat{f}^i(\widehat{x})))$, $\| A^{-1}\| ^{-1}\geq e^{\chi_{m_0}-\gamma}$, $\| B\| \leq e^{\chi_{m_0+1}+\gamma}$ et $\max(\| DR(X,Y)\|,\| DU(X,Y)\|)\leq 5h$ pour $\| (X,Y)\| \leq 5hr(\widehat{f}^i(\widehat{x}))$.

\noindent Notons $(X,\phi(X))$ le graphe qui représente $B_j^{i}$. On a $z_1=(X_1,\phi(X_1))$ et $z_2=(X_2,\phi(X_2))$.

D'où $\| g_{\widehat{f}^{i}(\widehat{x})}(X_1,\phi(X_1))-g_{\widehat{f}^{i}(\widehat{x})}(X_2,\phi(X_2))\|$ est supérieur à

\begin{equation*}
\begin{split}
& \| A(X_1-X_2)\| -\| R(X_1,\phi(X_1)) - R(X_2,\phi(X_2)) \|  \\
& - \| U(X_1,\phi(X_1)) - U(X_2,\phi(X_2)) \|-\| B\| \| \phi(X_1)-\phi(X_2)\| \\
& \geq \| A^{-1}\|^{-1}\| X_1-X_2\| -5h (1+\gamma_0)\| X_1-X_2\|-5h(1+\gamma_0)\| X_1-X_2\|\\
& -\gamma_0e^{\chi_{m_0+1}+\gamma}\| X_1-X_2\|  \geq (e^{\chi_{m_0}-\gamma}-12h-\gamma_0e^{\chi_{m_0+1}+\gamma})\| X_1 - X_2\|\\
\end{split}
\end{equation*}

(on a utilisé que $\|AX\| \geq \|A^{-1}\|^{-1} \|X\|$). Or:
 
$$e^{\chi_{m_0}-\gamma}-\gamma_0e^{\chi_{m_0+1}+\gamma}\geq e^{4\gamma}-\gamma_0e^{-4\gamma}\underbrace{\geq}_{\gamma_0<<\gamma} e^{3\gamma},$$

donc pour $h<<\gamma,\gamma_0$ on a: 

$$e^{\chi_{m_0}-\gamma}-\gamma_0e^{\chi_{m_0+ 1}+\gamma}-12h\geq e^{3\gamma}-12h\geq e^{\frac{5\gamma}{2}}.$$

De plus, $\| z_1-z_2\| \leq (1+\gamma_0)\| X_1 - X_2\|\leq e^{\frac{\gamma}{4}}\| X_1-X_2\|$ ce qui donne:

$$\| g_{\widehat{f}^{i}(\widehat{x})}(X_1,\phi(X_1))-g_{\widehat{f}^{i}(\widehat{x})}(X_2,\phi(X_2))\|\geq e^{\frac{5\gamma}{2}}\| X_1-X_2\| \geq e^{2\gamma} \| z_1-z_2\|.$$

Cela achève la démonstration.
\end{proof}

\begin{Lem}{\label{lemme5}}
Notons $A_l^i$ pour $i=0,\cdots , m$ les étapes intermédiaires (après cut-off) entre $A_l$ et $A_{l+1}$ lorsque l'on fait les cut-off successifs. En particulier, pour $i=1, \cdots, m-1$, $A_l^i$ est un graphe $(\phi(Y),Y)$ dans le repère 

$$C_{\gamma}^{-1}(\widehat{f}^{i}(\widehat{x}))E_u(\widehat{f}^{i}(\widehat{x}))\bigoplus C_{\gamma}^{-1}(\widehat{f}^{i}(\widehat{x}))E_s(\widehat{f}^{i}(\widehat{x}))$$

 au moins au-dessus de $B_{k_2}(0,hr(\widehat{f}^{i}(\widehat{x})))\subset C_{\gamma}^{-1}(\widehat{f}^{i}(\widehat{x}))E_s(\widehat{f}^{i}(\widehat{x}))$ avec $\mathrm{Lip}( \phi ) \leq \gamma_0 e^{-\gamma}$ et $\| \phi(0)\|\leq h r(\widehat{f}^{i}(\widehat{x}))e^{-\gamma}$ (pour $i=m$ ce sont les mêmes estimées sans les $e^{-\gamma}$ et pour $i=0$ c'est $B_{k_2}(0,hr(\widehat{x})e^{\gamma})$ au lieu de $B_{k_2}(0,hr(\widehat{x}))$). 

$A_l^0$ est le même que défini précédemment et $A_{l}^{m}=A_{l}$

Alors si $z_1,z_2\in A_l^{i}$ $(i=0,\cdots ,m-1)$ on a: 

$$\| g_{\widehat{f}^{i}(\widehat{x})}(z_1)-g_{\widehat{f}^{i}(\widehat{x})}(z_2)\|\leq e^{-2\gamma}\| z_1-z_2\|.$$
\end{Lem}

\begin{proof}

On reprend les mêmes notations que dans le lemme précédent. 

Notons $(\phi(Y),Y)$ le graphe qui représente $A_l^{i}$. On a $z_1=(\phi(Y_1),Y_1)$ et $z_2=(\phi(Y_2),Y_2)$. D'où: 

\begin{equation*}
\begin{split}
& \| g_{\widehat{f}^{i}(\widehat{x})}(\phi(Y_1),Y_1)-g_{\widehat{f}^{i}(\widehat{x})}(\phi(Y_2),Y_2)\| \\
& \leq \| A\| \| \phi(Y_1)-\phi(Y_2)\| +\| R(\phi(Y_1),Y_1) - R(\phi(Y_2),Y_2) \|  \\
 & +\| U(\phi(Y_1),Y_1) - U(\phi(Y_2),Y_2) \|+\| B\| \| Y_1-Y_2\|\\
& \leq e^{\chi_1+\gamma}\gamma_0\| Y_1-Y_2\| +5h (1+\gamma_0)\| Y_1-Y_2\|+5h(1+\gamma_0)\| Y_1-Y_2\|\\
 & +e^{\chi_{m_0+1}+\gamma}\| Y_1-Y_2\| \leq (\gamma_0e^{\chi_{1}+\gamma}+12h+e^{\chi_{m_0 +1}+\gamma})\| Y_1 - Y_2\|.\\
\end{split}
\end{equation*}

Or: 

$$ \gamma_0e^{\chi_{1}+\gamma}+e^{\chi_{m_0+1}+\gamma}\leq \gamma_0e^{\chi_{1}+\gamma}+ e^{-4\gamma}\underbrace{\leq}_{\gamma_0<<\gamma} e^{-\frac{7\gamma}{2}}.$$

Donc pour $h<<\gamma,\gamma_0$, on a $\gamma_0e^{\chi_{1}+\gamma}+12h+e^{\chi_{m_0 + 1}+\gamma}\leq e^{-3\gamma}$.

De plus $\| z_1-z_2\| \geq (1-\gamma_0)\| Y_1 - Y_2\|\geq e^{-\gamma}\| Y_1-Y_2\|$ ce qui donne: 

$$\| g_{\widehat{f}^{i}(\widehat{x})}(\phi(Y_1),Y_1)-g_{\widehat{f}^{i}(\widehat{x})}(\phi(Y_2),Y_2)\|\leq e^{-3\gamma}\| Y_1-Y_2\| \leq e^{-2\gamma} \| z_1-z_2\|.$$

C'est ce que l'on voulait démontrer.

\end{proof}

Utilisons maintenant ces deux lemmes pour montrer les points 2 et 3 du lemme \ref{lemme3}.

\bigskip

{\bf{Démonstration du point 2:}}

On a vu que pour $i=0,\cdots ,m-1$, on a $g_{\widehat{f}^{i}(\widehat{x})}\circ \cdots \circ g_{\widehat{x}}(C_{\gamma}^{-1}(\widehat{x})\tau_x^{-1}(z_{l_1,j}))$ dans la partie que l'on garde quand on fait le cut-off pour construire $B_{j+1}$, c'est-à-dire qu'il est dans $B_j^{i+1}$. En particulier, par le lemme \ref{lemme4} on a: 

\begin{equation*}
\begin{split}
& \| g_{\widehat{f}^{m-1}(\widehat{x})}\circ \cdots \circ g_{\widehat{x}}(C_{\gamma}^{-1}(\widehat{x})\tau_x^{-1}z_{l_1,j})-g_{\widehat{f}^{m-1}(\widehat{x})}\circ \cdots \circ g_{\widehat{x}}(C_{\gamma}^{-1}(\widehat{x})\tau_x^{-1}z_{l_2,j})\| \\
& \geq e^{2\gamma m}\| C_{\gamma}^{-1}(\widehat{x})\tau_x^{-1}z_{l_1,j}-C_{\gamma}^{-1}(\widehat{x})\tau_x^{-1}z_{l_2,j}\|.\\
\end{split}
\end{equation*}

Or: 

$$\tau_{f^m(x)}\circ C_{\gamma}(\widehat{f}^m(\widehat{x}))\circ g_{\widehat{f}^{m-1}(\widehat{x})}\circ \cdots \circ g_{\widehat{x}}(C_{\gamma}^{-1}(\widehat{x})\tau_x^{-1}z_{l_1,j})=f^m(z_{l_1,j})=z_{l_1-1,j+1}$$

par le premier point du lemme \ref{lemme3}. 

Donc si on note $C=C_{\gamma}^{-1}(\widehat{x}) \tau_x^{-1}\tau_{f^m(x)}C_{\gamma}(\widehat{f}^m(\widehat{x}))$, on a: 

\begin{equation*}
\begin{split}
& \| C_{\gamma}^{-1}(\widehat{x})\tau_x^{-1}z_{l_1,j}-C_{\gamma}^{-1}(\widehat{x}) \tau_x^{-1}z_{l_2,j}\|\\
& \leq e^{-2\gamma m}\| C^{-1}(C\circ g_{\widehat{f}^{m-1}(\widehat{x})}\circ \cdots \circ g_{\widehat{x}}(C_{\gamma}^{-1}(\widehat{x})\tau_x^{-1}z_{l_1,j})) \\
& -C^{-1}(C\circ g_{\widehat{f}^{m-1}(\widehat{x})}\circ \cdots \circ g_{\widehat{x}}(C_{\gamma}^{-1}(\widehat{x})\tau_x^{-1}z_{l_2,j}))\|\\
& \leq   e^{-2\gamma m}  \| DC^{-1}\| \times \| C_{\gamma}^{-1}(\widehat{x}
)\tau_x^{-1}z_{l_1-1,j+1}-C_{\gamma}^{-1}(\widehat{x}
)\tau_x^{-1}z_{l_2-1,j+1}\|. \\
\end{split}
\end{equation*}

Or $C^{-1}(w)=C_{\gamma}^{-1}(\widehat{f}^m(\widehat{x}))C_{\gamma}(\widehat{x})w+C_{\gamma}^{-1}(\widehat{f}^m(\widehat{x}))b_m$, donc $DC^{-1}=C_{\gamma}^{-1}(\widehat{f}^m(\widehat{x}))C_{\gamma}(\widehat{x})$ ce qui implique $\| DC^{-1}\|\leq e^{\gamma}$ si $\eta<<\gamma$. 

On en déduit le point 2 du lemme \ref{lemme3}.

\bigskip

{\bf{Démonstration du point 3:}}

Pour $i=0,\cdots,m-1$, on a 

$g_{\widehat{f}^{i}(\widehat{x})}\circ \cdots \circ g_{\widehat{x}}(C_{\gamma}^{-1}(\widehat{x})\tau_x^{-1}z_{l,j_1})$ dans $A_{l}^{i+1}$ (par construction des $A_l^{i}$) d'où par le lemme \ref{lemme5}: 

\begin{equation*}
\begin{split}
& \| g_{\widehat{f}^{m-1}(\widehat{x})}\circ \cdots \circ g_{\widehat{x}}(C_{\gamma}^{-1}(\widehat{x})\tau_x^{-1}z_{l,j_1})-g_{\widehat{f}^{m-1}(\widehat{x})}\circ \cdots \circ g_{\widehat{x}}(C_{\gamma}^{-1}(\widehat{x})\tau_x^{-1}z_{l,j_2})\| \\
& \leq e^{-2\gamma m }\| C_{\gamma}^{-1}(\widehat{x})\tau_x^{-1}z_{l,j_1}-C_{\gamma}^{-1}(\widehat{x})\tau_x^{-1}z_{l,j_2}\|. \\
\end{split}
\end{equation*}

Or: 

$$\tau_{f^m(x)}\circ C_{\gamma}(\widehat{f}^m(\widehat{x}))\circ g_{\widehat{f}^{m-1}(\widehat{x})}\circ \cdots \circ g_{\widehat{x}}(C_{\gamma}^{-1}(\widehat{x})\tau_x^{-1}z_{l,j_1})=f^m(z_{l,j_1})=z_{l-1,j_1+1}.$$

D'où, si on note toujours $C=C_{\gamma}^{-1}(\widehat{x})\tau_x^{-1}\tau_{f^m(x)}C_{\gamma}(\widehat{f}^m(\widehat{x}))$, on a:

\begin{equation*}
\begin{split}
& \| C_{\gamma}^{-1}(\widehat{x})\tau_x^{-1}(z_{l-1,j_1+1})-C_{\gamma}^{-1}(\widehat{x})\tau_x^{-1}(z_{l-1,j_2+1})\| \\
& =  \| C_{\gamma}^{-1}(\widehat{x})\tau_x^{-1}f^m(z_{l,j_1})-C_{\gamma}^{-1}(\widehat{x})\tau_x^{-1}f^m(z_{l,j_2})\|\\
&=\| C\circ g_{\widehat{f}^{m-1}(\widehat{x})}\circ \cdots \circ g_{\widehat{x}}(C_{\gamma}^{-1}(\widehat{x})\tau_x^{-1}z_{l,j_1})-C\circ g_{\widehat{f}^{m-1}(\widehat{x})}\circ \cdots \circ g_{\widehat{x}}(C_{\gamma}^{-1}(\widehat{x})\tau_x^{-1}z_{l,j_2})\|\\
& \leq \| DC\| e^{-2\gamma m}\| C_{\gamma}^{-1}(\widehat{x})\tau_{x}^{-1}(z_{l,j_1})- C_{\gamma}^{-1}(\widehat{x})\tau_{x}^{-1}(z_{l,j_2})\|. \\
\end{split}
\end{equation*}

Comme précédemment on a $DC(w)=C_{\gamma}^{-1}(\widehat{x})C_{\gamma}(\widehat{f}^m(\widehat{x}))w$ avec $\| DC\|\leq e^{\gamma}$ si $\eta<<\gamma$, d'où le point 3 du lemme \ref{lemme3}. 

Cela termine la démonstration du lemme \ref{lemme3}.

\end{proof}

On continue la démonstration de l'existence du point $z$ et on va utiliser le lemme \ref{lemme3} pour cela. 

L'idée est de démontrer que la suite $(z_{l+1,l})_l$ converge vers un point périodique de $f$. 

Par ce lemme \ref{lemme3}, on a 

\begin{equation*}
\begin{split}
& \| C_{\gamma}^{-1}(\widehat{x})\tau_x^{-1}z_{l+1,l}-C_{\gamma}^{-1}(\widehat{x})\tau_x^{-1}z_{l,l}\| \\
& \leq e^{-\gamma l}\| C_{\gamma}^{-1}(\widehat{x})\tau_x^{-1}z_{1,2l}-C_{\gamma}^{-1}(\widehat{x})\tau_x^{-1}z_{0,2l}\|  \leq  2hr(\widehat{x})e^{-\gamma l }\\  
\end{split}
\end{equation*}

car $C_{\gamma}^{-1}(\widehat{x})\tau_x^{-1}z_{1,2l}$ et $C_{\gamma}^{-1}(\widehat{x})\tau_x^{-1}z_{0,2l}$ sont dans $B_{2l}$.

De même: 

\begin{equation*}
\begin{split}
& \| C_{\gamma}^{-1}(\widehat{x})\tau_x^{-1}z_{l,l+1}-C_{\gamma}^{-1}(\widehat{x})\tau_x^{-1}z_{l,l}\|\\
& \leq  e^{-\gamma l}\| C_{\gamma}^{-1}(\widehat{x})\tau_x^{-1}z_{2l,1}-C_{\gamma}^{-1}(\widehat{x})\tau_x^{-1}z_{2l,0}\| \leq 2hr(\widehat{x})e^{-\gamma l}, \\
\end{split}
\end{equation*}

et pour $l\geq 1$,

\begin{equation*}
\begin{split}
& \| C_{\gamma}^{-1}(\widehat{x})\tau_x^{-1}z_{l,l-1}-C_{\gamma}^{-1}(\widehat{x})\tau_x^{-1}z_{l,l}\| \\
& \leq e^{-\gamma (l-1)}\| C_{\gamma}^{-1}(\widehat{x})\tau_x^{-1}z_{2l-1,0}-C_{\gamma}^{-1}(\widehat{x})\tau_x^{-1}z_{2l-1,1}\| \leq  2hr(\widehat{x})e^{-\gamma (l-1)}. \\
\end{split}
\end{equation*}

En combinant ces inégalités, on obtient: 

$$ \| C_{\gamma}^{-1}(\widehat{x})\tau_x^{-1}z_{l+1,l}-C_{\gamma}^{-1}(\widehat{x})\tau_x^{-1}z_{l,l-1}\|  \leq 4hr(\widehat{x})e^{-\gamma (l-1)}.$$

Cela implique que la suite $(C_{\gamma}^{-1}(\widehat{x})\tau_x^{-1}z_{l+1,l})_{l\geq 0}$ est une suite de Cauchy de $\mathbb{C}^k$. 

En effet si $q \geq l$ on a: 

\begin{equation*}
\begin{split}
 & \| C_{\gamma}^{-1}(\widehat{x})\tau_x^{-1}z_{q+1,q}-C_{\gamma}^{-1}(\widehat{x})\tau_x^{-1}z_{l+1,l}\| \\
& \leq \| C_{\gamma}^{-1}(\widehat{x})\tau_x^{-1}z_{q+1,q}-C_{\gamma}^{-1}(\widehat{x})\tau_x^{-1}z_{q,q-1}\|+ \cdots \\
&+\| C_{\gamma}^{-1}(\widehat{x})\tau_x^{-1}z_{l+2,l+1}-C_{\gamma}^{-1}(\widehat{x})\tau_x^{-1}z_{l+1,l}\| \leq \sum_{i\geq l} 4he^{-\gamma i}\underset{l\rightarrow +\infty}{\longrightarrow} 0. \\
\end{split}
\end{equation*}

La suite $(C_{\gamma}^{-1}(\widehat{x})\tau_x^{-1}z_{l+1,l})_l$ converge donc vers un point que l'on note $C_{\gamma}^{-1}(\widehat{x})\tau_x^{-1}z$. 

De plus, 

$$\| C_{\gamma}^{-1}(\widehat{x})\tau_x^{-1}z_{l+1,l}-C_{\gamma}^{-1}(\widehat{x})\tau_x^{-1}z_{l,l+1}\|\leq 2hr(\widehat{x})e^{-\gamma l}+2hr(\widehat{x})e^{-\gamma l}\underset{l\rightarrow +\infty}{\longrightarrow} 0$$ 

ce qui montre que $z_{l,l+1}\underset{l\rightarrow +\infty}{\longrightarrow} z$ aussi. 

Montrons que le point $z$ obtenu vérifie $f^m(z)=z$. 

Pour $i=0,\cdots , m-1$ le point $g_{\widehat{f}^{i-1}(\widehat{x})}\circ \cdots \circ g_{\widehat{x}}(C_{\gamma}^{-1}(\widehat{x})\tau_x^{-1}z_{l+1,l})$ est dans l'ensemble $B_{k_1}(0,hr(\widehat{f}^{i}(\widehat{x})))\times B_{k_2}(0,e^{\gamma} hr(\widehat{f}^{i}(\widehat{x})))$ sur lequel $g_{\widehat{f}^{i}(\widehat{x})}$ est continue. Ainsi,

$$f^m(z_{l+1,l})=\tau_{f^m(x)}C_{\gamma}(\widehat{f}^m(\widehat{x}))g_{\widehat{f}^{m-1}(\widehat{x})}\circ \cdots \circ g_{\widehat{x}}C_{\gamma}^{-1}(\widehat{x})\tau_x^{-1}(z_{l+1,l})\underset{l\rightarrow +\infty}{\longrightarrow} f^m(z).$$ 

Comme $f^m(z_{l+1,l})=z_{l,l+1}\underset{l\rightarrow +\infty}{\longrightarrow} z$, on obtient bien que $f^m(z)=z$. 

Cela termine donc la démonstration du point $1$ du Closing Lemma, c'est-à-dire l'existence du point périodique $z$.  

\subsubsection{{\bf Démonstration du deuxième point du Closing Lemma:}}{\label{point2}}

Il s'agit d'estimer la distance entre $f^{i}(x)$ et $f^{i}(z)$ pour $i=0, \cdots ,m$. 

On se place dans le repère $C_{\gamma}^{-1}(\widehat{x})E_u(\widehat{x})\bigoplus C_{\gamma}^{-1}(\widehat{x})E_s(\widehat{x})$ et on note $y_l:=B_0\cap A_l^0$. 

Pour $0\leq i \leq m-1 $, on a:

\begin{equation*}
\begin{split}
& \| g_{\widehat{f}^{i}(\widehat{x})}\circ \cdots \circ g_{\widehat{x}}(C_{\gamma}^{-1}(\widehat{x})\tau_x^{-1}z_{l,l+1}) \| \\
& \leq \| g_{\widehat{f}^{i}(\widehat{x})}\circ \cdots \circ g_{\widehat{x}}(C_{\gamma}^{-1}(\widehat{x})\tau_x^{-1}z_{l,l+1})-g_{\widehat{f}^{i}(\widehat{x})}\circ \cdots \circ g_{\widehat{x}}(y_l)\|+ \| g_{\widehat{f}^{i}(\widehat{x})}\circ \cdots \circ g_{\widehat{x}}(y_l)\|.\\
\end{split}
\end{equation*}

Mais $g_{\widehat{f}^{i}(\widehat{x})}\circ \cdots \circ g_{\widehat{x}}(y_l)\in g_{\widehat{f}^{i}(\widehat{x})}\circ \cdots \circ g_{\widehat{x}}(A_l^0)$ est dans la partie que l'on garde quand on fait le cut-off pour construire $B_1$ (voir le début de la démonstration du lemme \ref{lemme3}), il est donc dans $B_0^{i+1}$. Par le lemme \ref{lemme4}, on a donc pour $i=0,\cdots,m-1$: 

\begin{equation*}
\begin{split}
& \| g_{\widehat{f}^{i}(\widehat{x})}\circ \cdots \circ g_{\widehat{x}}(y_l)\| = \| g_{\widehat{f}^{i}(\widehat{x})}\circ \cdots \circ g_{\widehat{x}}(0)-g_{\widehat{f}^{i}(\widehat{x})}\circ \cdots \circ g_{\widehat{x}}(y_l)\| \\
& \leq e^{-2\gamma (m-i-1)}\| g_{\widehat{f}^{m-1}(\widehat{x})}\circ \cdots \circ g_{\widehat{x}}(0)-g_{\widehat{f}^{m-1}(\widehat{x})}\circ \cdots \circ g_{\widehat{x}}(y_l)\| \\
& \leq e^{-2\gamma(m-i-1)}\times 2hr(\widehat{f}^m(\widehat{x})) \leq 2he^{-2\gamma(m-i-1)}. \\
\end{split}
\end{equation*}

Par ailleurs, par le lemme \ref{lemme5}, on a aussi pour $i=0,\cdots,m-1$: 

\begin{equation*}
\begin{split}
& \| g_{\widehat{f}^{i}(\widehat{x})}\circ \cdots \circ g_{\widehat{x}}(C_{\gamma}^{-1}(\widehat{x})\tau_x^{-1}z_{l,l+1})-g_{\widehat{f}^{i}(\widehat{x})}\circ \cdots \circ g_{\widehat{x}}(y_l)\|  \\
& \leq  e^{-2\gamma i }\| C_{\gamma}^{-1}(\widehat{x})\tau_x^{-1}z_{l,l+1}-y_l\| \leq 2hr(\widehat{x})e^{-2\gamma i} \leq 2he^{-2\gamma i} .\\
\end{split}
\end{equation*}

En combinant les deux inégalités, on en déduit que

\begin{equation*}
\begin{split}
 \| g_{\widehat{f}^{i}(\widehat{x})}\circ \cdots \circ g_{\widehat{x}}(C_{\gamma}^{-1}(\widehat{x})\tau_x^{-1}z_{l,l+1})\| & \leq  2he^{-2\gamma (m-i-1)}+2he^{-2\gamma i } \\
&\leq 4h \max(e^{-2\gamma i}, e^{-2\gamma (m-i-1)}). \\
\end{split}
\end{equation*}

On fait tendre $l$ vers l'infini, et par continuité (comme précédemment), on obtient: 

$$\|  g_{\widehat{f}^{i}(\widehat{x})}\circ \cdots \circ g_{\widehat{x}}(C_{\gamma}^{-1}(\widehat{x})\tau_x^{-1}z)\|\leq 4h\max(e^{-2\gamma i},e^{-2\gamma (m-i-1)}).$$

Maintenant comme $f^{i}(z) =\tau_{f^{i}(x)}\circ C_{\gamma}(\widehat{f}^{i}(\widehat{x}))\circ  g_{\widehat{f}^{i-1}(\widehat{x})}\circ \cdots \circ g_{\widehat{x}}(C_{\gamma}^{-1}(\widehat{x})\tau_x^{-1}z)$ et $f^{i}(x) = \tau_{f^{i}(x)}\circ C_{\gamma}(\widehat{f}^{i}(\widehat{x}))(0)$, par le théorème des accroissements finis, on a pour $i=0,\cdots,m$: 

$$\mathrm{dist}(f^{i}(x),f^{i}(z))\leq C(X)\| C_{\gamma}(\widehat{f}^{i}(\widehat{x}))\|\times \| g_{\widehat{f}^{i-1}(\widehat{x})}\circ \cdots \circ g_{\widehat{x}}(C_{\gamma}^{-1}(\widehat{x})\tau_x^{-1}z)\|. $$

Mais $\| C_{\gamma}(\widehat{f}^{i}(\widehat{x}))\|\leq e^{\gamma i}\| C_{\gamma}(\widehat{x}) \|\leq \frac{e^{\gamma i}}{r_0}$ car $\| C_{\gamma}(\widehat{x})\|$ est tempérée et $\widehat{x}\in \Lambda_{\delta}$. De même, comme $\widehat{f}^m(\widehat{x})\in \Lambda_{\delta}$, on a $\| C_{\gamma}(\widehat{f}^{i}(\widehat{x}))\|\leq e^{\gamma (m-i)}\| C_{\gamma}(\widehat{f}^m(\widehat{x})) \|\leq \frac{e^{\gamma (m-i)}}{r_0}$. 

On obtient donc, pour $i=0,\cdots,m$, 

\begin{equation*}
\begin{split}
 & \mathrm{dist}(f^{i}(x),f^{i}(z)) \leq C(X)\| C_{\gamma}(\widehat{f}^{i}(\widehat{x}))\| 4h\max(e^{-2\gamma (i-1)}, e^{-2\gamma (m-i)})\\
& \leq C(X)e^{2\gamma}\| C_{\gamma}(\widehat{f}^{i}(\widehat{x}))\| 4h\max(e^{-2\gamma i}, e^{-2\gamma (m-i)}) \leq \frac{C(X)e^{2\gamma}4h}{r_0}\max(e^{-\gamma i},e^{-\gamma(m-i)}). \\
\end{split}
\end{equation*}

Il suffit alors de prendre $h$ suffisamment petit pour que $\frac{C(X)4he^{2\gamma}}{r_0}<\epsilon$ et on obtient l'estimée que l'on voulait. 

Il reste le point $3$ du Closing Lemma à  montrer, c'est-à-dire que $z$ est hyperbolique.

\subsubsection{{\bf Démonstration de l'hyperbolicité de $z$}} 

Nous allons montrer que $Df^m(z)$ a $m_0$ valeurs propres de module supérieur à  $e^{\gamma}$ et $k-m_0$ valeurs propres de module inférieur à  $e^{-\gamma}$ (on compte avec multiplicité). 

Pour cela on va construire des variétés stables et instables en $z$ à  partir des suites $(B_j)_j$ et $(A_l)_l$ précédentes. Commençons par la variété instable. 

On se place dans le repère $C_{\gamma}^{-1}(\widehat{x})E_u(\widehat{x})\bigoplus C_{\gamma}^{-1}(\widehat{x})E_s(\widehat{x})$ et on note $G_h$ l'ensemble des graphes de fonctions holomorphes $(X,\psi(X))$ au-dessus de $\overline{B_{k_1}(0,hr(\widehat{x}))}\subset C_{\gamma}^{-1}(\widehat{x})E_u(\widehat{x})$ avec $\mathrm{Lip}( \psi) \leq \gamma_0$ et $\| \psi(0)\| \leq h r(\widehat{x})$. Les $B_j$ sont dans $G_h$ par construction. 

On munit $G_h$ de la métrique suivante: 

$$d(\mathrm{graphe} \mbox{    } \phi,\mathrm{graphe} \mbox{    } \psi)=\max_{t\in \overline{B_{k_1}(0,hr(\widehat{x}))}} \| \phi(t)-\psi(t)\|.$$

Par le théorème d'Ascoli, $(G_h,d)$ est relativement compact (on a  une famille équicontinue car $\mathrm{Lip}( \psi) \leq \gamma_0$ et $\| \psi(t)\|\leq \|\psi(0)\| + \| \psi(t)-\psi(0)\|\leq h r(\widehat{x})+\gamma_0\| t\|\leq hr(\widehat{x})(1 + \gamma_0)$). 

Une limite uniforme de fonctions holomorphes étant holomorphe, $(G_h,d)$ est même compact. On a alors: 

\begin{Lem}{\label{lemme6}}
$(B_j)_j$ est une suite de Cauchy de $(G_h,d)$. En particulier il existe $B\in G_h$ tel que $B_j\underset{j\to +\infty}{\longrightarrow} B$. 
\end{Lem}

\begin{proof}

Comme dans la fin de la preuve de l'existence de $z$, il suffit de démontrer que: 

$$\forall j\geq 0,\ \mathrm{d}(B_j,B_{j+1})\leq e^{-\gamma j}\mathrm{d}(B_0,B_1)\leq e^{-\gamma j}3hr(\widehat{x}).$$

Pour cela il est suffisant de voir que: 

$$\forall j\geq 0,\ \mathrm{d}(B_{j+1},B_{j+2})\leq e^{-\gamma}\mathrm{d}(B_{j},B_{j+1}).$$

Dans le repère $C_{\gamma}^{-1}(\widehat{f}^{i}(\widehat{x}))E_u(\widehat{f}^{i}(\widehat{x}))\bigoplus C_{\gamma}^{-1}(\widehat{f}^{i}(\widehat{x}))E_s(\widehat{f}^{i}(\widehat{x}))$, on considère pour $i=0,\cdots, m-1$, l'ensemble $G_h^i$ des graphes de fonctions holomorphes $(X,\psi(X))$ au-dessus de $\overline{B_{k_1}(0,hr(\widehat{f}^{i}(\widehat{x}))}$ avec $\mathrm{Lip}( \psi) \leq \gamma_0$ et $ \| \psi(0)\| \leq hr(\widehat{f}^{i}(\widehat{x})) $ muni de la distance $d_i$ définie par: 

$$d_i(\mathrm{graphe}\ \phi,\mathrm{graphe}\ \psi)=\max_{t\in \overline{B_{k_1}(0,hr(\widehat{f}^{i}(\widehat{x})))}} \| \phi(t)-\psi(t)\|.$$

On a $d_0=d$. On considère aussi $G_h^m$ l'ensemble des graphes de fonctions holomorphes $(X,\psi(X))$ au-dessus de $\overline{B_{k_1}(0,e^{\gamma}hr(\widehat{f}^{m}(\widehat{x}))}$ avec $\mathrm{Lip}( \psi) \leq \gamma_0$ et $\| \psi(0)\| \leq hr(\widehat{f}^{m}(\widehat{x}))$ muni de la distance $d_m$ définie par: 

$$d_m(\mathrm{graphe}\ \phi,\mathrm{graphe}\ \psi)=\max_{t\in \overline{B_{k_1}(0,e^{\gamma}hr(\widehat{f}^{m}(\widehat{x})))}} \| \phi(t)-\psi(t)\|.$$

Les $B_j^{i}$ $(i=0,\cdots,m)$ sont dans $G_h^{i}$. 

On a pour $i=0,\cdots,m-1$ et $j\geq 0$:

$$d_{i+1}(B_{j}^{i+1},B_{j+1}^{i+1})\leq e^{-2\gamma}\mathrm{d}_{i}(B_{j}^{i},B_{j+1}^{i}).$$

En effet $g_{\widehat{f}^{i}(\widehat{x})}(X,Y)=(AX+R(X,Y),BY+U(X,Y))$ avec $\| A^{-1}\|^{-1}\geq e^{\chi_{m_0}-\gamma}$, $\| B\|\leq e^{\chi_{m_0 + 1}+\gamma}$ et $\max(\| DR(X,Y)\|,\| DU(X,Y)\|)\leq 5h$ pour $\| (X,Y)\|\leq 5hr(\widehat{f}^{i}(\widehat{x}))$.

Donc si on note $(X,\phi_1(X))$ le graphe de $B_j^{i}$, $(X,\psi_1(X))$ le graphe image par $g_{\widehat{f}^{i}(\widehat{x})}$ qui contient $B_{j}^{i+1}$, $(X,\phi_2(X))$ le graphe de $B_{j+1}^{i}$ et $(X,\psi_2(X))$ le graphe image par $g_{\widehat{f}^{i}(\widehat{x})}$ qui contient $B_{j+1}^{i+1}$, on a: 

\begin{equation*}
\begin{split}
& \| \psi_1(AX+R(X,\phi_1(X)))-\psi_2(AX+R(X,\phi_1(X)))\|\\
& \leq  \| \psi_1(AX+R(X,\phi_1(X)))-\psi_2(AX+R(X,\phi_2(X)))\|\\
& +\| \psi_2(AX+R(X,\phi_2(X)))-\psi_2(AX+R(X,\phi_1(X)))\|.\\
\end{split}
\end{equation*}

Mais $g_{\widehat{f}^{i}(\widehat{x})}(X,\phi_1(X))=(AX+R(X,\phi_1(X)), B\phi_1(X)+U(X,\phi_1(X)))$

est un point de $(X,\psi_1(X))$ (et de même en remplaçant $\phi_1$ par $\phi_2$ et $\psi_1$ par $\psi_2$), d'où: 

\begin{equation*}
\begin{split}
& \| \psi_1(AX+R(X,\phi_1(X)))-\psi_2(AX+R(X,\phi_1(X)))\|\\
& \leq \| B\phi_1(X)+U(X,\phi_1(X))-B\phi_2(X)-U(X,\phi_2(X))\|\\
& +\gamma_0\| R(X,\phi_2(X))-R(X,\phi_1(X))\| \\
& \leq \| B\| \| \phi_1(X)-\phi_2(X)\|+5h\| \phi_1(X)-\phi_2(X)\|+\gamma_0 5h \| \phi_1(X)-\phi_2(X)\| \\
& \leq (e^{\chi_{m_0 +1}+\gamma}+5h(1+\gamma_0))\| \phi_1(X)-\phi_2(X)\| \leq e^{-2\gamma}\| \phi_1(X)-\phi_2(X)\|. \\
\end{split}
\end{equation*}

$B_j^{i+1}$ et $B_{j+1}^{i+1}$ sont des sous-parties de $(X,\psi_1(X))$ et $(X,\psi_2(X))$ d'où: 

$$\forall i=0,\cdots,m-1,\forall j\geq 0,\ d_{i+1}(B_j^{i+1},B_{j+1}^{i+1})\leq e^{-2\gamma}d_{i}(B_j^{i},B_{j+1}^{i}). 
$$

On obtient ainsi:

$$d_{m}(B_j^{m},B_{j+1}^{m})\leq e^{-2\gamma m}d_0(B_j^0,B_{j+1}^0)=e^{-2\gamma m}d(B_j,B_{j+1}).$$

Il faut maintenant passer de $B_j^m$ à  $B_{j+1}$. 

On a $B_{j+1}$ qui est le cut-off de $C_{\gamma}^{-1}(\widehat{x})\circ \tau_x^{-1}  \circ \tau_{f^m(x)}\circ C_{\gamma}(\widehat{f}^m(\widehat{x}))(B_j^m)$. 

On reprend les notations utilisées à  savoir $C=C_{\gamma}^{-1}(\widehat{x})\circ \tau_x^{-1} \circ \tau_{f^m(x)}\circ C_{\gamma}(\widehat{f}^m(\widehat{x}))$ qui s'écrit $C(w)=g_1(w)+C_{\gamma}^{-1}(\widehat{x})a_m$ avec $g_1(w)=C_{\gamma}^{-1}(\widehat{x})C_{\gamma}(\widehat{f}^m(\widehat{x}))w$.

On a $g_1(X,Y)=(AX+CY,BY+DX)$ avec $\| A^{-1}\|^{-1}\geq 1-\epsilon(\eta)$, $\| B\| \leq 1+\epsilon(\eta)$, $\| C\|,\| D\|\leq \epsilon(\eta)$. En notant $C_{\gamma}^{-1}(\widehat{x})a_m=(t_1,t_2)$, on peut écrire,

$$C(X,Y)=(t_1+AX+CY,t_2+BY+DX).$$

Si on reprend le calcul précédent avec $C$ à  la place de $g_{\widehat{f}^{i}(\widehat{x})}$, $(X,\phi_1(X))$ pour $B_j^m$, $(X,\phi_2(X))$ pour $B_{j+1}^m$ et $(X,\psi_1(X)),(X,\psi_2(X))$ les images respectives de ces graphes par C, on a: 

\begin{equation*}
\begin{split}
& \| \psi_1(t_1+AX+C\phi_1(X))-\psi_2(t_1+AX+C\phi_1(X))\| \\
& \leq \| \psi_1(t_1+AX+C\phi_1(X))-\psi_2(t_1+AX+C\phi_2(X))\| \\
& + \|\psi_2(t_1+AX+C\phi_2(X))-\psi_2(t_1+AX+C\phi_1(X))\| \\
& \leq  \| B\phi_1(X)+DX-DX-B \phi_2(X)+t_2-t_2\|+\gamma_0\| C \| \| \phi_1(X)-\phi_2(X)\| \\
& \leq (\| B\| +\gamma_0 \| C\|)\| \phi_1(X)-\phi_2(X)\| \leq  (1+\epsilon(\eta)+\gamma_0\epsilon(\eta))\| \phi_1(X)-\phi_2(X)\| \\
& \leq  e^{\gamma}\| \phi_1(X)-\phi_2(X)\| \mbox{   } \mbox{  pour  } \eta<<\gamma,\gamma_0. \\
\end{split}
\end{equation*}

Comme $B_{j+1}$ et $B_{j+2}$ sont inclus dans $(X,\psi_1(X))$ et $(X,\psi_2(X))$, on obtient: 

$$d(B_{j+1},B_{j+2})\leq e^{\gamma}d_m(B_{j}^m,B_{j+1}^m)\leq e^{\gamma}e^{-2\gamma m}d(B_{j},B_{j+1})\leq e^{-\gamma}d(B_{j},B_{j+1}).$$

C'est ce qu'on voulait démontrer. La suite $(B_j)_j$ converge donc vers $B\in G_h$. 

Remarquons aussi que de la même façon, les $B_j^{i}$ convergent quand $j\rightarrow +\infty$ vers $B^{i} \in G_h^{i}$ pour $i=0,\cdots,m$ (cela découle des inégalités que l'on vient de montrer). 

\end{proof}

Passons maintenant à  la variété stable. 

Dans le repère $C_{\gamma}^{-1}(\widehat{x})E_u(\widehat{x})\bigoplus C_{\gamma}^{-1}(\widehat{x})E_s(\widehat{x})$, on note $G_v$ l'ensemble des graphes de fonctions holomorphes $(\phi(Y),Y)$ au-dessus de $\overline{B_{k_2}(0,he^{\gamma}r(\widehat{x}))}\subset C_{\gamma}^{-1}(\widehat{x})E_s(\widehat{x})$ pour lesquels $\mathrm{Lip}( \phi) \leq \gamma_0$  et $\| \phi(0)\| \leq hr(\widehat{x})$. Les $A_l^{0}$ sont dans $G_v$. 

On munit $G_v$ de la métrique: 

$$\mathrm{d}(\mathrm{graphe} \mbox{    } \phi,\mathrm{graphe} \mbox{    } \psi)=\max_{t\in \overline{B_{k_2}(0,he^{\gamma}r(\widehat{x}))}} \| \phi(t)-\psi(t)\|.$$

Comme précédemment, par le théorème d'Ascoli, $(G_v,d)$ est compact. 

On a alors: 

\begin{Lem}{\label{lemme7}}
$(A_l^0)_l$ est une suite de Cauchy de $(G_v,d)$. En particulier il existe $A^0 \in G_v$ avec $A_l^0 \underset{l\to +\infty}{\longrightarrow} A^0$. 
\end{Lem}

\begin{proof}

On va montrer que: 

$$\forall l\geq 1,\ d(A_l^{0},A_{l+1}^0)\leq e^{- \gamma} d(A_{l-1}^0,A_{l}^0)$$

et cela impliquera le résultat comme au lemme précédent. 

Dans le repère $C_{\gamma}^{-1}(\widehat{f}^{i}(\widehat{x}))E_u(\widehat{f}^{i}(\widehat{x}))\bigoplus C_{\gamma}^{-1}(\widehat{f}^{i}(\widehat{x}))E_s(\widehat{f}^{i}(\widehat{x}))$, on considère, pour $i=1,\cdots,m$, l'ensemble $G_v^i$ des graphes de fonctions holomorphes $(\phi(Y),Y)$ au-dessus de $\overline{B_{k_2}(0,hr(\widehat{f}^{i}(\widehat{x}))}$ avec $\mathrm{Lip}( \phi) \leq \gamma_0$ et  $\| \phi(0)\| \leq hr(\widehat{f}^{i}(\widehat{x}))$ muni de la distance $d_i$ définie par: 

$$d_i(\mathrm{graphe} \mbox{    } \phi,\mathrm{graphe} \mbox{    } \psi)=\max_{t\in \overline{B_{k_2}(0,hr(\widehat{f}^{i}(\widehat{x})))}} \| \phi(t)-\psi(t)\|.$$

Les $A_l^{i}$ sont dans $G_v^{i}$ par construction pour $i=1,\cdots,m$. De plus, on a (en notant $d_0=d$)

$$\forall i=0,\cdots,m-1 \mbox{  et  } \forall l \geq 0,\ d_i(A_{l}^{i},A_{l+1}^{i})\leq e^{-2\gamma}\mathrm{d}_{i+1}(A_{l}^{i+1},A_{l+1}^{i+1}).$$

En effet comme dans le lemme précédent $g_{\widehat{f}^{i}(\widehat{x})}(X,Y)=(AX+R(X,Y),BY+U(X,Y))$ avec $\| A^{-1}\|^{-1}\geq e^{\chi_{m_0}-\gamma}$, $\| B\|\leq e^{\chi_{m_0 + 1}+\gamma}$ et $\max(\| DR(X,Y)\|,\| DU(X,Y)\|)\leq 5h$ pour $\| (X,Y)\|\leq 5hr(\widehat{f}^{i}(\widehat{x}))$.

Donc si on note $(\phi_1(Y),Y)$ le graphe de $A_l^{i}$, $(\psi_1(Y),Y)$ son graphe image par $g_{\widehat{f}^{i}(\widehat{x})}$ qui est inclus dans $A_{l}^{i+1}$, $(\phi_2(Y),Y)$ le graphe de $A_{l+1}^{i}$ et $(\psi_2(Y),Y)$ son graphe image par $g_{\widehat{f}^{i}(\widehat{x})}$ qui est inclus dans $A_{l+1}^{i+1}$, on a: 

\begin{equation*}
\begin{split}
& \| \psi_1(BY+U(\phi_1(Y),Y))-\psi_2(BY+U(\phi_1(Y),Y))\|\\
& \geq \| \psi_1(BY+U(\phi_1(Y),Y))-\psi_2(BY+U(\phi_2(Y),Y))\|\\
& -\| \psi_2(BY+U(\phi_2(Y),Y))-\psi_2(BY+U(\phi_1(Y),Y))\|. \\
\end{split}
\end{equation*}

$g_{\widehat{f}^{i}(\widehat{x})}(\phi_1(Y),Y)$ est un point de $(\psi_1(Y),Y)$ (et de même en remplaçant $\phi_1$ par $\phi_2$ et $\psi_1$ par $\psi_2$) d'où la quantité précédente est supérieure à

\begin{equation*}
\begin{split}
& \| A\phi_1(Y)+R(\phi_1(Y),Y)-A\phi_2(Y)-R(\phi_2(Y),Y)\| -\gamma_0 5h\| \phi_1(Y)-\phi_2(Y))\| \\
& \geq  \| A( \phi_1(Y)-\phi_2(Y))\|-(1+\gamma_0)5h\| \phi_1(Y)-\phi_2(Y)\| \\
& \geq  (e^{\chi_{m_0}-\gamma}-5h(1+\gamma_0))\| \phi_1(Y)-\phi_2(Y)\|\\
& \geq e^{2\gamma}\| \phi_1(Y)-\phi_2(Y)\| \mbox{    } \mbox{  car  }  (e^{\chi_{m_0}-\gamma}-5h(1+\gamma_0))\geq e^{2\gamma}. \\
\end{split}
\end{equation*}

L'image de $A_{l}^{i}$ par $g_{\widehat{f}^{i}(\widehat{x})}$ est incluse dans $A_{l}^{i+1}$ d'où: 

$$\forall i=0,\cdots,m-1 \mbox{  et  }  \forall l \geq 0,\ d_{i+1}(A_l^{i+1},A_{l+1}^{i+1})\geq e^{2 \gamma} d_{i}(A_l^{i},A_{l+1}^{i}).$$

On obtient ainsi: 

$$d_{0}(A_l^{0},A_{l+1}^{0})\leq e^{-2\gamma m}d_m(A_l^m,A_{l+1}^m).$$

Si $C^{-1}=C_{\gamma}^{-1}(\widehat{f}^m(\widehat{x}))\circ \tau_{f^m(x)}^{-1} \circ \tau_{x}\circ C_{\gamma}(\widehat{x})$, on a $C^{-1}(A_{l-1}^0)\supset A_l=A_l^m$ (car on fait un cut-off) d'où: 

\begin{equation*}
\begin{split}
 d_0(A_l^{0},A_{l+1}^0) & =d(A_{l}^0,A_{l+1}^0) \leq e^{-2\gamma m}d_m(A_l^m, A_{l+1}^m)\\
& \leq e^{-2\gamma m}d_m(C^{-1}(A_{l-1}^0),C^{-1}(A_{l}^0)) \\
& \leq e^{\gamma}e^{-2\gamma m} d(A_{l-1}^0,A_l^{0}) \mbox{  (exactement comme au lemme précédent)} \\
& \leq e^{-\gamma}d(A_{l-1}^0,A_l^{0}).\\
\end{split}
\end{equation*}

C'est ce qu'on voulait démontrer. 

\end{proof}

Remarquons aussi que les $A_l^{i}$ convergent quand $l \rightarrow +\infty$ vers $A^{i}$ dans $G_v^{i}$ pour $i=0,\cdots,m$ par les inégalités que l'on a prouvées. 

Rappelons que le point $z_{l,l+1}$ vérifie $C_{\gamma}^{-1}(\widehat{x})\tau_x^{-1}(z_{l,l+1})\in B_{l+1}\cap A_{l}^0$. En particulier quand $l\rightarrow +\infty$ on a $C_{\gamma}^{-1}(\widehat{x})\tau_x^{-1}(z)\in B\cap A^0$.

$B$ et $A^0$ sont des graphes de fonctions holomorphes comme limites uniformes de fonctions holomorphes. 

On va maintenant utiliser ces deux ensembles pour montrer que $z$ est hyperbolique. 

Cela va se faire en deux étapes: tout d'abord nous construisons des espaces invariants $E_u(z)$ et $E_s(z)$ puis dans une deuxième étape nous estimerons les valeurs propres de $Df^m(z)$ dessus. 

{\bf{Construction de sous-espaces invariants:}}

On a $g_{\widehat{f}^{m-1}(\widehat{x})}\circ \cdots\circ g_{\widehat{x}}\circ C_{\gamma}^{-1}(\widehat{x})\tau_x^{-1}(\tau_x\circ C_{\gamma}(\widehat{x})(A_l^0)) \subset A_l^m$ et 

$A_l^m \subset  C_{\gamma}^{-1}(\widehat{f}^m(\widehat{x}))\circ \tau_{f^m(x)}^{-1}\circ \tau_x \circ C_{\gamma}(\widehat{x}) (A_{l-1}^0) $ ce qui implique: 

$$\tau_{f^m(x)}\circ C_{\gamma}(\widehat{f}^m(\widehat{x}))\circ g_{\widehat{f}^{m-1}(\widehat{x})}\circ \cdots \circ g_{\widehat{x}} \circ C_{\gamma}^{-1}(\widehat{x})\tau_x^{-1}(\tau_x\circ C_{\gamma}(\widehat{x})(A_l^0)) \subset \tau_x\circ C_{\gamma}(\widehat{x})(A_{l-1}^0).$$

 En passant à  la limite sur $l$, on obtient $f^m(\tau_x\circ C_{\gamma}(\widehat{x})(A^0)) \subset \tau_x\circ C_{\gamma}(\widehat{x})(A^0)$ (et $z=f^m(z)\in f^m(\tau_x\circ C_{\gamma}(\widehat{x})(A^0))$).

En particulier si on note $E_s(z)$ l'espace tangent à  $\tau_x\circ C_{\gamma}(\widehat{x})(A^0)$ en $z$, on a: 

$$Df^m(z)(E_s(z))\subset E_s(z).$$

On verra que si $0$ est valeur propre de $Df^m(z)$, l'inclusion $Df^m(z)(E_s(z))\subset E_s(z)$ peut être stricte. 

Faisons un raisonnement analogue pour l'espace $E_u(z)$. 

Dans la suite, on note $\Delta_0=\overline{B_{k_1}(0,hr(\widehat{x}))}\times \overline{B_{k_2}(0,e^{\gamma} h r(\widehat{x}))}$, pour $i=1,\cdots,m-1$, $\Delta_i=\overline{B_{k_1}(0, hr(\widehat{f}^{i}(\widehat{x})))}\times \overline{B_{k_2}(0,hr(\widehat{f}^{i}(\widehat{x})))}$ et $\Delta_m=\overline{B_{k_1}(0,e^{\gamma} hr(\widehat{f}^{m}(\widehat{x})))}\times \overline{B_{k_2}(0,h r(\widehat{f}^{m}(\widehat{x})))}$.

En passant à la limite sur les $B_j$ on a: 

$$C_{\gamma}^{-1}(\widehat{x})\tau_x^{-1} \tau_{f^m(x)}C_{\gamma}(\widehat{f}^m(\widehat{x}))(g_{\widehat{f}^{m-1}(\widehat{x})}(\cdots(g_{\widehat{f}(\widehat{x})}(g_{\widehat{x}}(B)\cap\Delta_1)\cap \Delta_2)\cdots)\cap\Delta_m)\cap \Delta_0=B$$

(cut-off successifs).

Le point $C_{\gamma}^{-1}(\widehat{x}) \tau_x^{-1}(z)$ est dans $B$, ensuite $g_{\widehat{f}^{i}(\widehat{x})}\circ \cdots\circ g_{\widehat{x}}(C_{\gamma}^{-1}(\widehat{x})\tau_x^{-1}(z))$ est dans l'intérieur de $\Delta_{i+1}$ pour $i=0,\cdots,m-1$ et $C_{\gamma}^{-1}(\widehat{x})\tau_{x}^{-1}\tau_{f^m(x)}C_{\gamma}(\widehat{f}^m(\widehat{x}))g_{\widehat{f}^{m-1}(\widehat{x})}\circ \cdots \circ g_{\widehat{x}}(C_{\gamma}^{-1}(\widehat{x}) \tau_x^{-1}(z))$ est dans l'intérieur de $\Delta_{0}$. 

On peut donc trouver un voisinage $U$ de $C_{\gamma}^{-1}(\widehat{x})\tau_x^{-1}(z)$ tel que $g_{\widehat{f}^{i}(\widehat{x})}\circ \cdots\circ g_{\widehat{x}}(U)\subset \Delta_{i+1} \mbox{  pour  } i=0,\cdots,m-1$ et $C_{\gamma}^{-1}(\widehat{x})\tau_x^{-1}\tau_{f^m(x)}C_{\gamma}(\widehat{f}^{m}(\widehat{x}))g_{\widehat{f}^{m-1}(\widehat{x})}\circ \cdots\circ g_{\widehat{x}}(U)\subset \Delta_0$.

En particulier:

$$C_{\gamma}^{-1}(\widehat{x})\tau_x^{-1}\tau_{f^m(x)}C_{\gamma}(\widehat{f}^{m}(\widehat{x}))g_{\widehat{f}^{m-1}(\widehat{x})}\circ \cdots\circ g_{\widehat{x}}(U\cap B)\subset B$$

ce qui implique que $f^m(\tau_x C_{\gamma}(\widehat{x})(U\cap B))\subset \tau_x C_{\gamma}(\widehat{x})(B)$. 

Si on note $E_u(z)$ l'espace tangent à  $z$ en $\tau_x C_{\gamma}(\widehat{x})(B)$, on a  comme précédemment:

$$Df^m(z)E_u(z)\subset E_u(z).$$

\begin{Rem}
On verra qu'ici on a une égalité, c'est-à-dire que $Df^m(z)$ est un isomorphisme de $E_u(z)$ dans lui-même. 
\end{Rem}

{\bf{Estimation des valeurs propres de $Df^m(z)$:}}

Soit $v$ un vecteur tangent à $A^0$ au point $ C_{\gamma}^{-1}(\widehat{x})\tau_x^{-1}(z)$ avec $v\neq 0$. 

On note $g=g_{\widehat{f}^{m-1}(\widehat{x})}\circ \cdots \circ g_{\widehat{x}}$ et on va évaluer $\|Dg(C_{\gamma}^{-1}(\widehat{x})\tau_x^{-1}(z))v\|$. 

Supposons dans un premier temps que cette quantité est non nulle. On a: 

\begin{equation*}
\begin{split}
& \|Dg(C_{\gamma}^{-1}(\widehat{x})\tau_x^{-1}(z))v\| \\
&=\| Dg_{\widehat{f}^{m-1}(\widehat{x})}(g_{\widehat{f}^{m-2}(\widehat{x})}\circ \cdots \circ g_{\widehat{x}} C_{\gamma}^{-1}(\widehat{x}) \tau_{x}^{-1}(z)) D(g_{\widehat{f}^{m-2}(\widehat{x})}\circ \cdots \circ g_{\widehat{x}})(C_{\gamma}^{-1}(\widehat{x}) \tau_x^{-1}(z))(v) \|\\
&=\frac{\| Dg_{\widehat{f}^{m-1}(\widehat{x})}(g_{\widehat{f}^{m-2}(\widehat{x})}\circ \cdots \circ g_{\widehat{x}} C_{\gamma}^{-1}(\widehat{x})  \tau_{x}^{-1}(z)) D(g_{\widehat{f}^{m-2}(\widehat{x})}\circ \cdots \circ g_{\widehat{x}})(C_{\gamma}^{-1}(\widehat{x}) \tau_x^{-1}(z))(v) \|}{\| D(g_{\widehat{f}^{m-2}(\widehat{x})}\circ \cdots \circ g_{\widehat{x}})(C_{\gamma}^{-1}(\widehat{x})\circ \tau_x^{-1}(z))(v) \|} \\
 & \times \| D(g_{\widehat{f}^{m-2}(\widehat{x})}\circ \cdots \circ g_{\widehat{x}})(C_{\gamma}^{-1}(\widehat{x})\circ \tau_x^{-1}(z))(v) \|, \\
\end{split}
\end{equation*}

puis on recommence ce que l'on vient de faire avec $g_{\widehat{f}^{m-2}(\widehat{x})} \circ \cdots \circ g_{\widehat{x}}$ au lieu de $g_{\widehat{f}^{m-1}(\widehat{x})} \circ \cdots \circ g_{\widehat{x}}$. Après $m$ fois on a: 

\begin{equation}{\label{equation1}}
\| Dg(C_{\gamma}^{-1}(\widehat{x})\tau_x^{-1}(z))v\|=\frac{\| Dg_{\widehat{f}^{m-1}(\widehat{x})}(Z_{m-1})(U_{m-1})\| }{\| U_{m-1}\|}\times \cdots \times \frac{\| Dg_{\widehat{x}}(Z_0)(U_0)\| }{\| U_0\|} \| U_0\|
\end{equation}

avec $Z_i=g_{\widehat{f}^{i-1}(\widehat{x})}\circ \cdots\circ g_{\widehat{x}} C_{\gamma}^{-1}(\widehat{x})\tau_{x}^{-1}(z)$ et $U_i=D(g_{\widehat{f}^{i-1}(\widehat{x})}\circ \cdots\circ g_{\widehat{x}})(C_{\gamma}^{-1}(\widehat{x})\tau_{x}^{-1}(z))(v)$ pour $i=0,\cdots,m-1$.

Maintenant on a: 

$$g_{\widehat{f}^{i-1}(\widehat{x})}\circ \cdots\circ g_{\widehat{x}}(A^0)\subset A^{i}=\lim_{l\to +\infty} A_{l}^{i} \mbox{    } \mbox{  par construction }.$$

Cela signifie que $U_i$ est tangent à  $A^{i}$ au point $Z_i$. 

$A^{i}$ est un graphe de la forme $(\phi^{i}(Y),Y)$ avec $\| Y\|\leq hr(\widehat{f}^{i}(\widehat{x}))$ pour $i=1,\cdots,m$ et $\| Y\|\leq e^{\gamma} hr(\widehat{x})$ pour $i=0$. On a aussi
$\mathrm{Lip}( \phi^{i}) \leq \gamma_0$  pour $i=0,\cdots,m$. Le vecteur $U_i$ est donc de la forme $U_i=(D\phi^{i}(Y)\alpha_i,\alpha_i)$. 

On a alors $\| Dg_{\widehat{f}^{i}(\widehat{x})}(Z_i)(U_i)\|\leq e^{-2\gamma}\| U_i\|$. En effet, pour $\| (X,Y)\|\leq 5hr(\widehat{f}^{i}(\widehat{x}))$:

$$Dg_{\widehat{f}^{i}(\widehat{x})}(X,Y)(w_1,w_2)=(Aw_1+DR(X,Y)(w_1,w_2),Bw_2+DU(X,Y)(w_1,w_2))$$

d'où: 

$$\| Dg_{\widehat{f}^{i}(\widehat{x})}(Z_i)(w_1,w_2)\| \leq \| A\| \| w_1\| +10h\| (w_1,w_2)\|+\| B\| \| w_2\| $$

avec $(w_1,w_2)=U_i$. Mais $\| w_1\|=\| D\phi^{i}(Y)\alpha_i\|\leq \gamma_0 \| \alpha_i\|=\gamma_0 \| w_2\|$ et donc:

\begin{equation*}
\begin{split}
& \| Dg_{\widehat{f}^{i}(\widehat{x})}(Z_i)(w_1,w_2)\| \leq (\gamma_0\| A\|+10h(1+\gamma_0)+\| B\|)\| w_2\| \\
& \leq (\gamma_0e^{\chi_1+\gamma}+10h(1+\gamma_0)+e^{\chi_{m_0+1}+\gamma})\| w_2\| \leq e^{-3\gamma}\| w_2\| \\
\end{split}
\end{equation*}

(voir la démonstration du lemme \ref{lemme5} pour la dernière inégalité). 

Cette dernière quantité est bien plus petite que $e^{-2\gamma}\| U_i\|$ car 

$$ \| U_i\| \geq \| w_2\| -\| w_1\|\geq (1-\gamma_0)\| w_2\|\geq e^{-\gamma}\| w_2\|.$$

Si on reprend l'équation (\ref{equation1}), on obtient ainsi

$$\|Dg(C_{\gamma}^{-1}(\widehat{x})\tau_x^{-1}(z))v\|\leq e^{-2\gamma m}\| v\|.$$

C'est aussi vrai si le membre de gauche est nul. 

Considérons encore $C=C_{\gamma}^{-1}(\widehat{x})\tau_x^{-1}\tau_{f^m(x)}C_{\gamma}(\widehat{f}^{m}(\widehat{x}))$. On a: 

\begin{equation*}
\begin{split}
& \| D(C\circ g)(C_{\gamma}^{-1}(\widehat{x})\tau_x^{-1}(z))(v)\| \\
& \leq  \| DC(g\circ C_{\gamma}^{-1}(\widehat{x})\circ \tau_x^{-1}(z))\| \| Dg(C_{\gamma}^{-1}(\widehat{x})\tau_x^{-1}(z))(v)\| \\
& \leq \| C_{\gamma}^{-1}(\widehat{x})\circ C_{\gamma}(\widehat{f}^m(\widehat{x}))\| e^{-2\gamma m}\| v\|\\
& \leq e^{\gamma-2\gamma m}\| v\|  \mbox{    } \mbox{  car  } \eta<<\gamma\\
& \leq e^{-\gamma}\| v\|. \\
\end{split}
\end{equation*}

Maintenant on va itérer cette inégalité. Pour cela, on remarque que:

$$(C\circ g)(C_{\gamma}^{-1}(\widehat{x})\circ\tau_x^{-1}(z))=C_{\gamma}^{-1}(\widehat{x})\tau_x^{-1}\circ f^m(z)=C_{\gamma}^{-1}(\widehat{x})\circ\tau_x^{-1}(z)$$

et $(C\circ g)(A^0)\subset C(A^m)\subset A^0$. 

On peut donc recommencer et par une formule du même type que celle de l'équation (\ref{equation1}), on a pour $k\geq 1$: 

$$\| D(\underbrace{C\circ g \circ \cdots \circ C\circ g}_{k \mbox{  fois}})(C_{\gamma}^{-1}(\widehat{x})\circ\tau_x^{-1}(z))(v)\|   \leq  e^{-\gamma k}\| v\|.$$

Ensuite:

\begin{equation*}
\begin{split}
& D(C\circ g \circ \cdots \circ C\circ g)(C_{\gamma}^{-1}(\widehat{x})\circ\tau_x^{-1}(z))v \\
&= D(C_{\gamma}^{-1}(\widehat{x})\circ\tau_x^{-1}\circ f^{mk}\circ \tau_x\circ C_{\gamma}(\widehat{x}))(C_{\gamma}^{-1}(\widehat{x})\circ\tau_x^{-1}(z))v \\
&= D(C_{\gamma}^{-1}(\widehat{x})\circ\tau_x^{-1})(f^{mk}(z))\circ Df^{mk}(z)\circ D(\tau_x \circ C_{\gamma}(\widehat{x}))(C_{\gamma}^{-1}(\widehat{x})\circ\tau_x^{-1}(z))v,\\
\end{split}
\end{equation*}

d'où: 

\begin{equation*}
\begin{split}
& \| D(C\circ g \circ \cdots \circ C\circ g)(C_{\gamma}^{-1}(\widehat{x})\circ\tau_x^{-1}(z))v\| \\
& \geq \| D(C_{\gamma}^{-1}(\widehat{x})\circ\tau_x^{-1})(z)^{-1}\| ^{-1} \times \| Df^{mk}(z)\circ D(\tau_x \circ C_{\gamma}(\widehat{x}))(C_{\gamma}^{-1}(\widehat{x})\circ\tau_x^{-1}(z))v\|.\\
\end{split}
\end{equation*}

Soit $w\in E_s(z)$. Par définition de $E_s(z)$, $w$ s'écrit $w=D(\tau_x\circ C_{\gamma}(\widehat{x}))(C_{\gamma}^{-1}(\widehat{x})\tau_x^{-1}(z))(v)$ avec $v$ tangent à  $A^0$ en $C_{\gamma}^{-1}(\widehat{x})\tau_{x}^{-1}(z)$. D'où: 

$$\| Df^{mk}(z)w\|\leq \| D(C_{\gamma}^{-1}(\widehat{x})\tau_x^{-1})(z)^{-1}\|e^{-\gamma k}\| v\|\leq C(X,r_0)e^{-\gamma k}\| w\|.$$

Donc $Df^m(z)$ est contractante sur $E_s(z)$ et ses valeurs propres sont de module inférieur ou égal à  $e^{-\gamma}$. 

Passons à  $E_u(z)$. Soit $v$ un vecteur tangent à  $B$ au point $C_{\gamma}^{-1}(\widehat{x})\tau_x^{-1}(z)$ avec $v\neq 0$. Notons $Z_i=g_{\widehat{f}^{i-1}(\widehat{x})}\circ \cdots\circ g_{\widehat{x}}C_{\gamma}^{-1}(\widehat{x}) \tau_x^{-1}(z)$ et
$U_i=D(g_{\widehat{f}^{i-1}(\widehat{x})}\circ \cdots\circ g_{\widehat{x}})(C_{\gamma}^{-1}(\widehat{x}) \tau_x^{-1}(z))(v)$ pour $i=1,\cdots,m-1$.

On commence par estimer $\| Dg_{\widehat{f}^{i}(\widehat{x})}(Z_i)U_i\|$. 

Si $U$ est un petit voisinage de $C_{\gamma}^{-1}(\widehat{x}) \tau_x^{-1}(z)$ on a: 

$$g_{\widehat{f}^{i-1}(\widehat{x})}\circ \cdots\circ g_{\widehat{x}}(U\cap B)\subset B^{i}=\lim_{j\rightarrow +\infty} B_j^{i}\ \mbox{  pour  }i=0,\cdots,m.$$

En particulier $U_i$ est tangent à  $B^{i}$ en $Z_i$. 

Comme $B^{i}$ est un graphe du type $(X,\psi^{i}(X))$ avec $\| X \|\leq hr(\widehat{f}^{i}(\widehat{x}))$ pour $i=0,\cdots,m-1$, $\| X \|\leq e^{\gamma} hr(\widehat{f}^m(\widehat{x}))$ pour $i=m$ et que
$\mathrm{Lip}( \psi^{i}) \leq \gamma_0$ pour $i=0,\cdots,m$,  le vecteur $U_i$ est de la forme $U_i=(\alpha_i,D\psi^{i}(X)\alpha_i)$. 

On a alors: 

$$\| Dg_{\widehat{f}^{i}(\widehat{x})}(Z_i)U_i\|\geq e^{2\gamma}\| U_i\| .$$

En effet: 

$$Dg_{\widehat{f}^{i}(\widehat{x})}(X,Y)(w_1,w_2)=(Aw_1+DR(X,Y)(w_1,w_2),Bw_2+DU(X,Y)(w_1,w_2)),$$

d'où: 

\begin{equation*}
\begin{split}
& \| Dg_{\widehat{f}^{i}(\widehat{x})}(Z_i)(w_1,w_2)\| \\
& \geq \| Aw_1\| - \| DR(Z_i)(w_1,w_2)\|-\| B\| \| w_2\| - \| DU(Z_i)(w_1,w_2)\|  \\
& \geq \| A^{-1}\| ^{-1}\| w_1\| -5h\| (w_1,w_2)\| -\| B\| \| w_2\| -5h\| (w_1,w_2)\|.\\
\end{split}
\end{equation*}

Or $U_i=(\alpha_i,D\psi^{i}(X)\alpha_i)=(w_1,w_2)$ donc $\| w_2\| \leq \gamma_0 \| \alpha_i\| =\gamma_0 \| w_1\|$, et ainsi

\begin{equation*}
\begin{split}
& \| Dg_{\widehat{f}^{i}(\widehat{x})}(Z_i)(w_1,w_2)\| \\
& \geq e^{\chi_{m_0}-\gamma}\| w_1\| -5h(1+\gamma_0)\| w_1\|-e^{\chi_{m_0+1}+\gamma}\gamma_0 \| w_1\| -5h(1+\gamma_0)\| w_1\| \\
& =(e^{\chi_{m_0}-\gamma}-10h(1+\gamma_0)-\gamma_0 e^{\chi_{m_0+1}+\gamma})\| w_1\|.\\
\end{split}
\end{equation*}

Comme $\| (w_1,w_2)\| \leq (1+\gamma_0)\| w_1\|$ on obtient:

$$\| Dg_{\widehat{f}^{i}(\widehat{x})}(Z_i)U_i\|\geq \frac{e^{\chi_{m_0}-\gamma}-10h(1+\gamma_0)-\gamma_0e^{\chi_{m_0+1}+\gamma}}{1+\gamma_0}\times \| U_i\|\geq e^{2\gamma}\| U_i\|$$

(voir la preuve du lemme \ref{lemme4} pour la dernière inégalité).

Remarquons que cela implique aussi que $Dg_{\widehat{f}^{i}(\widehat{x})}(Z_i)(U_i)=U_{i+1}$ est non nul si $U_i\neq 0$. 

En particulier, comme $U_0=v\neq 0$, on a $U_i\neq 0$ pour $i=0,\cdots,m-1$. 

Maintenant, si on utilise la formule de l'équation (\ref{equation1}), on a: 

$$\| Dg(C_{\gamma}^{-1}(\widehat{x})\tau_x^{-1}(z))v\| \geq e^{2\gamma m}\| v\|,$$

et si on compose par $C=C_{\gamma}^{-1}(\widehat{x})\tau_x^{-1}\tau_{f^m(x)}C_{\gamma}(\widehat{f}^m(\widehat{x}))$, on obtient

\begin{equation*}
\begin{split}
& \| D(C\circ g)(C_{\gamma}^{-1}(\widehat{x})\tau_x^{-1}(z))v\| \\
&= \| DC(g\circ C_{\gamma}^{-1}(\widehat{x})\tau_x^{-1}(z))\circ Dg(C_{\gamma}^{-1}(\widehat{x})\tau_x^{-1}(z))v\|\\
& \geq  \| DC(g\circ C_{\gamma}^{-1}(\widehat{x})\tau_x^{-1}(z))^{-1}\|^{-1}e^{2\gamma m}\| v\| \\
&= \| C_{\gamma}^{-1}(\widehat{f}^m(\widehat{x}))\circ C_{\gamma}(\widehat{x})\|^{-1}e^{2\gamma m}\| v\| \geq  e^{-\gamma}e^{2\gamma m}\| v\|  \mbox{    } \mbox{  si  } \eta<<\gamma \\
& \geq  e^{\gamma}\| v\|.\\
\end{split}
\end{equation*}

On va itérer cette inégalité. 

Pour cela, on remarque $C\circ g(C_{\gamma}^{-1}(\widehat{x})\tau_x^{-1}(z))=C_{\gamma}^{-1}(\widehat{x})\tau_x^{-1}(z)$ et si $W_1$ est un voisinage suffisamment petit de $C_{\gamma}^{-1}(\widehat{x})\tau_x^{-1}(z)$, on a $C\circ g(B\cap W_1)\subset B\cap U$.

On peut donc recommencer ce que l'on vient de faire. En itérant $k$ fois le procédé (quitte à  prendre $W_k$ à  chaque fois), on a par une formule du même type que celle de l'équation (\ref{equation1}): 

$$\| D(C\circ g\circ \cdots\circ C\circ g)(C_{\gamma}^{-1}(\widehat{x})\tau_x^{-1}(z))(v)\|\geq e^{\gamma k}\| v\|.$$

Ensuite comme précédemment: 

\begin{equation*}
\begin{split}
& D(C\circ g\circ \cdots\circ C\circ g)(C_{\gamma}^{-1}(\widehat{x})\tau_x^{-1}(z))(v) \\
&=D(C_{\gamma}^{-1}(\widehat{x})\tau_x^{-1})(z)\circ Df^{mk}(z)(D(\tau_x\circ C_{\gamma}(\widehat{x})) (C_{\gamma}^{-1}(\widehat{x})\tau_x^{-1}z)(v)), \\
\end{split}
\end{equation*}

d'où: 

\begin{equation*}
\begin{split}
& \| D(C\circ g\circ \cdots\circ C\circ g)(C_{\gamma}^{-1}(\widehat{x})\tau_x^{-1}(z))(v)\| \\
& \leq  \| D(C_{\gamma}^{-1}(\widehat{x})\tau_x^{-1})(z)\| \times \| Df^{mk}(z)(D(\tau_x\circ C_{\gamma}(\widehat{x})) (C_{\gamma}^{-1}(\widehat{x})\tau_x^{-1}z)(v)) \|. \\
\end{split}
\end{equation*}

Soit $w\in E_u(z)$. Par définition $w=D(\tau_x\circ C_{\gamma}(\widehat{x}))(C_{\gamma}^{-1}(\widehat{x})\tau_x^{-1}(z))(v)$ avec $v$ tangent à  $B$ en $C_{\gamma}^{-1}(\widehat{x})\tau_x^{-1}(z)$. Donc $e^{\gamma k}\| v\| \leq \| D(C_{\gamma}^{-1}(\widehat{x})\tau_x^{-1})(z)\| \| Df^{mk}(z)w\|$ ou encore $e^{\gamma k}\| w\|\leq C(X,r_0)\| Df^{mk}(z) w\|$. En particulier, $Df^m(z)$ est dilatante sur $E_u(z)$ et ses valeurs propres sont de module supérieur ou égal à $e^{\gamma}$. 

Cela termine la démonstration du Closing Lemma.

Faisons deux remarques supplémentaires: 
\begin{enumerate}
\item Tout d'abord comme $Df^m(z)_{| E_u(z)}$ est dilatante, $Df^m(z)$ est un isomorphisme de $E_u(z)$ dans lui-même. On a donc $Df^m(z)E_u(z)=E_u(z)$. 
\item L'inclusion $Df^m(z)E_s(z)\subset E_s(z)$ peut être stricte. En effet $Ker Df^m(z)\subset E_s(z)$ (pour le voir faire une décomposition de $v\in Ker Df^m(z)$ dans une base composée de vecteurs de $E_s(z)$ et $E_u(z)$). 
\end{enumerate}

\section{{\bf Approximation par des Bernoulli}}
\par

Le but de ce paragraphe est de démontrer le théorème \ref{theoreme2}. Nous allons pour cela suivre la méthode d'A. Katok et L. Mendoza (voir \cite{KH}). Comme dans le Closing Lemma, ce qui change c'est que $f$ n'est pas inversible et possède un ensemble d'indétermination.

On garde les $\gamma$, $\delta$ et $\epsilon_0$ du paragraphe \ref{closinglemma}. Soit $\rho>0$ et $\varphi_1,\cdots,\varphi_{k_0}\in \mathrm{C}^0(X,\mathbb{R})$. 

On fixe $0 < r < 1$ petit de sorte que $0<r<\rho$ et $\frac{h_{\mu}(f)}{1+r}-3r>h_{\mu}(f)-\rho$. 

On choisit ensuite $\epsilon>0$ suffisamment petit pour que $dist(x,y)<\epsilon$ implique $| \varphi_i(x)-\varphi_i(y)|<\frac{r}{2}$ pour $i=1,\cdots,k_0$. C'est possible car les fonctions $\varphi_i$ sont uniformément continues ($X$ est compact). 

Par le théorème de Brin-Katok, si on note

$$B_{m} (x, \epsilon)= \{ y \in X \mbox{  ,  } dist(f^{p}(x),f^{p}(y)) < \epsilon \mbox{  pour  } p=0, \cdots m-1 \},$$ 

on a

$$h_{\mu}(f)=\lim_{\epsilon\rightarrow 0} \liminf_{m \rightarrow + \infty} -\frac{1}{m} \log \mu (B_m(x,\epsilon)) \mbox{    } \mbox{  pour  }\mu-\mbox{presque tout  }x.$$

On pose: 

$$\Gamma_{\epsilon,m_0}=\{x, \forall m\geq m_0\ \mu (B_m(x,\epsilon))\leq e^{-h_{\mu}(f)m+rm}\}.$$

Si $\epsilon$ est assez petit, on a: 

$$1-\frac{\delta}{2}\leq \mu(\{x,\liminf_{m \rightarrow + \infty} -\frac{1}{m} \log  \mu (B_m(x,\epsilon)) \geq h_{\mu}(f)-\frac{r}{2}\}) \leq \mu (\cup_{m_0} \Gamma_{\epsilon,m_0})$$

donc $\mu (\Gamma_{\epsilon,m_0})\geq 1-\delta$ si $m_0$ est assez grand.

Dans la suite on considère le $\gamma_0$ (qui était petit par rapport à  $\gamma$), le $h$ et le $\eta$ du Closing Lemma ou de sa preuve. On mettra des conditions un peu plus fortes sur ces quantités que l'on aurait pu supposer dans la démonstration du Closing Lemma si on voulait. 

Voici maintenant le plan de la preuve: 

Dans un premier temps, on va construire $\widehat{y}\in\Lambda_{\delta}$ avec beaucoup de récurrence autour de lui. Ensuite nous utiliserons ce point pour fabriquer le codage et vérifier les trois points du théorème. 

\subsection{\bf{Construction du $\widehat{y}\in \Lambda_{\delta}$ avec de la récurrence:}}

Les $\{B(\widehat{y}, \eta),\widehat{y}\in \Lambda_{\delta}\}$ forment un recouvrement de $\Lambda_{\delta}$ qui est compact: on peut donc trouver un sous-recouvrement fini $\cup_{i=1}^{t} B(\widehat{y_{i}}, \eta)$. 

Soit $\xi$ une partition finie de l'extension naturelle $\widehat{X}$, plus fine que $(\Lambda_{\delta},\widehat{X}\setminus \Lambda_{\delta})$ et telle que le diamètre des éléments de $\xi$ soit plus petit que la constante de Lebesgue du recouvrement de $\Lambda_{\delta}$ par $\cup_{i=1}^{t} B(\widehat{y_{i}}, \eta)$. 

On note: 

\begin{equation*}
\begin{split}
\Lambda_{\delta,m_0'} & =\{ \widehat{x}\in \Lambda_{\delta}, \forall m\geq m_0' \mbox{  on a  } \widehat{f}^q(\widehat{x})\in \xi(\widehat{x}) \mbox{  pour un  } q\in [m,m(1+r)] \mbox{  et  }  \\
& \forall i=1,\cdots,k_0 \mbox{  on a  } \left| \frac{1}{m}\sum_{l=0}^{m-1}\varphi_i\circ \pi(\widehat{f}^l(\widehat{x}))-\int\varphi_i \mathrm{d} \mu \right|<\frac{r}{2} \}. \\
\end{split}
\end{equation*}

\begin{Lem}
Pour $m_0'$ assez grand, on a $\widehat{\mu}(\Lambda_{\delta,m_0'})\geq 1-2\delta$.
\end{Lem}

\begin{proof}

Les fonctions $\varphi_i \circ \pi$ sont continues sur $\widehat{X}$, donc par le théorème de Birkhoff, on a sur un ensemble de masse $1$ pour $\widehat{\mu}$: 

$$\frac{1}{m}\sum_{l=0}^{m-1} \varphi_i\circ\pi(\widehat{f}^l(\widehat{x})) \stackrel{m}{\longrightarrow} \int \varphi_i\circ \pi\mathrm{d}\widehat{\mu}(\widehat{x})=\int\varphi_i\mathrm{d}\mu.$$

Cela implique que

$$\bigcup_{m_1} \left\{ \widehat{x},\forall m\geq m_1 \mbox{  ,  } \forall i=1, \cdots, k_0 \mbox{  on a  } \left| \frac{1}{m}\sum_{l=0}^{m-1}\varphi_i\circ \pi(\widehat{f}^l(\widehat{x}))-\int\varphi_i\mathrm{d}\mu \right|<\frac{r}{2} \right\}$$

est de mesure $1$ pour $\widehat{\mu}$.

Si $m_1$ est assez grand on a donc: 

$$\widehat{\mu} \left( \left\{ \widehat{x},\forall m\geq m_1 \mbox{  ,  } \forall i=1, \cdots, k_0 \mbox{  on a  } \left| \frac{1}{m}\sum_{l=0}^{m-1}\varphi_i\circ \pi(\widehat{f}^l(\widehat{x}))-\int\varphi_i\mathrm{d}\mu \right|<\frac{r}{2} \right\} \right) \geq 1-\frac{\delta}{2}.$$

Soit maintenant $C\in \xi$ avec $C\subset \Lambda_{\delta}$. On suppose que $\widehat{\mu}(C)>0$. 

Si on applique le théorème de Birkhoff à  la fonction $\widehat{x}\to \chi_C(\widehat{x})$ (qui vaut $1$ sur $C$ et $0$ ailleurs), on a: 

$$\frac{1}{m}\sum_{l=0}^{m-1}\chi_C(\widehat{f}^l(\widehat{x}))\stackrel{m}{\longrightarrow} \int \chi_C(\widehat{x})\mathrm{d}\widehat{\mu}(\widehat{x})=\widehat{\mu}(C)$$

sur un ensemble de masse $1$ pour $\widehat{\mu}$. 

Cela signifie en particulier que: 

\begin{equation*}
\begin{split}
\bigcup_{m_2}\{\widehat{x},\forall m\geq m_2 \mbox{  on a  } & \frac{1}{m}\sum_{l=0}^{m-1}\chi_C(\widehat{f}^l(\widehat{x}))<\widehat{\mu}(C)(1+\frac{r}{3}) \mbox{  et}\\
& \frac{1}{[m(1+r)]+1}\sum_{l=0}^{[m(1+r)]}\chi_C(\widehat{f}^l(\widehat{x}))>\widehat{\mu}(C)(1-\frac{r}{3})\}
\end{split}
\end{equation*}

est de mesure $1$ pour $\widehat{\mu}$. 

Si $m_2=m_2(C)$ est assez grand, on a donc: 

\begin{equation*}
\begin{split}
& \widehat{\mu}(\Gamma(C))=\widehat{\mu}(\{\widehat{x} \in C,\forall m\geq m_2 \mbox{  on a  } \frac{1}{m}\sum_{l=0}^{m-1}\chi_C(\widehat{f}^l(\widehat{x})) < \widehat{\mu}(C)(1+\frac{r}{3}) \mbox{  et  } \\
& \frac{1}{[m(1+r)]+1}\sum_{l=0}^{[m(1+r)]}\chi_C(\widehat{f}^l(\widehat{x}))>\widehat{\mu}(C)(1-\frac{r}{3})\})\geq \widehat{\mu}(C)-\frac{\delta}{2\mathrm{card}\xi}.
\end{split}
\end{equation*}

Si $\widehat{x}$ est dans $\Gamma(C)$ et si $\widehat{f}^l(\widehat{x})\notin C$ pour $l=m,\cdots,[m(1+r)]$ (avec $m \geq m_2$), on a:

$$(1+r)(1-\frac{r}{3})\widehat{\mu}(C)<\frac{1}{m}\sum_{l=0}^{[m(1+r)]}\chi_C(\widehat{f}^l(\widehat{x}))=\frac{1}{m}\sum_{l=0}^{m-1}\chi_C(\widehat{f}^l(\widehat{x}))<\widehat{\mu}(C)(1+\frac{r}{3}),$$

d'où: 

$$(1+r)(1-\frac{r}{3})<1+\frac{r}{3}\Longrightarrow \frac{r}{3}<\frac{r^2}{3} \mbox{, ce qui est absurde.}$$

Cela signifie que pour $\widehat{x}\in \Gamma(C)$ on a $\forall m\geq m_2(C)$, $\widehat{x}\in C$ et il existe $q\in[m,m(1+r)]$ avec $\widehat{f}^q(\widehat{x})\in C=\xi(\widehat{x})$. Soit

$$m_2=\max_{C\in\xi,C \subset\Lambda_{\delta},\widehat{\mu}(C)>0}\ m_2(C).$$

L'ensemble $\{\widehat{x}\in \Lambda_{\delta}, \forall m\geq m_2 \mbox{  il existe  } q\in[m,m(1+r)] \mbox{  avec  }\widehat{f}^q(\widehat{x})\in \xi(\widehat{x})\}$ contient donc $\bigcup_{C\in\xi,C \subset \Lambda_{\delta},\widehat{\mu}(C)>0}\Gamma(C)$. On en déduit que: 

\begin{equation*}
\begin{split}
\widehat{\mu}(\{\widehat{x}\in \Lambda_{\delta}, \forall m\geq m_2 \mbox{  il existe  } & q\in[m,m(1+r)] \mbox{  avec  }\widehat{f}^q(\widehat{x})\in \xi(\widehat{x})\}) \\
& \geq \widehat{\mu}(\Lambda_{\delta})-\frac{\delta}{2}\geq 1-\frac{3\delta}{2}.
\end{split}
\end{equation*}

Ainsi si $m_0'=\max(m_1,m_2)$, on a: 

$$\widehat{\mu}(\Lambda_{\delta,m_0'})\geq 1-\frac{\delta}{2}-\frac{3\delta}{2}=1-2\delta$$

car $\Lambda_{\delta,m_0'}$ contient 

$$\left\{\widehat{x},\forall m\geq m_1 \mbox{  ,  } \forall i=1, \cdots, k_0 \mbox{  on a  } \left| \frac{1}{m}\sum_{l=0}^{m-1}\varphi_i\circ \pi(\widehat{f}^l(\widehat{x}))-\int\varphi_i\mathrm{d}\mu \right|<\frac{r}{2} \right\}$$ 

intersecté avec

$$\left\{ \widehat{x}\in \Lambda_{\delta}, \forall m\geq m_2 \mbox{  il existe  } q\in[m,m(1+r)] \mbox{  avec  }\widehat{f}^q(\widehat{x})\in \xi(\widehat{x}) \right\}.$$

Cela démontre le lemme.

\end{proof}

Maintenant on a $\mu(\pi(\Lambda_{\delta,m_0'})\cap \Gamma_{\epsilon,m_0})\geq 1-3\delta$. Pour $m\geq \max(m_0,m_0')$, soit $E_m$ un ensemble $(m,\epsilon)$-séparé de cardinal maximal dans $\pi(\Lambda_{\delta,m_0'})\cap \Gamma_{\epsilon,m_0}$. 

Si $x\in E_m$, on a $\mu(B_m(x,\epsilon))\leq e^{-h_{\mu}(f)m+rm}$. Le cardinal de $E_m$ est donc supérieur à $(1-3\delta)e^{h_{\mu}(f)m-rm}$.

Pour $q=m,\cdots,[(1+r)m]$, on pose $V_q=\{x\in E_m, \widehat{f}^q(\widehat{x})\in \xi(\widehat{x})\}$ où $\widehat{x}$ est un point de $\Lambda_{\delta,m_0'}$ qui se projette sur $x$. 

Soit $n$ qui maximise le cardinal de $V_q$. On a:

$$\mathrm{card}\ V_n \geq \frac{(1-3\delta)e^{h_{\mu}(f)m-rm}}{rm+1}\geq (1-3\delta)e^{h_{\mu}(f)m-2rm}\geq \frac{1}{2}e^{h_{\mu}(f)m-2rm}.$$

Soit $x\in V_n$. On a $\widehat{f}^n(\widehat{x})\in \xi(\widehat{x})$ et alors $\mbox{dist}(\widehat{x},\widehat{f}^n(\widehat{x})) \leq \mbox{diam}( \xi(\widehat{x}))$ qui est lui-même plus petit que la constante de Lebesgue du recouvrement de $\Lambda_{\delta}$ par $\cup_{i=1}^{t} B(\widehat{y_{i}}, \eta)$.

En particulier il existe $i\in\{1,\cdots,t\}$ tel que $\widehat{x}$ et $\widehat{f}^n(\widehat{x})$ sont dans $B(\widehat{y_i},\eta)$. 

Soit $i$ qui maximise le nombre de $x\in V_n$ qui vérifient cette propriété. On notera $\widehat{y}=\widehat{y_i}$, $y=\pi(\widehat{y})$, et $N$ le nombre de $x\in V_n$ avec $\widehat{x}$ et $\widehat{f}^n(\widehat{x})$ dans $B(\widehat{y},\eta)$. On a $N\geq \frac{e^{h_{\mu}(f)m-2rm}}{2t}$. 

C'est le point $\widehat{y}\in \Lambda_{\delta}$ que l'on voulait construire. 

Remarquons que $\frac{e^{h_{\mu}(f)m-2rm}}{2t}\geq e^{h_{\mu}(f)m-3rm}$ si on prend $m$ grand et que ce nombre est supérieur à  $ e^{h_{\mu}(f)\frac{n}{1+r}-3rn}$ car $m\leq n\leq (1+r)m$. 

On a donc $N\geq  e^{h_{\mu}(f)n-\rho n}$ grâce à  l'hypothèse faite sur $r$ au début. En particulier quitte à  réduire le nombre de points on prendra dans la suite $N= e^{h_{\mu}(f)n-\rho n}$ et nous noterons $x_1,\cdots,x_N$ les $x$ de $V_n$ tels que $\widehat{x}$ et $\widehat{f}^n(\widehat{x})$ sont dans $B(\widehat{y},\eta)$. 

Par construction $x_1,\cdots,x_N$ sont dans $\pi(\Lambda_{\delta,m_0'})$ et nous appellerons $\widehat{x_1},\cdots,\widehat{x_N}$ des points de $\Lambda_{\delta,m_0'}$ qui se projettent sur eux (et qui vérifient $\widehat{x_i},\widehat{f}^n(\widehat{x_i})$ dans $B(\widehat{y},\eta)$). 

Nous allons maintenant utiliser ces points pour construire le codage. 

\subsection{\bf{Construction du codage}}

Nous utiliserons le lemme suivant: 

\begin{Lem}{\label{lemme3.2}}
Pour $\eta>0$ petit, on a: 

Pour tout $\widehat{x},\widehat{y}\in\Lambda_{\delta}$ avec $\mathrm{dist}(\widehat{x},\widehat{y})<\eta$, si $(X,\psi(X))$ est un graphe au-dessus d'une partie de $C_{\gamma}^{-1}(\widehat{y})E_u(\widehat{y})$ (dans le repère $C_{\gamma}^{-1}(\widehat{y})E_u(\widehat{y})\oplus C_{\gamma}^{-1}(\widehat{y})E_s(\widehat{y})$) pour $\| X\| \leq e^{\frac{\gamma}{2}}r(\widehat{y})h$ (respectivement $\| X\| \leq e^{\gamma}r(\widehat{y})h$) avec $\| \psi(0)\|\leq e^{-\frac{\gamma}{2}}r(\widehat{y})h$ (respectivement $\| \psi(0)\|\leq e^{-\gamma}r(\widehat{y})h$) et $\mathrm{Lip}(\psi)\leq \gamma_0e^{-\frac{\gamma}{2}}$ (respectivement $\mathrm{Lip}(\psi)\leq \gamma_0e^{-\gamma}$) alors l'image de ce graphe par $C=C_{\gamma}^{-1}(\widehat{x})\tau_x^{-1}\tau_y C_{\gamma}(\widehat{y})$ est un graphe $(X,\phi(X))$ au-dessus d'une partie de $C_{\gamma}^{-1}(\widehat{x})E_u(\widehat{x})$ pour au moins $\| X\| \leq r(\widehat{x})h$ (respectivement $\| X\| \leq e^{\frac{\gamma}{2}}r(\widehat{x})h$) avec $\| \phi(0)\|\leq r(\widehat{x})h$ (respectivement $\| \phi(0)\|\leq e^{-\frac{3\gamma}{4}}r(\widehat{x})h$) et $\mathrm{Lip}(\phi)\leq \gamma_0$ (respectivement $\mathrm{Lip}(\phi)\leq \gamma_0e^{-\frac{\gamma}{2}}$).
\end{Lem}

\begin{proof}
C'est exactement ce que l'on a fait dans la démonstration du Closing Lemma (voir après le lemme \ref{lemme1}). Il suffit de remplacer $\widehat{f}^m(\widehat{x})$ par $\widehat{y}$ et de mettre des conditions un peu plus fortes sur les $\gamma_0$, $h$ et $\eta$. Ces conditions pouvaient être supposées dans le Closing Lemma si on voulait.
\end{proof}

\begin{Rem}
On a un lemme du même type en remplaçant $E_u(\widehat{x})$ et $E_u(\widehat{y})$ par $E_s(\widehat{x})$ et $E_s(\widehat{y})$ et des graphes $(\phi(Y),Y)$ au lieu de $(X,\psi(X))$. 
\end{Rem}

Le codage va être construit à  partir de graphes provenant de transformées de graphes en arrière et en avant. 

Commençons par ceux construits par des tirés en arrière. 

On se place dans le repère $C_{\gamma}^{-1}(\widehat{y})E_u(\widehat{y})\oplus C_{\gamma}^{-1}(\widehat{y})E_s(\widehat{y})$ pour le $\widehat{y}$ construit dans le paragraphe précédent. 

\noindent Soit $\mathcal{A}^0$ l'ensemble constitué de l'union des graphes d'applications holomorphes $(\psi(Y),Y)$ au-dessus de $C_{\gamma}^{-1}(\widehat{y})E_s(\widehat{y})$ avec $ \| Y\|\leq hr(\widehat{y})e^{\frac{\gamma}{2}}$, $\psi(Y)=\mathrm{constante}$,  $\| \psi(0)\|\leq hr(\widehat{y})e^{-\frac{\gamma}{2}}$ et $\mathrm{Lip}( \psi)=0 \leq \gamma_0e^{-\frac{\gamma}{2}}$.

L'image de ces graphes par $C_i=C_{\gamma}^{-1}(\widehat{f}^n(\widehat{x_i}))\tau_{f^n(x_i)}^{-1}\tau_y C_{\gamma}(\widehat{y})$ (pour $i=1,\cdots,N$) sont des graphes $(\phi(Y),Y)$ au-dessus d'une partie de $C_{\gamma}^{-1}(\widehat{f}^n(\widehat{x_i}))E_s(\widehat{f}^n(\widehat{x_i}))$ pour au moins $\| Y\|\leq hr(\widehat{f}^n(\widehat{x_i}))$ avec $\| \phi(0)\|\leq hr(\widehat{f}^n(\widehat{x_i}))$ et $\mathrm{Lip}( \phi) \leq \gamma_0$ grâce à  la remarque précédente car $\mathrm{dist}(\widehat{f}^n(\widehat{x_i}),\widehat{y})<\eta$ et $\widehat{f}^n(x_i)\in \xi(\widehat{x_i})\subset \Lambda_{\delta}$, $\widehat{y}\in \Lambda_{\delta}$. 

Ensuite on tire en arrière ces graphes $(\phi(Y),Y)$ par $g_{\widehat{f}^{n-1}(\widehat{x_i})},\cdots,g_{\widehat{x_i}}$ (voir le lemme \ref{lemme2}). On obtient à  la fin des graphes $(\phi_0(Y),Y)$ au-dessus d'une partie de $C_{\gamma}^{-1}(\widehat{x_i})E_s(\widehat{x_i})$ pour au-moins $\| Y\|\leq e^{\gamma}hr(\widehat{x_i})$, $\| \phi_0(0)\|\leq e^{-\gamma}hr(\widehat{x_i})$ et $\mathrm{Lip}(\phi_0)\leq e^{-\gamma}\gamma_0$. 

Maintenant, on remet ces graphes dans le repère initial $C_{\gamma}^{-1}(\widehat{y})E_u(\widehat{y})\oplus C_{\gamma}^{-1}(\widehat{y})E_s(\widehat{y})$ c'est-à-dire que l'on prend leur image par $C_{\gamma}^{-1}(\widehat{y})\tau_y^{-1}\tau_{x_i} C_{\gamma}(\widehat{x_i})$.

Comme $\mathrm{dist}(\widehat{x_i},\widehat{y})<\eta$, la remarque précédente implique que l'on obtient des graphes $(\psi_0(Y),Y)$ pour au moins $\| Y\|\leq e^{\frac{\gamma}{2}}hr(\widehat{y})$ avec $\| \psi_0(0)\| \leq e^{-\frac{3\gamma}{4}}hr(\widehat{y})$ et $\mathrm{Lip} (\psi_0) \leq e^{-\frac{\gamma}{2}}\gamma_0$. 

Ce sont exactement les mêmes conditions qu'au départ (c'est-à-dire que l'on pourra recommencer le procédé avec un $\widehat{x_j}$ pas nécessairement égal au $\widehat{x_i}$), sauf qu'en plus on a $\| \psi_0(0)\|\leq e^{-\frac{3\gamma}{4}}hr(\widehat{y})$. 

Grâce à  cette marge, si on prend $\| Y\|\leq e^{\frac{\gamma}{2}}hr(\widehat{y})$ on a (quitte à supposer que $\gamma_0$ est suffisamment petit pour que $e^{-\frac{3\gamma}{4}} + \gamma_0 < e^{-\frac{\gamma}{2}}$): 

$$\| \psi_0(Y)\| \leq \| \psi_0(0)\|+ \| \psi_0(Y)-\psi_0(0)\|\leq  e^{-\frac{3\gamma}{4}}hr(\widehat{y})+\gamma_0e^{-\frac{\gamma}{2}}\| Y\| < hr(\widehat{y})e^{-\frac{\gamma}{2}}$$

c'est-à-dire que les graphes $(\psi_0(Y),Y)$ ne sortent de $B_{k_1}(0,e^{-\frac{\gamma}{2}}hr(\widehat{y}))\times B_{k_2}(0,e^{\frac{\gamma}{2}}hr(\widehat{y}))$ que par le bord $B_{k_1}(0,e^{-\frac{\gamma}{2}}hr(\widehat{y}))\times \partial B_{k_2}(0,e^{\frac{\gamma}{2}}hr(\widehat{y}))$. Ce sont des graphes verticaux. 

Notons $\mathcal{A}_1^0,\cdots,\mathcal{A}_N^0$ les ensembles constitués de l'union de ces graphes obtenus pour $i=1,\cdots,N$. On a donc $\mathcal{A}_i^0\subset \mathcal{A}^0$. On notera aussi $\mathcal{A}_i$ l'image de $\mathcal{A}_i^0$ par $\tau_y\circ C_{\gamma}(\widehat{y})$. 

On a alors: 
\begin{Lem}{\label{lemme3.3}}
Les ensembles $\overline{\mathcal{A}_i}$ sont disjoints pour $i=1,\cdots,N$.
\end{Lem}

\begin{proof}

Soit $z\in \mathcal{A}_i$. Pour $l=0,\cdots,n-1$ on a:

\begin{equation*}
\begin{split}
& \| g_{\widehat{f}^l(\widehat{x_i})}\circ \cdots\circ g_{\widehat{x_i}}\circ C_{\gamma}^{-1}(\widehat{x_i})\tau_{x_i}^{-1}(z)\| \\
& \leq \| g_{\widehat{f}^l(\widehat{x_i})}\circ \cdots\circ g_{\widehat{x_i}}\circ C_{\gamma}^{-1}(\widehat{x_i})\tau_{x_i}^{-1}(z)-g_{\widehat{f}^l(\widehat{x_i})}\circ \cdots\circ g_{\widehat{x_i}}(y)\|+\| g_{\widehat{f}^l(\widehat{x_i})}\circ \cdots\circ g_{\widehat{x_i}}(y)\| 
\end{split}
\end{equation*}

où $y$ est le point d'intersection entre le graphe de la forme $(\phi_0(Y),Y)$ précédent qui contient $C_{\gamma}^{-1}(\widehat{x_i})\tau_{x_i}^{-1}(z)$ et le plan complexe horizontal $(X,0)$ pour $\| X\|\leq hr(\widehat{x_i})$.

On obtient donc (voir le paragraphe \ref{point2}):

$$\| g_{\widehat{f}^l(\widehat{x_i})}\circ \cdots\circ g_{\widehat{x_i}}\circ C_{\gamma}^{-1}(\widehat{x_i})\tau_{x_i}^{-1}(z)\| \leq 4h\max(e^{-2\gamma l},e^{-2\gamma(n-l-1)}).$$

Par continuité cette inégalité est encore vraie si $z\in\overline{\mathcal{A}_i}$. 

Ainsi, toujours par le paragraphe \ref{point2},

$$\mathrm{dist}(f^l(z),f^l(x_i))\leq \frac{C(X)e^{2\gamma}4h}{r_0}\max(e^{-\gamma l},e^{-\gamma(n-l)})$$

car $\widehat{f}^n(\widehat{x_i})\in\xi(\widehat{x_i})\subset \Lambda_{\delta}$. 

Si $h$ est assez petit, on a donc:

$$\mathrm{dist}(f^l(z),f^l(x_i))\leq \frac{\epsilon}{4} \mbox{  pour  } l=0,\cdots,n.$$

Comme les points $x_i$ sont $(m,\epsilon)$-séparés, ils sont $(n,\epsilon)$-séparés car $n\geq m$. En particulier si $i\neq j$ il existe $l$ compris entre $0$ et $n$ avec $\mathrm{dist}(f^l(x_i),f^l(x_j))\geq \epsilon$. 

Si on avait un point $z\in \overline{\mathcal{A}_i}\cap \overline{\mathcal{A}_j}$ on aurait alors: 

$$\epsilon\leq \mathrm{dist}(f^l(x_i),f^l(x_j)) \leq \mathrm{dist}(f^l(x_i),f^l(z))+\mathrm{dist}(f^l(z),f^l(x_j)) \leq 2\times\frac{\epsilon}{4}$$

ce qui est absurde. 

Les ensembles $\overline{\mathcal{A}_i}$ sont donc bien disjoints et cela démontre le lemme.
\end{proof}

Les graphes qui constituent l'ensemble $\mathcal{A}_j^0$ vérifient les mêmes conditions qu'au départ. On peut donc recommencer le processus par rapport à  $\widehat{x_i}$, c'est-à-dire que l'on passe ces graphes dans le repère $C_{\gamma}^{-1}(\widehat{f}^n(\widehat{x_i}))E_u(\widehat{f}^n(\widehat{x_i}))\oplus C_{\gamma}^{-1}(\widehat{f}^n(\widehat{x_i}))E_s(\widehat{f}^n(\widehat{x_i}))$ puis on tire en arrière par $g_{\widehat{f}^{n-1}(\widehat{x_i})},\cdots,g_{\widehat{x_i}}$ et on repasse dans le repère initial, à  savoir $C_{\gamma}^{-1}(\widehat{y})E_u(\widehat{y})\oplus C_{\gamma}^{-1}(\widehat{y})E_s(\widehat{y})$. 

On obtient ainsi, en faisant varier $i$ et $j$, $N^2$ ensembles $\mathcal{A}_{ij}^0$ (avec le même type de notation que précédemment). 

On a alors: 
\begin{Lem}
$\mathcal{A}_{ij}^{0}\subset \mathcal{A}_i^0$ pour tout $i=1,\cdots,N$ et $j=1,\cdots,N.$
\end{Lem}

\begin{proof}

Soit $z\in \mathcal{A}_{ij}^{0}$. Par construction $z$ appartient à  un graphe $(\psi_0(Y),Y)$ au-dessus de $B_{k_2}(0,hr(\widehat{y})e^{\frac{\gamma}{2}})\subset C_{\gamma}^{-1}(\widehat{y})E_s(\widehat{y})$. 

Soit $B$ un plan complexe parallèle à $C_{\gamma}^{-1}(\widehat{y})E_u(\widehat{y})$ et qui passe par $z$.
 
$B$ est un graphe $(X,\phi(X))$ au-dessus de $B_{k_1}(0,hr(\widehat{y})e^{\frac{5 \gamma}{2}})$ avec $\phi$ constante (donc $\mathrm{Lip} (\phi)=0$) et $\| \phi(0)\|\leq hr(\widehat{y})e^{\frac{\gamma}{2}}$. 

Par le lemme \ref{lemme3.2} (en remplaçant $h$ par $e^{\frac{3\gamma}{2}}h$) l'image de $B$ par $C_{\gamma}^{-1}(\widehat{x_i})\tau_{x_i}^{-1}\tau_y C_{\gamma}(\widehat{y})$ est un graphe $B_0$ de la forme $(X,\phi_0(X))$ au moins pour $\| X\| \leq e^{\gamma}hr(\widehat{x_i})$ avec $\| \phi_0(0) \| \leq e^{\frac{3\gamma}{4}} hr(\widehat{x_i})$ et $\mathrm{Lip}(\phi_0) \leq\gamma_0$. 

Maintenant on peut pousser en avant ce graphe par $g_{\widehat{x_i}},\cdots,g_{\widehat{f}^{n-1}(\widehat{x_i})}$ comme dans le lemme \ref{lemme1} avec $e^{\gamma}h$ à la place de $h$. A chaque étape on fait aussi le cut-off (pour $\| X\|\leq he^{\gamma}r(\widehat{f}^l(\widehat{x_i}))$). 

A la fin on obtient un graphe $B_n$ de la forme $(X,\phi_n(X))$ au-dessus d'une partie de $C_{\gamma}^{-1}(\widehat{f}^n(\widehat{x_i}))E_u(\widehat{f}^n(\widehat{x_i}))$ qui contient $B(0,hr(\widehat{f}^n(\widehat{x_i}))e^{2\gamma})$ avec $\| \phi_n(0)\|\leq e^{-\frac{\gamma}{4}}hr(\widehat{f}^n(\widehat{x_i}))$ et $\mathrm{Lip}(\phi_n) \leq \gamma_0e^{-\gamma}$. 

Pour $l=0,\cdots,n$, notons: 

$$z_l=g_{\widehat{f}^{l-1}(\widehat{x_i})}\circ \cdots \circ g_{\widehat{x_i}}(C_{\gamma}^{-1}(\widehat{x_i})\tau_{x_i}^{-1}\tau_y C_{\gamma}(\widehat{y})(z)).$$

\noindent Par construction de $\mathcal{A}_{ij}^0$, $z_l$ est dans un graphe $(\psi_l'(Y),Y)$ avec au pire $\| Y\|\leq e^{\gamma} hr(\widehat{f}^l(\widehat{x_i}))$, $\| \psi_l'(0)\|\leq hr(\widehat{f}^l(\widehat{x_i}))$ et $\mathrm{Lip}( \psi_l') \leq \gamma_0$. 

On a donc: 

\begin{equation*}
\begin{split}
&\| \psi_l'(Y)\| \leq \| \psi_l'(0)\|+\| \psi_l'(Y)-\psi_l'(0)\| \\
& \leq hr(\widehat{f}^l(\widehat{x_i}))+\gamma_0 e^{\gamma} hr(\widehat{f}^l(\widehat{x_i})) \leq e^{\gamma}hr(\widehat{f}^l(\widehat{x_i})) \\ 
\end{split}
\end{equation*}

car $\gamma_0$ est petit par rapport à  $\gamma$.

Le point $z_l$ est donc dans la partie de $B_l$ que l'on garde quand on fait les cut-off. 

Maintenant, quand on prend l'image du dernier graphe $(X,\phi_n(X))$ par 

$C_{\gamma}^{-1}(\widehat{y})\tau_y^{-1}\tau_{f^n(x_i)}C_{\gamma}(\widehat{f}^n(\widehat{x_i}))$ on obtient un graphe $\widetilde{B}$ de la forme $(X,\widetilde{\phi}(X))$ au moins au-dessus de $B(0,hr(\widehat{y})e^{\gamma})$ avec  $\| \widetilde{\phi}(0)\|\leq hr(\widehat{y})$ et $\mathrm{Lip}(\widetilde{\phi}) \leq \gamma_0e^{-\frac{\gamma}{2}}$ (toujours par le lemme \ref{lemme3.2} avec $e^{\frac{3 \gamma}{4}}h$ au lieu de $h$).

Par définition de $\mathcal{A}_{ij}^0$, 

$$\widetilde{z_n}=C_{\gamma}^{-1}(\widehat{y})\tau_{y}^{-1}\tau_{f^n(x_i)} C_{\gamma}(\widehat{f}^n(\widehat{x_i}))(z_n) \in \mathcal{A}_j^0\subset \mathcal{A}^0.$$

Le point $\widetilde{z_n}$ appartient donc à un graphe $\widetilde{\Delta}$ qui constitue l'ensemble $\mathcal{A}^0$. Ce graphe $\widetilde{\Delta}$ fait partie des graphes que l'on utilise pour construire $\mathcal{A}_i^0$ via le processus qui utilise $g_{\widehat{x_i}},\cdots,g_{\widehat{f}^{n-1}(\widehat{x_i})}$. 

Soit $\Delta$ le graphe de $\mathcal{A}_i^0$ issu de $\widetilde{\Delta}$ via ce procédé. 

$\Delta$ est un graphe $(\psi(Y),Y)$ au-dessus de $B(0,hr(\widehat{y})e^{\frac{\gamma}{2}})\subset C_{\gamma}^{-1}(\widehat{y}) E_s(\widehat{y})$ avec $\| \psi(0)\|\leq e^{-\frac{\gamma}{2}}hr(\widehat{y})$ et $\mathrm{Lip}( \psi) \leq \gamma_0e^{-\frac{\gamma}{2}}$. $\Delta$ coupe donc $B$ en un unique point que l'on note $z'$. 

Montrons que $z=z'$: cela impliquera que $z\in \mathcal{A}_i^0$ et terminera la démonstration du lemme. 

Supposons que $z\neq z'$. 

Leurs images par $C_{\gamma}^{-1}(\widehat{x_i})\tau_{x_i}^{-1}\tau_y C_{\gamma}(\widehat{y})$ donnent deux points $z_0$ et $z_0'$ distincts de $B_0$. 

Comme précédemment les points $z_l'=g_{\widehat{f}^{l-1}(\widehat{x_i})}\circ \cdots\circ g_{\widehat{x_i}}(C_{\gamma}^{-1}(\widehat{x_i})\tau_{x_i}^{-1}\tau_{y} C_{\gamma}(\widehat{y})(z'))$ restent dans la partie de $B_l$ que l'on garde quand on fait le cut-off (par construction de $\mathcal{A}_i^0$).

De plus, par un résultat du même type que le lemme \ref{lemme4} on a: 

$$\| g_{\widehat{f}^{l-1}(\widehat{x_i})}(z_{l-1})-g_{\widehat{f}^{l-1}(\widehat{x_i})}(z_{l-1}')\| \geq e^{2\gamma}\| z_{l-1}-z_{l-1}'\| \mbox{  pour tout  } l=1,\cdots,n.$$

Par récurrence cela donne que pour $l=0,\cdots,n$ le point $z_l$ est distinct de $z_l'$. En particulier $z_n\neq z_n'$. Mais $C_{\gamma}^{-1}(\widehat{y})\tau_{y}^{-1}\tau_{f^n(x_i)} C_{\gamma}(\widehat{f}^n(\widehat{x_i}))(z_n')$ est dans $\widetilde{\Delta}\cap \widetilde{B}$ qui est réduit à $\widetilde{z_n}=C_{\gamma}^{-1}(\widehat{y})\tau_{y}^{-1}\tau_{f^n(x_i)} C_{\gamma}(\widehat{f}^n(\widehat{x_i}))(z_n)$. 

On obtient ainsi la contradiction recherchée. 
\end{proof}

A la deuxième génération on a donc construit $N^2$ ensembles $\mathcal{A}_{ij}^0$ avec $\mathcal{A}_{ij}^0\subset \mathcal{A}_i^0$ pour $j=1,\cdots,N$. On va maintenant itérer le processus. 

\noindent On part de $\mathcal{A}_{ij}^0$ qui rappelons le provient de $\mathcal{A}_j^0$ via le procédé qui utilise $g_{\widehat{x_i}},\cdots,g_{\widehat{f}^{n-1}(\widehat{x_i})}$. 

$\mathcal{A}_{j}^0$ contient $N$ ensembles $\mathcal{A}_{jl}^0$ pour $l=1,\cdots,N$. Si on applique le processus aux graphes de $\mathcal{A}_{jl}^0$ (via $g_{\widehat{x_i}},\cdots,g_{\widehat{f}^{n-1}(\widehat{x_i})}$), on obtient un ensemble que l'on note $\mathcal{A}_{ijl}^0$ ($l=1,\cdots,N$). Ils vérifient $\mathcal{A}_{ijl}^0\subset \mathcal{A}_{ij}^0$ par le même raisonnement que $\mathcal{A}_{ij}^0\subset \mathcal{A}_i^0$. 

Et on continue le procédé. 

A la $(l+1)$-ème génération on a $N^{l+1}$ ensembles $\mathcal{A}_{w_0\cdots w_{-l}}^0$ avec $w_i=1,\cdots,N$. 

$\mathcal{A}_{w_0\cdots w_{-l}}^0$ est constitué de graphes $(\psi(Y),Y)$ avec $\| Y\|\leq e^{\frac{\gamma}{2}}hr(\widehat{y})$, $\| \psi(0)\|\leq e^{-\frac{3\gamma}{4}}h r(\widehat{y})$ et $\mathrm{Lip}( \psi) \leq \gamma_0 e^{-\frac{\gamma}{2}}$. 

Ces graphes proviennent des graphes de $\mathcal{A}_{w_{-1}\cdots w_{-l}}^0$ via le procédé qui utilise $g_{\widehat{x_{w_0}}}$ , $\cdots$ ,$g_{\widehat{f}^{n-1}(\widehat{x_{w_0}})}$. 

Comme précédemment on a $\mathcal{A}_{w_{0}\cdots w_{-l}}^0\subset \mathcal{A}_{w_{0}\cdots w_{-l+1}}^0$ pour tout $w_0,\cdots,w_{-l}\in \{1,\cdots,N\}$ et $l\geq 0$. 

De plus on a: 

\begin{Lem}
Les $N^{l+1}$ ensembles $\overline{\mathcal{A}_{w_0\cdots w_{-l}}^0}$ sont disjoints. 
\end{Lem}

\begin{proof}

On démontre le lemme avec une récurrence sur $l$. 

Pour $l=0$, c'est déjà  fait. Supposons la propriété vraie au rang $l-1$ avec $l \geq 1$.

Au rang $l$ considérons deux ensembles $\overline{\mathcal{A}_{w_0\cdots w_{-l}}^0}$ et $\overline{\mathcal{A}_{w_0'\cdots w_{-l}'}^0}$ avec $(w_0,\cdots,w_{-l})\neq (w_0',\cdots,w_{-l}')$. 

On distingue les deux cas suivants: 

- si $(w_0,\cdots,w_{-l+1})\neq (w_0',\cdots,w_{-l+1}')$ alors on a $\overline{\mathcal{A}_{w_0\cdots w_{-l+1}}^0}$ et $\overline{\mathcal{A}_{w_0'\cdots w_{-l+1}'}^0}$ qui sont disjoints par hypothèse de récurrence . Comme $\overline{\mathcal{A}_{w_0\cdots w_{-l}}^0}\subset \overline{\mathcal{A}_{w_0\cdots w_{-l+1}}^0}$ et $\overline{\mathcal{A}_{w_0'\cdots  w_{-l}'}^0}\subset \overline{\mathcal{A}_{w_0'\cdots w_{-l+1}'}^0}$ on obtient que les deux ensembles $\overline{\mathcal{A}_{w_0\cdots w_{-l}}^0}$ et $\overline{\mathcal{A}_{w_0'\cdots w_{-l}'}^0}$ sont disjoints. 

- si $(w_0,\cdots,w_{-l+1})= (w_0',\cdots,w_{-l+1}')$ alors nécessairement $w_{-l}\neq w_{-l}'$. Par hypothèse de récurrence, $\overline{\mathcal{A}_{w_{-1} \cdots w_{-l}}^0}$ et $\overline{\mathcal{A}_{w_{-1}'\cdots w_{-l}'}^0}$ sont disjoints. 

Comme $\mathcal{A}_{w_0\cdots w_{-l}}^0$ vient de $\mathcal{A}_{w_{-1}\cdots w_{-l}}^0$ via le processus qui utilise $g_{\widehat{x_{w_0}}},\cdots,g_{\widehat{f}^{n-1}(\widehat{x_{w_0}})}$ et que $\mathcal{A}_{w_0'\cdots w_{-l}'}^0$ vient de $\mathcal{A}_{w_{-1}'\cdots w_{-l}'}^0$ via le même procédé car $w_0=w_0'$, on en déduit que $\overline{\mathcal{A}_{w_0\cdots w_{-l}}^0}$ est disjoint de $\overline{\mathcal{A}_{w_0'\cdots w_{-l}'}^0}$.

Cela termine la démonstration du lemme. 

\end{proof}

On veut maintenant évaluer la distance entre les graphes d'un même ensemble $\overline{\mathcal{A}_{w_0\cdots w_{-l}}^0}$ de la $(l+1)$-ème génération. 

Pour cela, soit $G_v$ l'ensemble des graphes de fonctions holomorphes $(\phi(Y),Y)$ au-dessus de $\overline{B_{k_2}(0,e^{\gamma/2}hr(\widehat{y}))}$ avec $ \|\phi(0)\|\leq e^{-\frac{\gamma}{2}}hr(\widehat{y})$ et $\mathrm{Lip}( \phi) \leq\gamma_0$.

Les graphes qui constituent les $\mathcal{A}_{w_0\cdots w_{-l}}^0$ sont dans $G_v$. 

Sur $G_v$ on met la métrique: 

$$\mathrm{d}(\mathrm{graphe} \mbox{    } \phi,\mathrm{graphe} \mbox{    } \psi)=\max_{\overline{B_{k_2}(0,e^{\gamma/2}hr(\widehat{y}))}}\| \phi-\psi\|.$$

On a alors: 

\begin{Lem}{\label{lemme3.6}}
Pour tout $l\geq 1$ et pour tout $w_0,\cdots,w_{-l+1}\in \{1,\cdots,N\}$ on a:

$$\forall A_l^1,A_l^2 \mbox{    } \mbox{  graphes qui constituent  }\ \mathcal{A}_{w_0\cdots w_{-l+1}}^0, \mbox{    }  \mathrm{d}(A_l^1,A_l^2)\leq 2he^{-2\gamma n l +2\gamma l}.$$

\end{Lem}

\begin{proof}

$A_l^1$ vient d'un graphe $A_{l-1}^1$ de $\mathcal{A}_{w_{-1} \cdots w_{-l+1}}^0$ qui vient d'un graphe $A_{l-2}^1$ de $\mathcal{A}_{w_{-2}\cdots w_{-l+1}}^0$, et ainsi de suite. On notera $A_{l-p}^1$ ($p=0,\cdots,l$) les graphes de $\mathcal{A}_{w_{-p}\cdots w_{-l+1}}^0$ qui font partie du processus pour construire $A_l^1$. On fait de même pour $A_l^2$. 

On a alors:

$$\forall p=1,\cdots,l,\mbox{    }  \mathrm{d}(A_p^1,A_p^2)\leq e^{-2\gamma n +2\gamma}\mathrm{d}(A_{p-1}^1,A_{p-1}^2).$$

En effet, pour passer de $A_{p-1}^1$ à  $A_p^1$ et de $A_{p-1}^2$ à  $A_p^2$ on prend tout d'abord les images de $A_{p-1}^1$ et $A_{p-1}^2$ par l'application $C_{\gamma}^{-1}(\widehat{f}^n(\widehat{x_{w_{-l+p}}}))\tau_{f^n(x_{w_{-l+p}})}^{-1}\tau_y C_{\gamma}(\widehat{y})$.

Mais on a vu dans la démonstration du lemme \ref{lemme6} que comme $\widehat{f}^n(\widehat{x_{w_{-l+p}}})$ est très proche de $\widehat{y}$, cette opération multiplie la distance entre les graphes $\mathrm{d}(A_{p-1}^1,A_{p-1}^2)$ par au plus $e^{\gamma}$. Ensuite, on tire en arrière par $g_{\widehat{f}^{n-1}(\widehat{x_{w_{-l+p}}})},\cdots,g_{\widehat{x_{w_{-l+p}}}}$. On a montré précédemment (voir la preuve du lemme \ref{lemme7}) qu'à chaque étape la distance entre les graphes est divisée par au moins $e^{2\gamma}$. 

Enfin à  la dernière étape on prend l'image par $C_{\gamma}^{-1}(\widehat{y})\tau_y^{-1}\tau_{x_{w_{-l+p}}}C_{\gamma}(\widehat{x_{w_{-l+p}}})$ et là  cela multiplie encore la distance par au plus $e^{\gamma}$. 

En appliquant l'inégalité: 

$$\forall p=1,\cdots,l,\mbox{    } \mathrm{d}(A_p^1,A_p^2)\leq e^{-2\gamma n +2\gamma}\mathrm{d}(A_{p-1}^1,A_{p-1}^2),$$

on obtient: 

$$ \mathrm{d}(A_l^1,A_l^2) \leq e^{-2\gamma n l +2\gamma l}\mathrm{d}(A_{0}^1,A_{0}^2) \leq e^{-2\gamma n l+2\gamma l}\times 2e^{-\frac{\gamma}{2}}hr(\widehat{y}) \leq 2he^{-2\gamma n l+2\gamma l} .$$

C'est ce que l'on voulait démontrer. 
\end{proof}

Grâce à  ce lemme on peut définir une application de $\{1,\cdots,N\}^{\mathbb{N}}$ dans $G_v$ de la façon suivante.

Soit $w=(w_0,w_{-1},\cdots,w_{-l},\cdots)\in \{1,\cdots,N\}^{\mathbb{N}}$. On prend $A_l$ un graphe qui constitue $\mathcal{A}_{w_0 w_{-1}\cdots w_{-l}}^0$ pour $l\geq 0$. 

$A_l$ est dans $G_v$ et par le théorème d'Ascoli, il existe une sous-suite $(A_{l_j})$ qui converge vers $A \in G_v$ (une limite uniforme de fonctions holomorphes est holomorphe). 

Montrons que $(A_l)_l$ converge aussi vers $A$. 

Soit $\beta>0$. Il existe $j_0$ tel que $\mathrm{d}(A,A_{l_j})<\frac{\beta}{2}$ si $j\geq j_0$. Ensuite

$$A_{l_j}^0 \subset \mathcal{A}_{w_0\cdots w_{-l_j}}^0 \subset \mathcal{A}_{w_0\cdots w_{-j}}^0$$

donc $\mathrm{d}(A_{l_j},A_j)\leq 2he^{-2\gamma n (j+1)+2\gamma (j+1)}$ en utilisant le lemme précédent. Cette quantité est plus petite que $\frac{\beta}{2}$ si $j\geq j_1$ (car $n\geq 3$ si on veut). 

Pour $j\geq \max(j_0,j_1)$ on a donc: 

$$\mathrm{d}(A_j,A)\leq \mathrm{d}(A_j,A_{l_j})+\mathrm{d}(A_{l_j},A)<\beta.$$

La suite $(A_j)$ converge vers $A$ et on a donc bien défini une application: 

\begin{center}
\begin{tabular}{cccc}
$\sigma_v:$ & $\{1,\cdots,N\}^{\mathbb{N}}$ & $\longrightarrow$ & $G_v$\\
             & $w$                         & $\longrightarrow$ & $A$ 
\end{tabular}
\end{center}

\begin{Rem}
L'application $\sigma_v$ est injective et on a même $\sigma_v(w)\cap \sigma_v(w')=\varnothing$ si $w\neq w'$. 

En effet, soit $w,w'\in \{1,\cdots,N\}^{\mathbb{N}}$ avec $w\neq w'$. On considère $l$ le plus petit entier avec $w_{-l}\neq w_{-l}'$. Soit $A_p\in \mathcal{A}_{w_0\cdots w_{-p}}^0$ et $A_p'\in \mathcal{A}_{w_0'\cdots w_{-p}'}^0$ pour $p\geq 0$ (on a donc $A_p\longrightarrow\sigma_v(w)$ et $A_p'\longrightarrow\sigma_v(w')$).  

Comme $w_{-l}\neq w_{-l}'$, les ensembles $\overline{\mathcal{A}_{w_0\cdots w_{-l}}^0}$ et $\overline{\mathcal{A}_{w_0'\cdots w_{-l}'}^0}$ sont disjoints.

Or pour tout $p\geq l$, $A_p \subset \mathcal{A}_{w_0\cdots w_{-p}}^0\subset\overline{\mathcal{A}_{w_0\cdots w_{-l}}^0}$ et $A_p' \subset \mathcal{A}_{w_0'\cdots w_{-p}'}^0\subset \overline{\mathcal{A}_{w_0'\cdots w_{-l}'}^0}$ donc $\sigma_v(w)\cap \sigma_v(w')=\varnothing$. En particulier $\sigma_v$ est injective.
\end{Rem}

On refait maintenant ce que l'on a réalisé avec des transformées de graphes en avant.

Comme précédemment on se place dans le repère $C_{\gamma}^{-1}(\widehat{y})E_u(\widehat{y})\oplus C_{\gamma}^{-1}(\widehat{y})E_s(\widehat{y})$ pour le $\widehat{y}\in \Lambda_{\delta}$ fixé. 

On part de l'ensemble des graphes de la forme $(X,\phi(X))$ au-dessus de $C_{\gamma}^{-1}(\widehat{y})E_u(\widehat{y})$ avec $\| X\|\leq hr(\widehat{y})e^{\frac{\gamma }{2}}$, $\phi(X)=\mathrm{constante}$, $\| \phi(0)\|\leq hr(\widehat{y})e^{-\frac{\gamma}{2}}$ et $\mathrm{Lip}( \phi)=0\leq \gamma_0e^{-\frac{\gamma}{2}}$. 

On notera $\mathcal{B}^0$ l'ensemble constitué par l'union de ces graphes. 

L'image de ces graphes par $C_i=C_{\gamma}^{-1}(\widehat{x_i})\tau_{x_i}^{-1}\tau_y C_{\gamma}(\widehat{y})$ (pour $i=1,\cdots,N$) sont des graphes $(X,\psi(X))$ au-dessus d'une partie de $C_{\gamma}^{-1}(\widehat{x_i})E_u(\widehat{x_i})$ pour au moins $\| X\|\leq hr(\widehat{x_i})$ avec $\| \psi(0)\|\leq hr(\widehat{x_i})$ et $\mathrm{Lip}( \psi) \leq \gamma_0$ par le lemme \ref{lemme3.2} car $\mathrm{dist}(\widehat{y},\widehat{x_i})<\eta$.

Ensuite on pousse en avant ces graphes $(X,\psi(X))$ par $g_{\widehat{x_i}},\cdots, g_{\widehat{f}^{n-1}(\widehat{x_i})}$. On obtient à  la fin des graphes $(X,\psi_n(X))$ au-dessus d'une partie de $C_{\gamma}^{-1}(\widehat{f}^n(\widehat{x_i}))E_u(\widehat{f}^n(\widehat{x_i}))$ pour au moins $\| X\|\leq e^{\gamma}hr(\widehat{f}^n(\widehat{x_i}))$ avec $\| \psi_n(0)\|\leq e^{-\gamma}hr(\widehat{f}^n(\widehat{x_i}))$ et $\mathrm{Lip}( \psi_n) \leq e^{-\gamma}\gamma_0$ (voir le lemme \ref{lemme1}).

Maintenant, on remet ces graphes dans le repère initial $C_{\gamma}^{-1}(\widehat{y})E_u(\widehat{y})\oplus C_{\gamma}^{-1}(\widehat{y})E_s(\widehat{y})$ c'est-à-dire que l'on prend leur image par $C_{\gamma}^{-1}(\widehat{y})\tau_y^{-1}\tau_{f^n(x_i)} C_{\gamma}(\widehat{f}^n(\widehat{x_i}))$. 

Comme $\mathrm{dist}(\widehat{f^n}(\widehat{x_i}),\widehat{y})<\eta$, on obtient des graphes $(X, \phi_n(X))$ pour au moins $\| X\|\leq e^{\frac{\gamma}{2}}hr(\widehat{y})$ avec $\| \phi_n(0)\| \leq e^{-\frac{3\gamma}{4}}hr(\widehat{y})$ et $\mathrm{Lip}( \phi_n) \leq e^{-\frac{\gamma}{2}}\gamma_0$ (toujours par le lemme \ref{lemme3.2}).

Ce sont exactement les mêmes conditions qu'au départ (c'est-à-dire que l'on pourra recommencer le procédé avec un $\widehat{x_j}$ pas nécessairement égal au $\widehat{x_i}$), sauf qu'en plus on a $\| \phi_n(0)\|\leq e^{-\frac{3\gamma}{4}}hr(\widehat{y})$. 

Grâce à  cette marge, si on prend $\| X\|\leq e^{\frac{\gamma}{2}}hr(\widehat{y})$ on a: 

$$\| \phi_n(X)\| \leq \| \phi_n(0)\|+ \| \phi_n(X)-\phi_n(0)\|\leq  e^{-\frac{3\gamma}{4}}hr(\widehat{y})+\gamma_0e^{-\frac{\gamma}{2}}\| X\| < hr(\widehat{y})e^{-\frac{\gamma}{2}}$$

c'est-à-dire que les graphes $(X,\phi_n(X))$ ne sortent de $B(0,e^{\frac{\gamma}{2}}hr(\widehat{y}))\times B(0,e^{-\frac{\gamma}{2}}hr(\widehat{y}))$ que par le bord $\partial B(0,e^{\frac{\gamma}{2}}hr(\widehat{y}))\times B(0,e^{-\frac{\gamma}{2}}hr(\widehat{y}))$. Ce sont des graphes horizontaux. 

Notons $\mathcal{B}_1^0,\cdots,\mathcal{B}_N^0$ les ensembles constitués de l'union de ces graphes obtenus pour $i=1,\cdots,N$. On a donc $\mathcal{B}_i^0\subset \mathcal{B}^0$. On notera aussi $\mathcal{B}_i$ l'image de $\mathcal{B}_i^0$ par $\tau_y\circ C_{\gamma}(\widehat{y})$. 

Ici les $\mathcal{B}_i^0$ ne sont pas forcément disjoints car $f$ n'est pas a priori injective. 

On recommence le processus. On fixe $i$ et $j$ dans $\{1,\cdots,N\}$ et on considère les graphes de $\mathcal{B}_j^0$. On les met dans le repère $C_{\gamma}^{-1}(\widehat{x_i})E_u(\widehat{x_i})\oplus C_{\gamma}^{-1}(\widehat{x_i})E_s(\widehat{x_i})$ en appliquant $C_i$, puis on les pousse en avant par $g_{\widehat{x_i}},\cdots,g_{\widehat{f}^{n-1}(\widehat{x_i})}$ et on repasse dans le repère initial. 

On note $\mathcal{B}_{ij}^0$ ($i,j=1,\cdots,N$) les $N^2$ ensembles ainsi obtenus. On a alors: 

\begin{Lem}
$\mathcal{B}_{ij}^0\subset \mathcal{B}_i^0$ pour tout $i,j=1,\cdots,N$.
\end{Lem}

\begin{proof}
Soit $z\in \mathcal{B}_{ij}^0$. Par construction $z$ est dans un graphe de $\mathcal{B}_{ij}^0$ qui provient d'un graphe de $\mathcal{B}_j^0$ via le processus qui utilise $g_{\widehat{x_i}},\cdots,g_{\widehat{f}^{n-1}(\widehat{x_i})}$. En particulier, 

$$z=C_{\gamma}^{-1}(\widehat{y})\tau_{y}^{-1}\tau_{f^n(x_i)}C_{\gamma}(\widehat{f}^n(\widehat{x_i}))g_{\widehat{f}^{n-1}(\widehat{x_i})}\circ \cdots\circ g_{\widehat{x_i}}\circ C_{\gamma}^{-1}(\widehat{x_i})\tau_{x_i}^{-1}\tau_y C_{\gamma}(\widehat{y})(z_0)$$

où $z_0$ est dans un graphe de $\mathcal{B}_j^0$. 

Cela signifie que $C_{\gamma}^{-1}(\widehat{x_i})\tau_{x_i}^{-1}\tau_y C_{\gamma}(\widehat{y})(z_0)$ et $g_{\widehat{f}^{l}(\widehat{x_i})}\circ \cdots\circ g_{\widehat{x_i}}\circ C_{\gamma}^{-1}(\widehat{x_i})\tau_{x_i}^{-1}\tau_y C_{\gamma}(\widehat{y})(z_0)$ (pour $l=0,\cdots,n-1$) et $z$ sont toujours dans la partie du cut-off que l'on garde. Maintenant $z_0\in \mathcal{B}_j^0\subset \mathcal{B}^0$ donc $z_0$ appartient à  un graphe de $\mathcal{B}^0$. Quand on applique le processus à  ce graphe de $\mathcal{B}^0$ via $g_{\widehat{x_i}}, \cdots, g_{\widehat{f}^{n-1}(\widehat{x_i})}$ on obtient à  la fin un graphe de $\mathcal{B}_i^0$ et le point $z$ est dedans car les images successives de $z_0$ restent dans la partie du cut-off que l'on garde par la remarque précédente. On a donc $z\in \mathcal{B}_i^0$ et le lemme est démontré. 
\end{proof}

A la deuxième génération on a donc construit $N^2$ ensembles $\mathcal{B}_{ij}^0$ avec $\mathcal{B}_{ij}^0\subset \mathcal{B}_i^0$. On va maintenant itérer le processus. 

$\mathcal{B}_j^0$ contient $N$ ensembles $\mathcal{B}_{jl}^0$ ($l=1,\cdots,N$). Si on applique le processus aux graphes de $\mathcal{B}_{jl}^0$ (via $g_{\widehat{x_i}},\cdots,g_{\widehat{f}^{n-1}(\widehat{x_i})}$), on obtient un ensemble que l'on note $\mathcal{B}_{ijl}^0$. Comme précédemment, on a $\mathcal{B}_{ijl}^0\subset \mathcal{B}_{ij}^0$ pour tout $i,j,l=1,\cdots,N$. Et on continue ainsi le procédé. 

A la $l$-ème génération on a $N^l$ ensembles $\mathcal{B}_{w_1\cdots w_l}^0$ avec $w_1,\cdots,w_l=1,\cdots,N$. 

$\mathcal{B}_{w_1 \cdots w_l}^0$ est constitué de graphes $(X,\phi(X))$ avec $\| X\|\leq e^{\frac{\gamma}{2}}hr(\widehat{y})$, $\| \phi(0)\|\leq e^{-\frac{\gamma}{2}}h r(\widehat{y})$ et $\mathrm{Lip}( \phi) \leq \gamma_0e^{-\frac{\gamma}{2}}$. Ces graphes proviennent de graphes de $\mathcal{B}_{w_2\cdots w_l}^0$ via le procédé qui utilise $g_{\widehat{x_{w_1}}},\cdots,g_{\widehat{f}^{n-1}(\widehat{x_{w_1}})}$. On a encore $\mathcal{B}_{w_1\cdots w_l}^0\subset \mathcal{B}_{w_1\cdots w_{l-1}}^0$ pour tout $l\geq 1$ et pour tout $w_1,\cdots,w_l=1,\cdots,N$. 

Comme expliqué plus haut, les $\mathcal{B}_{w_1\cdots w_l}^0$ ne sont pas nécessairement disjoints. 

On veut maintenant évaluer la distance entre les graphes qui constituent un même ensemble $\mathcal{B}_{w_1\cdots w_l}^0$. Pour cela, soit $G_h$ l'ensemble des graphes de fonctions holomorphes $(X,\phi(X))$ au-dessus de $\overline{B_{k_1}(0,hr(\widehat{y})e^{\gamma/2})}$ avec $\| \phi(0)\|\leq e^{-\gamma/2}hr(\widehat{y})$ et $\mathrm{Lip}( \phi) \leq \gamma_0$.

Sur $G_h$ on met la métrique: 

$$\mathrm{d}(\mathrm{graphe} \mbox{    } \phi,\mathrm{graphe} \mbox{    } \psi)=\max_{\overline{B_{k_1}(0,hr(\widehat{y})e^{\gamma/2})}}\| \phi-\psi\|.$$

On a alors: 
\begin{Lem}{\label{lemme3.8}}
Pour tout $l\geq 1$ et tout $w_1,\cdots,w_l=1,\cdots,N$ on a: 

$$\forall B_l^1,B_l^2 \mbox{    } \mbox{  graphes qui constituent  } \mathcal{B}_{w_1 \cdots w_l}^0,\mbox{    } \mathrm{d}(B_l^1,B_l^2)\leq 2he^{-2\gamma n l + 2\gamma l}.$$

\end{Lem}

\begin{proof}
$B_l^1$ vient d'un graphe $B_{l-1}^1$ de $\mathcal{B}_{w_2\cdots w_l}^0$ qui vient d'un graphe $B_{l-2}^1$ de $\mathcal{B}_{w_3\cdots w_l}^0$, et ainsi de suite. On note $B_p^1$ pour $p=0,\cdots,l$ les graphes de $\mathcal{B}_{w_{l-p+1}\cdots w_l}^0$ qui font partie du processus pour construire $B_l^1$. On fait de même pour $B_l^2$. 

On a alors pour $p=1,\cdots,l$: 

$$\mathrm{d}(B_p^1,B_p^2)\leq e^{-2\gamma n + 2\gamma}\mathrm{d}(B_{p-1}^1,B_{p-1}^2).$$

La démonstration est la même que dans le lemme \ref{lemme3.6}, en utilisant ce que l'on a fait dans la preuve du lemme \ref{lemme6}. 

En particulier on a maintenant: 

$$\mathrm{d}(B_l^1,B_l^2) \leq e^{-2\gamma n l+ 2\gamma l}\mathrm{d}(B_0^1,B_0^2) \leq 2e^{-\frac{\gamma}{2}}hr(\widehat{y})e^{-2\gamma n l + 2\gamma l} \leq 2he^{-2\gamma n l +2\gamma l}.$$

Cela démontre le lemme.
\end{proof}

Grâce à  ce lemme on peut définir une application de $\{1,\cdots,N\}^{\mathbb{N}}$ dans $G_h$ de la façon suivante. 

Soit $w=(w_1,\cdots,w_l,\cdots)\in \{1,\cdots,N\}^{\mathbb{N}}$. On prend $B_l$ un graphe qui constitue $\mathcal{B}_{w_1\cdots w_l}^0$ pour $l\geq 1$. $B_l$ est dans $G_h$ et par le théorème d'Ascoli, il existe une sous-suite $(B_{l_j})$ qui converge vers $B\in G_h$. Alors comme pour les $\mathcal{A}^0$, grâce au lemme précédent, on a aussi $(B_l)$ qui converge vers $B$. Cela définit une application $\sigma_h: \{1,\cdots,N\}^{\mathbb{N}} \longrightarrow G_v$ par $\sigma_h(w)=B$. 

Si $f$ n'est pas inversible, $\sigma_h$ n'est pas a priori injective. 

Par contre si $f$ est inversible et $\int \log  \mathrm{d}(x,I(f^{-1}))\mathrm{d}\mu(x)>-\infty$, comme dans le cas $\mathcal{A}^0$, $\sigma_h$ est injective (il suffit de prendre des $r(\widehat{x})$ adaptés à  la fois pour $f$ et $f^{-1}$). 

Voici maintenant le codage que l'on cherchait: 

Soit $w=(\cdots w_{-l},\cdots,w_0,w_1,\cdots,w_l,\cdots)\in \{1,\cdots,N\}^{\mathbb{Z}}$. 

$(\cdots,w_{-l},\cdots,w_0)$ définit un élément $A\in G_v$ (qui est égal à  $\sigma_v(w_0,w_{-1},\cdots,w_{-l},\cdots)$) et $(w_1,\cdots,w_l,\cdots)$ définit un $B\in G_h$ (qui est $\sigma_h(w_1,\cdots,w_l,\cdots)$). 

L'intersection de $A$ avec $B$ est composée d'un unique point (voir la démonstration de la proposition S.3.7 dans \cite{KH} car $\gamma_0^2 < 1$). 

On notera $S_0(w)$ l'image de ce point par $\tau_y C_{\gamma}(\widehat{y})$. On a obtenu ainsi une application $S_0: \{1,\cdots,N\}^{\mathbb{Z}}\longrightarrow X$. C'est le codage que l'on cherchait. Il s'agit maintenant de vérifier les propriétés du théorème \ref{theoreme2}.

\subsection{\bf{Propriétés du codage}}{\label{codage}}

Montrons tout d'abord que l'on a le diagramme suivant: 

$$
\xymatrix{
\{1,\cdots,N\}^\mathbb{Z}  \ar[r]^{\sigma}  \ar[d]_{S_0}  & \{1,\cdots,N\}^\mathbb{Z} \ar[d]^{S_0}\\
X \ar[r]^{f^n} & X
}
$$

avec $\sigma$ le décalage à  droite et $f^n\circ S_0=S_0\circ \sigma$. 

Soit $w=(\cdots,w_{-l},\cdots,w_{-1},w_0,w_1,\cdots)\in \{1,\cdots,N\}^{\mathbb{Z}}$. 

$S_0(w)=\tau_y C_{\gamma}(\widehat{y})(A\cap B)$ avec $A=\lim A_l^0$ et $B=\lim B_l^0$ où $A_l^0$ est un graphe de $\mathcal{A}_{w_0 w_{-1} \cdots w_{-l}}^0$ (pour tout $l\geq 0$) et $B_l^0$ est un graphe de $\mathcal{B}_{w_1\cdots w_l}^0$ (pour tout $l\geq 1$). 

Pour $l\geq 1$, on a par construction $C_{\gamma}^{-1}(\widehat{y})\tau_y^{-1}\tau_{f^n(x_{w_0})}C_{\gamma}(\widehat{f}^n(\widehat{x_{w_0}}))g_{\widehat{f}^{n-1}(\widehat{x_{w_0}})}\circ \cdots \circ g_{\widehat{x_{w_0}}}\circ C_{\gamma}^{-1}(\widehat{x_{w_0}})\tau_{x_{w_0}}^{-1}\tau_y C_{\gamma}(\widehat{y})(A_l^0)$ qui est inclus dans un graphe $A_l^1$ de $\mathcal{A}_{w_{-1}\cdots w_{-l}}^0$. 

$A_l^1$ converge donc vers $\sigma_v(w_{-1},\cdots,w_{-l},\cdots)$ quand $l\to+\infty$.

Maintenant les $g_{\widehat{f}^{p-1}(\widehat{x_{w_0}})}\circ \cdots \circ g_{\widehat{x_{w_0}}}\circ C_{\gamma}^{-1}(\widehat{x_{w_0}})\tau_{x_{w_0}}^{-1}\tau_y C_{\gamma}(\widehat{y})(A_l^0)$ sont dans des boîtes $\overline{B_{k_1}(0,e^{-\gamma/2}hr(\widehat{f}^p(\widehat{x_{w_0}})))}\times \overline{B_{k_2}(0,hr(\widehat{f}^p(\widehat{x_{w_0}})))}$ pour $p=1,\cdots,n-1$ grâce aux propriétés des graphes, $C_{\gamma}^{-1}(\widehat{x_{w_0}})\tau_{x_{w_0}}^{-1}\tau_y C_{\gamma}(\widehat{y})(A_l^0)$ est dans la boîte $\overline{B_{k_1}(0,e^{-\gamma/2}hr(\widehat{x_{w_0}}))}\times \overline{B_{k_2}(0,e^{\gamma}hr(\widehat{x_{w_0}}))}$ et pour $p=n$ le point précédent est dans $\overline{B_{k_1}(0,e^{\gamma/2}hr(\widehat{f}^n(\widehat{x_{w_0}})))}\times \overline{B_{k_2}(0,hr(\widehat{f}^n(\widehat{x_{w_0}})))}$ (toujours par les propriétés des graphes et $1+\gamma_0\leq e^{\frac{\gamma}{2}}$). 

En particulier, $g_{\widehat{f}^{p-1}(\widehat{x_{w_0}})}\circ \cdots\circ g_{\widehat{x_{w_0}}}C_{\gamma}^{-1}(\widehat{x_{w_0}})\circ \tau_{x_{w_0}}^{-1} \tau_y C_{\gamma}(\widehat{y})(A_l^0)$ vit dans un endroit où $g_{\widehat{f}^p(\widehat{x_{w_0}})}$ est holomorphe donc continue. On peut donc passer à  la limite quand $l \to+\infty$ et on obtient que 

\begin{equation*}
\begin{split}
C_{\gamma}^{-1}(\widehat{y})\tau_y^{-1}f^n(S_0(w)) & =C_{\gamma}^{-1}(\widehat{y})\tau_y^{-1}\tau_{f^n(x_{w_0})}C_{\gamma}(\widehat{f}^n(\widehat{x_{w_0}}))g_{\widehat{f}^{n-1}(\widehat{x_{w_0}})}\circ \cdots \\
&\circ g_{\widehat{x_{w_0}}}\circ C_{\gamma}^{-1}(\widehat{x_{w_0}})\tau_{x_{w_0}}^{-1}\tau_y C_{\gamma}(\widehat{y})(C_{\gamma}^{-1}(\widehat{y})\tau_{y}^{-1}(S_0(w)))\\
\end{split}
\end{equation*}

est dans $\sigma_v(w_{-1},\cdots,w_{-l},\cdots)$. 

Par ailleurs on obtient aussi que: 

$$g_{\widehat{f}^{p-1}(\widehat{x_{w_0}})}\circ \cdots\circ g_{\widehat{x_{w_0}}}\circ C_{\gamma}^{-1}(\widehat{x_{w_0}})\tau_{x_{w_0}}^{-1}\tau_y C_{\gamma}(\widehat{y})(C_{\gamma}^{-1}(\widehat{y})\tau_{y}^{-1}(S_0(w)))$$

\noindent se trouve dans $\overline{B_{k_1}(0,e^{-\gamma/2}hr(\widehat{f}^p(\widehat{x_{w_0}})))}\times \overline{B_{k_2}(0,hr(\widehat{f}^p(\widehat{x_{w_0}})))}$ pour $p=1,\cdots,n-1$, dans 
$\overline{B_{k_1}(0,e^{-\gamma/2}hr(\widehat{x_{w_0}}))}\times \overline{B_{k_2}(0,e^{\gamma}hr(\widehat{x_{w_0}}))}$ pour $p=0$ et dans $\overline{B_{k_1}(0,e^{\gamma/2}hr(\widehat{f}^n(\widehat{x_{w_0}})))}\times \overline{B_{k_2}(0,hr(\widehat{f}^n(\widehat{x_{w_0}})))}$ pour $p=n$.

Soit $(b_l)$ une suite de $B_l^0$ qui converge vers $C_{\gamma}^{-1}(\widehat{y})\tau_y^{-1}(S_0(w))$. Par continuité pour $l$ assez grand on a: $g_{\widehat{f}^{p-1}(\widehat{x_{w_0}})}\circ \cdots\circ g_{\widehat{x_{w_0}}}\circ C_{\gamma}^{-1}(\widehat{x_{w_0}})\tau_{x_{w_0}}^{-1}\tau_y C_{\gamma}(\widehat{y})(b_l)$ qui est au-dessus de $\overline{B_{k_1}(0,hr(\widehat{f}^p(\widehat{x_{w_0}})))}$ pour $p=0,\cdots,n-1$ et au-dessus de $\overline{B_{k_1}(0,e^{\gamma}hr(\widehat{f}^n(\widehat{x_{w_0}})))}$ pour $p=n$: ce sont donc des points qui sont dans la partie du cut-off que l'on garde quand on passe de $B_l^0$ à  son successeur via le processus qui passe de $g_{\widehat{x_{w_0}}},\cdots,g_{\widehat{f}^{n-1}(\widehat{x_{w_0}})}$. Le successeur de $B_l^0$ est par définition un graphe $B_l^1$ qui constitue $\mathcal{B}_{w_0 w_1\cdots w_l}^0$. 

On a $B_l^1\to\sigma_h(w_0,w_1,\cdots w_l,\cdots)$ et ainsi

\begin{equation*}
\begin{split}
C_{\gamma}^{-1}(\widehat{y}) \tau_y^{-1}f^n(S_0(w))& =C_{\gamma}^{-1}(\widehat{y})\tau_y^{-1}\tau_{f^n(x_{w_0})}C_{\gamma}(\widehat{f}^n(\widehat{x_{w_0}}))g_{\widehat{f}^{n-1}(\widehat{x_{w_0}})}\circ \cdots \\
&\circ g_{\widehat{x_{w_0}}}\circ C_{\gamma}^{-1}(\widehat{x_{w_0}})\tau_{x_{w_0}}^{-1}\tau_y C_{\gamma}(\widehat{y})(C_{\gamma}^{-1}(\widehat{y})\tau_{y}^{-1}(S_0(w))) \\
\end{split}
\end{equation*}

est dans $\sigma_h(w_{0},w_1, \cdots,w_{l},\cdots)$.

Si on fait le bilan on a que $C_{\gamma}^{-1}(\widehat{y}) \tau_y^{-1}f^n(S_0(w))$ est l'unique point de 

$$\sigma_v(w_{-1},\cdots,w_{-l}, \cdots)   \cap \sigma_h(w_{0},w_1, \cdots,w_{l},\cdots).$$ 

En particulier cela signifie que: 

$$f^n(S_0(w))=\tau_y C_{\gamma}(\widehat{y})(\sigma_v(w_{-1},\cdots)\cap \sigma_h(w_0,\cdots))=S_0(\sigma(w))$$

où $\sigma$ est le décalage à  droite. C'est bien ce que l'on voulait démontrer. 

Montrons maintenant que $S_0$ est continue. Sur $\{1,\cdots,N\}^{\mathbb{Z}}$ on met la métrique: 

$$\mathrm{d}(w,w')=\sum_{i=1}^{+\infty}\frac{| w_i-w_i'|}{2^i}+\sum_{i=-\infty}^{0}\frac{| w_i-w_i'|}{2^{-i}}
.$$

Soit $\beta>0$. Si $p$ est suffisamment grand, on a $8he^{-2\gamma n p + 2\gamma p} \leq \frac{\beta r_0}{C(X)}$. 

Maintenant, si $\mathrm{d}(w,w')<\frac{1}{2^{p+1}}$, on a forcément 

$$w_0=w_0',\cdots,w_{-p}=w_{-p}',w_1=w_1',\cdots,w_p=w_p'.$$

$\sigma_v(w_0, w_{-1}, \cdots, w_{-l},\cdots)$ est limite de $A_l^1$ avec $A_l^1$ graphe de $\mathcal{A}_{w_0 w_{-1} \cdots w_{-l}}^0$ pour $l\geq 0$. 

Pour $l\geq p$ on a $A_l^1 \subset \mathcal{A}_{w_0\cdots w_{-l}}^0 \subset \mathcal{A}_{w_0\cdots w_{-p}}^0$. En prenant la limite quand $l$ tend vers $+ \infty$, on obtient que $\sigma_v(w_0, w_{-1}, \cdots)\subset \overline{\mathcal{A}_{w_0 \cdots w_{-p}}^0}$. Le point $C_{\gamma}^{-1}(\widehat{y})\tau_y^{-1}(S_0(w))$ est donc dans $\overline{\mathcal{A}_{w_0\cdots w_{-p}}^0}$. De même $C_{\gamma}^{-1}(\widehat{y})\tau_y^{-1}(S_0(w'))$ est dans $\overline{\mathcal{A}_{w_0' \cdots w_{-p}'}^0}$ qui est égal à  $\overline{\mathcal{A}_{w_0\cdots w_{-p}}^0}$. 

Par un raisonnement similaire, $C_{\gamma}^{-1}(\widehat{y})\tau_y^{-1}(S_0(w))$ et $C_{\gamma}^{-1}(\widehat{y})\tau_y^{-1}(S_0(w'))$ sont dans $\overline{\mathcal{B}_{w_1\cdots w_p}^0}$. 

Comme $C_{\gamma}^{-1}(\widehat{y})\tau_y^{-1}(S_0(w))$ est dans $\overline{\mathcal{A}_{w_0\cdots w_{-p}}^0}$, il existe une suite de points de $\mathcal{A}_{w_0\cdots w_{-p}}^0$ qui converge vers lui. Par définition la suite de points s'écrit $(\phi_l^1(Y_l),Y_l)$ où $\| Y_l\|\leq e^{\frac{\gamma}{2}}hr(\widehat{y})$, $\mathrm{Lip}( \phi_l^1) \leq \gamma_0$ et $\| \phi_l^1(0)\|\leq e^{-\frac{\gamma}{2}}h r(\widehat{y})$. 

Par le théorème d'Ascoli, il existe une sous-suite $(\phi_{l_j}^1)$ qui converge uniformément vers $\phi^1$ qui est dans $G_v$. Par ailleurs, si $C_{\gamma}^{-1}(\widehat{y}) \tau_y^{-1}(S_0(w))=(X,Y)$, on a $Y_l\to Y$ et $\| \phi_{l_j}^1(Y)-\phi_{l_j}^1(Y_{l_j})\| \leq \gamma_0 \| Y-Y_{l_j}\| \to 0$ donc $\phi^1(Y)=X$ c'est-à-dire que $C_{\gamma}^{-1}(\widehat{y}) \tau_y^{-1}(S_0(w))$ appartient au graphe $(\phi^1(Y),Y)$ de $G_v$. 

De même, $C_{\gamma}^{-1}(\widehat{y}) \tau_y^{-1}(S_0(w'))$ appartient à  un graphe $(\phi^2(Y),Y)$ construit par le même procédé. 

Par le lemme \ref{lemme3.6}, on a $\mathrm{d}(\mathrm{graphe} \mbox{    }\phi_l^1,\mathrm{graphe} \mbox{    }\phi_l^2)\leq 2he^{-2\gamma n (p+1)+2\gamma (p+1)}$ et ainsi, en passant à la limite sur $l$, on a $\mathrm{d}(\mathrm{graphe} \mbox{    }\phi^1,\mathrm{graphe} \mbox{    }\phi^2)\leq 2he^{-2\gamma n (p+1)+2\gamma (p+1)}$. 

De la même façon, $C_{\gamma}^{-1}(\widehat{y})\tau_y^{-1}(S_0(w))$ appartient à  un graphe $(X,\psi^1(X))$ de $G_h$ (limite de graphes de $\mathcal{B}_{w_1\cdots w_p}^0$ et $C_{\gamma}^{-1}(\widehat{y}) \tau_y^{-1}(S_0(w'))$ est dans un graphe $(X,\psi^2(X)$ de $G_h$. On a aussi, en utilisant cette fois-ci le lemme \ref{lemme3.8} que  $\mathrm{d}(\mathrm{graphe} \mbox{    }\psi^1,\mathrm{graphe} \mbox{    }\psi^2)\leq 2he^{-2\gamma n p + 2\gamma p}$. 

Le point d'intersection $C_{\gamma}^{-1}(\widehat{y})\tau_y^{-1}(S_0(w))$ entre les graphes $(\phi^1(Y),Y)$ et $(X,\psi^1(X))$ est de la forme $(\phi^1(Y_1),Y_1)$ où $Y_1$ est l'unique point fixe de $\psi^1 \circ \phi^1 : \overline{B_{k_2}(0, e^{\gamma / 2}h r(\widehat{y}))} \longrightarrow \overline{B_{k_2}(0, e^{\gamma / 2}h r(\widehat{y}))}$ (voir la démonstration de la proposition S.3.7 dans \cite{KH} car on a $\gamma_0^2 < 1$). De même le point d'intersection $C_{\gamma}^{-1}(\widehat{y}) \tau_y^{-1}(S_0(w'))$ entre les graphes $(\phi^2(Y),Y)$ et $(X,\psi^2(X))$ est de la forme $(\phi^2(Y_2),Y_2)$ où $Y_2$ est l'unique point fixe de $\psi^2 \circ \phi^2$. On a alors

\begin{equation*}
\begin{split}
\|Y_1 - Y_2 \| &=\| \psi^1 \circ \phi^1 (Y_1) - \psi^2 \circ \phi^2 (Y_2) \| \\
& \leq \| \psi^1 \circ \phi^1 (Y_1) -  \psi^2 \circ \phi^1 (Y_1) \| + \|  \psi^2 \circ \phi^1 (Y_1) -  \psi^2 \circ \phi^2 (Y_2) \| \\
& \leq \mathrm{d}(\mathrm{graphe} \mbox{    } \psi^1,\mathrm{graphe} \mbox{    } \psi^2) + \gamma_0 \| \phi^1 (Y_1) - \phi^2 (Y_2) \| \\
& \leq \mathrm{d}(\mathrm{graphe} \mbox{    } \psi^1,\mathrm{graphe} \mbox{    } \psi^2) + \gamma_0 \left( \| \phi^1 (Y_1) - \phi^2(Y_1) \| + \| \phi^2(Y_1) - \phi^2 (Y_2) \| \right) \\
& \leq \mathrm{d}(\mathrm{graphe} \mbox{    } \psi^1,\mathrm{graphe} \mbox{    } \psi^2) + \gamma_0 \mathrm{d}(\mathrm{graphe} \mbox{    } \phi^1,\mathrm{graphe} \mbox{    } \phi^2) + \gamma_0^2 \| Y_1 - Y_2 \|.
\end{split}
\end{equation*}

D'où

$$ \|Y_1 - Y_2 \| \leq \frac{\mathrm{d}(\mathrm{graphe} \mbox{    } \psi^1,\mathrm{graphe} \mbox{    } \psi^2) + \gamma_0 \mathrm{d}(\mathrm{graphe} \mbox{    } \phi^1,\mathrm{graphe} \mbox{    } \phi^2)}{1 - \gamma_0^2}.$$

De là on en déduit que

\begin{equation*}
\begin{split}
 & \| C_{\gamma}^{-1}(\widehat{y})\tau_y^{-1}S_0(w) - C_{\gamma}^{-1}(\widehat{y})\tau_y^{-1}S_0(w') \|  \leq  \|Y_1 - Y_2 \| + \| \phi^1(Y_1)- \phi^2(Y_2) \| \\
& \leq \|Y_1 - Y_2 \| + \| \phi^1(Y_1)- \phi^2(Y_1) \| + \| \phi^2(Y_1) - \phi^2(Y_2) \| \\
& \leq (1+ \gamma_0) \|Y_1 - Y_2 \| +  \mathrm{d}(\mathrm{graphe} \mbox{    } \phi^1,\mathrm{graphe} \mbox{    } \phi^2) \\
&\leq \frac{1}{1 - \gamma_0}(\mathrm{d}(\mathrm{graphe} \mbox{    } \psi^1,\mathrm{graphe} \mbox{    } \psi^2) + \mathrm{d}(\mathrm{graphe} \mbox{    } \phi^1,\mathrm{graphe} \mbox{    } \phi^2)).
\end{split}
\end{equation*}

Autrement dit, 

$$ \| C_{\gamma}^{-1}(\widehat{y})\tau_y^{-1}S_0(w) - C_{\gamma}^{-1}(\widehat{y})\tau_y^{-1}S_0(w') \|  \leq 8he^{-2\gamma n p + 2\gamma p}  \leq \frac{\beta r_0}{C(X)}$$

et par le théorème des accroissements finis, comme $\| C_{\gamma}(\widehat{y}) \| \leq\frac{1}{r_0}$, on a finalement 

$$\mathrm{dist}(S_0(w),S_0(w'))\leq \frac{C(X)}{r_0}\times \frac{\beta r_0}{C(X)}=\beta.$$

On a donc montré que $S_0$ est uniformément continue. 

\begin{Rem}
Lorsque $f$ est inversible et que l'on a en plus $\int \log\mathrm{d}(x,I(f^{-1}))\mathrm{d}\mu(x) > -\infty$, $S_0$ est un homéomorphisme de $\{1,\cdots,N\}^{\mathbb{Z}}$ sur son image. 

En effet, tout d'abord $S_0$ est injective: si $w,w' \in \{1,\cdots,N\}^{\mathbb{Z}}$ avec $w\neq w'$ on a soit $(\cdots,w_{-l},\cdots,w_0)\neq (\cdots,w_{-l}',\cdots,w_0')$, soit $(w_1,\cdots,w_l,\cdots)\neq (w_1',\cdots,w_l',\cdots)$. 

Dans le premier cas $C_{\gamma}^{-1}(\widehat{y})\tau_y^{-1}S_0(w)$ et $C_{\gamma}^{-1}(\widehat{y})\tau_y^{-1}S_0(w')$ sont respectivement dans $\sigma_v(w_0,\cdots,w_{-l},\cdots)$ et $\sigma_v(w_0',\cdots,w_{-l}',\cdots)$ qui sont disjoints. On a donc $S_0(w)\neq S_0(w')$. 

Dans le second cas, comme $f$ est inversible et que $\int \log \mathrm{d}(x,I(f^{-1})) \mathrm{d}\mu (x)>-\infty$, on a aussi $\sigma_h(w_1,\cdots,w_l,\cdots)$ et $\sigma_h(w_1',\cdots,w_l',\cdots)$ qui sont disjoints et ainsi $S_0(w)\neq S_0(w')$.

Montrons, toujours sous l'hypothèse $f$ inversible et $\int \log \mathrm{d}(x,I(f^{-1})) \mathrm{d}\mu (x)>-\infty$, que $S_0^{-1}$ est continue sur $\Lambda_0=S_0(\{1,\cdots,N\})^{\mathbb{Z}}$. 

Soit $\beta>0$. Si $p$ est suffisamment grand on a $\frac{4N}{2^{p+1}}<\beta$. On a vu que les $\overline{\mathcal{A}_{w_0\cdots w_{-p}}}$ sont disjoints. Comme $f$ est inversible et $\int \log \mathrm{d}(x,I(f^{-1})) \mathrm{d}\mu (x)>-\infty$, il en est de même pour les $\overline{\mathcal{B}_{w_1\cdots w_p}}$. Soit donc $\alpha>0$ plus petit que la distance minimale entre deux $\overline{\mathcal{A}_{w_0\cdots w_{-p}}}$ et que la distance minimale entre deux $\overline{\mathcal{B}_{w_1 \cdots w_{p}}}$. 

Soit $z$ et $z'\in \Lambda_0$ avec $\mathrm{dist}(z,z') < \frac{\alpha r_0}{C(X)}$. 

Comme on l'a vu quand on montrait que $S_0$ est continue, $C_{\gamma}^{-1}(\widehat{y})\tau_y^{-1}(z)\in \overline{\mathcal{A}_{w_0\cdots w_{-p}}^0}\cap \overline{\mathcal{B}_{w_1\cdots w_{p}}^0}$ où $S_{0}^{-1}(z)=(\cdots,w_{-p},\cdots,w_0,\cdots,w_p,\cdots)$. De même, $C_{\gamma}^{-1}(\widehat{y})\tau_y^{-1}(z')\in \overline{\mathcal{A}_{w_0'\cdots w_{-p}'}^0}\cap \overline{\mathcal{B}_{w_1'\cdots w_{p}'}^0}$ où $S_{0}^{-1}(z')=(\cdots,w_{-p}',\cdots,w_0',\cdots,w_p',\cdots)$. 

Par le théorème des accroissement finis, 

$$\| C_{\gamma}^{-1}(\widehat{y})\tau_y^{-1}(z) - C_{\gamma}^{-1}(\widehat{y})\tau_y^{-1}(z') \|<\frac{C(X)}{r_0}\times \frac{\alpha r_0}{C(X)}=\alpha$$

car $\| C_{\gamma}^{-1}(\widehat{y})\|\leq \frac{1}{r_0}$.

On en déduit que $C_{\gamma}^{-1}(\widehat{y})\tau_y^{-1}(z)$ et $C_{\gamma}^{-1}(\widehat{y})\tau_y^{-1}(z')$ ne peuvent pas être dans deux $\overline{\mathcal{A}_{w_0\cdots w_{-p}}^0}$ ou dans deux $\overline{\mathcal{B}_{w_0\cdots w_{-p}}^0}$ différents. 

On a donc $w_0=w_0',\cdots,w_{-p}=w_{-p}'$ et $w_1=w_1',\cdots,w_{p}=w_{p}'$. 

En particulier 

\begin{equation*}
\begin{split}
\mathrm{dist}(S_0^{-1}(z),S_0^{-1}(z')) & = \sum_{i=1}^{+\infty}\frac{| w_i-w_i'|}{2^i}+\sum_{i=-\infty}^{0}\frac{| w_i-w_i'|}{2^{-i}} \\
& \leq \sum_{i=p+1}^{+\infty}\frac{N}{2^i}+\sum_{i=-\infty}^{-p-1}\frac{N}{2^{-i}} \\
& \leq \frac{N}{2^{p+1}}\times 2+\frac{N}{2^{p+1}}\times 2\leq \frac{4N}{2^{p+1}}<\beta.\\ 
\end{split}
\end{equation*}

L'application $S_0^{-1}$ est donc uniformément continue sur $\Lambda_0$. 
\end{Rem}

Revenons au cas général. 

Pour terminer la démonstration du premier point du théorème \ref{theoreme2} il reste à  étudier la mesure $\nu_0=(S_0)_*\lambda_0$ (où $\lambda_0$ est la mesure de Bernoulli sur $\{1,\cdots,N\}^\mathbb{Z}$). On a

$$(f^n)_* \nu_0=(f^n)_* (S_0)_*\lambda_0=(f^n\circ S_0)_*\lambda_0=(S_0\circ \sigma)_*\lambda_0=(S_0)_*\lambda_0=\nu_0.$$

C'est donc une mesure $(f^n)$-invariante. Calculons maintenant $h_{\nu_0}(f^n)$. 

Soit $\alpha>0$ la distance minimale entre deux $\overline{\mathcal{A}_i}$ pour $i=1,\cdots,N$. On a 

\begin{Lem}
Pour tout $z_1\in\overline{\mathcal{A}_{w_0\cdots w_{-l}}}$ et $z_2\in\overline{\mathcal{A}_{w_0'\cdots w_{-l}'}}$ avec $(w_0,\cdots,w_{-l})\neq (w_0',\cdots,w_{-l}')$, les points $z_1$ et $z_2$ sont $(l+1,\alpha)$-séparés pour $f^n$.
\end{Lem}

\begin{proof}

On considère deux points $z_1,z_2$ comme ci-dessus et soit $p$ le plus petit entier avec $w_{-p}\neq w_{-p}'$ ($0\leq p\leq l$).

On a vu au début du paragraphe \ref{codage} que 

$$C_{\gamma}^{-1}(\widehat{y})\tau_y^{-1}f^n \tau_y C_{\gamma}(\widehat{y})(\mathcal{A}_{w_0\cdots w_{-l}}^0)\subset \mathcal{A}_{w_{-1}\cdots w_{-l}}^0$$

c'est-à-dire $f^n(\mathcal{A}_{w_0\cdots w_{-l}})\subset \mathcal{A}_{w_{-1}\cdots w_{-l}}$. 

Ainsi $f^{np}(\mathcal{A}_{w_0\cdots w_{-l}})\subset \mathcal{A}_{w_{-p}\cdots w_{-l}}\subset \mathcal{A}_{w_{-p}}$ et de la même façon $f^{np}(\mathcal{A}_{w_0'\cdots w_{-l}'})\subset \mathcal{A}_{w_{-p}'\cdots w_{-l}'}\subset \mathcal{A}_{w_{-p}'}$. Par continuité $f^{np}(\overline{\mathcal{A}_{w_0\cdots w_{-l}}})\subset \overline{\mathcal{A}_{w_{-p}}}$ et $f^{np}(\overline{\mathcal{A}_{w_0'\cdots w_{-l}'}})\subset \overline{\mathcal{A}_{w_{-p}'}}$. 

On a donc $\mathrm{dist}(f^{np}(z_1),f^{np}(z_2))\geq \alpha$ et c'est ce que l'on voulait démontrer.

\end{proof}

Si $w\in \{1,\cdots,N\}^{\mathbb{Z}}$, on a vu que pour $l \geq 0$, on a $S_0(w)\in \overline{\mathcal{A}_{w_0\cdots w_{-l}}}$. 

Le support de $\nu_0$ est donc inclus dans 

$$\bigcup_{w_0,\cdots,w_{-l}=1,\cdots,N}\overline{\mathcal{A}_{w_0\cdots w_{-l}}}.$$. 

De plus,

\begin{equation*}
\begin{split}
\nu_0(\overline{\mathcal{A}_{w_0\cdots w_{-l}}}) & = \lambda_0(S_0^{-1}(\overline{\mathcal{A}_{w_0\cdots w_{-l}}})) \\
 & = \lambda_0(\{w',w_0'=w_0,\cdots,w_{-l}'=w_{-l}\}) \mbox{    } \mbox{  car les  } \overline{\mathcal{A}_{w_0\cdots w_{-l}}} \mbox{  sont disjoints} \\
 & = \left( \frac{1}{N} \right)^{l+1}. \\
\end{split}
\end{equation*}

Soit $z\in \bigcup_{w_0,\cdots,w_{-l}=1,\cdots,N}\overline{\mathcal{A}_{w_0\cdots w_{-l}}}$. On a $z\in \overline{\mathcal{A}_{w_0\cdots w_{-l}}}$ pour un certain $(w_0,\cdots,w_{-l})\in \{1,\cdots,N\}^{l+1}$ et donc, par le lemme précédent, si on note 

$$B_{l+1} ( f^n,\alpha,z)= \{ y \in X \mbox{  ,  } dist(f^{np}(z),f^{np}(y)) < \alpha \mbox{  pour  } p=0, \cdots, l \},$$

on a

\begin{equation*}
\begin{split}
\nu_0(B_{l+1}(f^n,\alpha,z)) & =\nu_0 \left( B_{l+1} ( f^n,\alpha,z)\cap \bigcup_{w_0,\cdots,w_{-l}=1,\cdots,N}\overline{\mathcal{A}_{w_0\cdots w_{-l}}} \right) \\
 & = \nu_0(B_{l+1}(f^n,\alpha,z)\cap \overline{\mathcal{A}_{w_0\cdots w_{-l}}})\leq \left( \frac{1}{N} \right)^{l+1}. \\
\end{split}
\end{equation*}

Par le théorème de Brin-Katok, on a: 

$$h_{\nu_0}(f^n)=\lim_{\alpha\to 0} \limsup_{l \rightarrow + \infty} -\frac{1}{l+1}\log \nu_0 (B_{l+1}(f^n,\alpha,z))$$

pour $\nu_0$ presque tout point (car comme $\lambda_0$ est ergodique, $\nu_0$ l'est aussi grâce à  la relation $f^n\circ S_0=S_0\circ \sigma$). 

Si $z\in \bigcap_{l\geq 0} \bigcup_{w_0,\cdots,w_{-l}=1,\cdots,N}\overline{\mathcal{A}_{w_0\cdots w_{-l}}}$ (qui est de masse $1$ pour $\nu_0$), on a pour tout $l\geq 0$: 

$$-\frac{1}{l+1}\log \nu_0 (B_{l+1}(f^n,\alpha,z))\geq -\log \frac{1}{N}=\log N,$$

ce qui implique $h_{\nu_0}(f^n)\geq \log N=h_{\mu}(f)n-\rho n$. C'est que l'on voulait démontrer. 

Il reste à  démontrer les points $2$ et $3$ du théorème \ref{theoreme2}.

Tout d'abord, $h_{top}(f_{| \Lambda_0})\geq h_{top}((f^n)_{| \Lambda_0})\times \frac{1}{n}$ car un ensemble $(l,\alpha)$-séparé pour $f^n$ est $(ln-n+1,\alpha)$-séparé pour $f$. 

Ensuite on a (par exemple en utilisant le théorème de Brin-Katok), $h_{top}((f^n)_{| \Lambda_0})\geq h_{\nu_0}(f^n)\geq h_{\mu}(f)n-\rho n$. On a donc $h_{top}(f_{| \Lambda_0})\geq h_{\mu}(f)-\rho $. C'est ce que l'on voulait.

Passons au dernier point. 

On note $\Lambda=\Lambda_0\cup f(\Lambda_0)\cup \cdots \cup f^{n-1}(\Lambda_0)$. 

Soit $z\in \Lambda_0$. On a $z=S_0(w)$ avec $w=(\cdots,w_{-l},\cdots,w_0,\cdots,w_l,\cdots) \in \{1,\cdots,N\}^{\mathbb{Z}}$. 

$z\in \overline{\mathcal{A}_{w_0}}$ et on a vu dans la démonstration du lemme \ref{lemme3.3} que $\mathrm{dist}(f^l(z),f^l(x_{w_0}))\leq \frac{\epsilon}{4}$ pour $l=0,\cdots,n$. Le point $x_{w_0}$ est dans le support de $\mu$ si on veut (on peut toujours intersecter $\pi(\Lambda_{\delta,m_0'})\cap \Gamma_{\epsilon,m_0}$ par le support de $\mu$ si on le souhaite). 

Cela signifie que $\Lambda$ est inclus dans un $\frac{\epsilon}{4}$-voisinage de support de $\mu$. Si on prend $\frac{\epsilon}{4}<\rho$, on a donc la première partie du point $3$. 

Si $z\in \Lambda_0$, on a $z=S_0(w)$ et on a vu aussi que $g_{\widehat{f}^l(\widehat{x_{w_0}})}$ est continue au voisinage de $g_{\widehat{f}^{l-1}(\widehat{x_{w_0}})}\circ \cdots \circ g_{\widehat{x_{w_0}}}\circ C_{\gamma}^{-1}(\widehat{x_{w_0}})\tau_{x_{w_0}}^{-1}(z)$ pour $l=0,\cdots,n-1$. Cela implique que $f$ est continue au voisinage de $f^{l}(S_0(w))$ pour $l=0,\cdots,n-1$. 

On peut donc définir: 

$$\nu=\frac{1}{n}(\nu_0+f_* \nu_0+\cdots+(f^{n-1})_* \nu_0)$$

et on a: 

$$f_*\nu=\frac{1}{n}(f_*\nu_0+(f^2)_* \nu_0+\cdots+ \underbrace{(f^{n})_* \nu_0}_{= \nu_0} )=\nu.$$

La mesure $\nu$ est donc bien invariante par $f$ et portée par $\Lambda$. Enfin, pour $i=1, \cdots , k_0$, on a
 
\begin{equation*}
\begin{split}
I& =\left| \int \varphi_i\mathrm{d}\nu-\int \varphi_i\mathrm{d}\mu \right| =\left| \int \varphi_i\frac{1}{n}\mathrm{d}(\nu_0+\cdots+(f^{n-1})_*\nu_0)-\int \varphi_i\mathrm{d}\mu \right| \\
&=\left| \frac{1}{n}\int \sum_{l=0}^{n-1}\varphi_i(f^l(x))\mathrm{d}\nu_0(x)-\int \varphi_i\mathrm{d}\mu \right| =\left| \frac{1}{n}\int \sum_{l=0}^{n-1}\varphi_i(f^l(S_0(w)))\mathrm{d}\lambda_0(w)-\int \varphi_i\mathrm{d}\mu \right| \\
& =\left| \int \left( \frac{1}{n}\sum_{l=0}^{n-1}\varphi_i(f^l(S_0(w)))-\int \varphi_i\mathrm{d}\mu \right) \mathrm{d}\lambda_0(w) \right|. \\
\end{split}
\end{equation*}

Mais pour $l=0, \cdots , n$, on a $\mathrm{dist}(f^l(S_0(w)),f^l(x_{w_0}))\leq \frac{\epsilon}{4}$. L'hypothèse faite au début implique donc que $| \varphi_i(f^l(S_0(w)))-\varphi_i(f^l(x_{w_0}))|<\frac{r}{2}$ pour $l=0,\cdots,n-1$. 

D'où:

$$I\leq \int \left| \frac{1}{n}\sum_{l=0}^{n-1}\varphi(f^l(x_{w_0}))-\int \varphi_i\mathrm{d}\mu \right| \mathrm{d}\lambda_0(w)+\frac{r}{2}.$$

Enfin par définition de $\Lambda_{\delta,m_0'}$, on a:

$$\left|\frac{1}{n}\sum_{l=0}^{n-1}\varphi(f^l(x_{w_0}))-\int \varphi_i\mathrm{d}\mu \right|<\frac{r}{2}$$

ce qui donne $I\leq r<\rho$ et cela termine la démonstration du théorème.

\newpage

\noindent Henry De Thélin, Université Paris 13, Sorbonne Paris Cité, LAGA, CNRS (UMR 7539), F-93430, Villetaneuse, France.  

\noindent Email: {\tt dethelin@math.univ-paris13.fr}

\bigskip

\noindent Franck Nguyen Van Sang, Université Paris 13, Sorbonne Paris Cité, LAGA, CNRS (UMR 7539), F-93430, Villetaneuse, France. 

\noindent Email: {\tt   Franck.Nguyen@ens-rennes.fr}


\newpage

\begin{thebibliography}{00}

\bibitem{Br} J.-Y. Briend, \textit{La propriété de Bernoulli pour les endomorphismes de $\Pp^k(\Cc)$}, Ergodic Theory Dynam. Systems, {\bf 22} (2002), 323-327.

\bibitem{Det1} H. De Thélin, \textit{Sur les exposants de Lyapounov des applications
  méromorphes}, Invent. Math., {\bf 172} (2008), 89-116.

\bibitem{Det3} H. De Thélin, \textit{Endomorphismes pseudo-aléatoires dans les espaces projectifs II}, J. Geom. Anal., {\bf 25} (2015), 204-225.

\bibitem{Det2} H. De Thélin, \textit{Un théorème de semi-continuité pour l'entropie des applications méromorphes}, Math. Ann., {\bf 362} (2015), 1-23.

\bibitem{DiDu} T.-C. Dinh et C. Dupont, \textit{Dimension de la mesure
d'équilibre d'applications méromorphes}, J. Geom. Anal., {\bf 14}
  (2004), 613-627.

\bibitem{Du1} C. Dupont, \textit{Bernoulli coding map and almost sure invariance principle for endomorphisms of $\Pp^k$}, Probab. Theory Related Fields, {\bf 146} (2010), 337-359.

\bibitem{Du} C. Dupont, \textit{Large entropy measures for endomorphisms of $\Cc \Pp^k$}, Israel J. Math., {\bf 192} (2012), 505-533.

\bibitem{FS1}  J.E. Forn{\ae}ss et N. Sibony, \textit{The closing lemma for holomorphic maps}, Ergodic Theory Dynam. Systems, {\bf 17} (1997), 821-837.

\bibitem{FS2} J.E. Forn{\ae}ss et N. Sibony, \textit{Closing lemma for holomorphic functions in $\Cc$}, Ergodic Theory Dynam. Systems, {\bf 18} (1998), 153-170.

\bibitem{FLQ} G. Froyland, S. Lloyd et A. Quas, \textit{Coherent structures and isolated spectrum for Perron-Frobenius cocycles}, Ergodic Theory Dynam. Systems, {\bf 30} (2010), 729-756.

\bibitem{Ge} K. Gelfert, \textit{Repellers for non-uniformly expanding maps with singular or critical points}, Bull. Braz. Math. Soc., {\bf 41} (2010), 237-257.

\bibitem{Ka} A. Katok, \textit{Lyapunov exponents, entropy and periodic orbits for diffeomorphisms}, Inst. Hautes Etudes Sci. Publ. Math., {\bf 51} (1980), 137-173.

\bibitem{KH} A. Katok et B. Hasselblatt, \textit{Introduction to the
  modern theory of dynamical systems}, Encycl. of Math. and its Appl.,
  vol. 54, Cambridge University Press, (1995).

\bibitem{KS} A. Katok et J.-M. Strelcyn, \textit{Invariant Manifolds, Entropy and Billiards; Smooth Maps with Singularities}, avec la collaboration de F. Ledrappier et F. Przytycki, Lect. Notes Math., {\bf 1222} (1986).

\bibitem{M1} R. Ma{\~n}é, \textit{Lyapounov exponents and stable manifolds for compact transformations}, Lecture Notes Math., {\bf 1007} (1983), 522-577.

\bibitem{N} S.E. Newhouse, \textit{Entropy and volume}, Ergodic Theory Dynam. Systems, {\bf 8} (1988), 283-299.

\bibitem{Pe} J.B. Pesin, \textit{Characteristic Ljapunov exponents, and smooth ergodic theory}, Uspehi Mat. Nauk, {\bf 32} (1977), 55-112.

\bibitem{Th} P. Thieullen, \textit{Fibrés dynamiques asymptotiquement compacts. Exposants de Lyapounov. Entropie. Dimension}, Ann. Inst. H. Poincaré Anal. Non Linéaire, {\bf 4} (1987), 49-97.

\end{thebibliography}
\end{document}